\documentclass[11pt, a4paper]{article}

\usepackage{a4}
\usepackage{latexsym}
\usepackage{amssymb}
\usepackage{amsfonts}
\usepackage{amsthm}
\usepackage{amsmath}
\usepackage[pdftex]{graphicx}
\usepackage{cite}
\usepackage{cases}

\newcommand{\bq}{\begin{equation}}
\newcommand{\eq}{\end{equation}}

\newtheorem{theorem}{Theorem}
\newtheorem{corollary}[theorem]{Corollary}
\newtheorem{proposition}[theorem]{Proposition}
\newtheorem{lemma}[theorem]{Lemma}
\newtheorem{definition}[theorem]{Definition}
\newtheorem{example}[theorem]{Example}
\newtheorem{remark}[theorem]{Remark}

\begin{document}
\title{Nash equilibirum and the Legendre transform in optimal stopping games with one dimensional diffusions}
\author{Jenny Sexton\footnote{School of Mathematics, University of Manchester. Oxford Road, {\sc Manchester, M13 9PL, United Kingdom.} E-mail: jennifer.sexton@postgrad.manchester.ac.uk}}
\date{\today}
\maketitle

\begin{abstract}
We show that the value function of an optimal stopping game driven by a one-dimensional diffusion can be characterised using the extension of the Legendre transform introduced in \cite{Peskir}. It is shown that under certain integrability conditions, a Nash equilibrium of the optimal stopping game can be derived from this extension of the Legendre transform. This result is an analytical complement to the results in \cite{Peskir} where the `duality' between a concave biconjugate which is modified to remain below an upper barrier and a convex biconjugate which is modified to remain above a lower barrier is proven by appealing to the probabilistic result in \cite{Peskir2}. The main contribution of this paper is to show that, for optimal stopping games driven by a one-dimensional diffusion, the semiharmonic characterisation of the value function may be proven using only results from convex analysis.
\end{abstract}

\begin{tabbing}
{\footnotesize Keywords: optimal stopping, convex analysis.}\\

{\footnotesize Mathematics Subject Classification (2000): 26A51, 60G40, 91A15, 91G80}

\end{tabbing}
 
\section{Introduction}
\setcounter{equation}{0}

This paper examines the connection between convex analysis and optimal stopping of one dimensional diffusions. The connection between the study of superharmonic functions and optimal stopping problems dates back to Dynkin \cite{Dynkin2} where the solution to an optimal stopping problem of the type 
\begin{equation*}
V\left( x\right) =\sup_{\tau }E_{x}\left[ G\left( X_{\tau }\right) \right]
\end{equation*}
(where $X$ is a Markov process) was first characterised as the smallest superharmonic majorant of the gains function $G$. When $X$ is a one-dimensional diffusion, the corresponding superharmonic functions can been characterised in terms of a generalised type of concavity (see \cite{DY} pp. 115). This so called `superharmonic characterisation' of the value function is discussed in further detail in \cite{PS} Chapter IV Section 9, to illustrate the properties of the solution to certain free-boundary problems associated with optimal stopping problems. More recently, the change of time and scale technique underpinning the superharmonic characterisation was `re-introduced' in \cite{DK}. Furthermore, the connection between the superharmonic characterisation and convex analysis has been highlighted in \cite{Peskir}.

The minimax version of this problem is referred to as an optimal stopping (or Dynkin) game. The sup-player selects a stopping time $\tau $ with the aim of maximising the functional $R_{x}\left( \tau ,\sigma \right) =E_{x}\left[ G\left( X_{\tau }\right) \mathbb{I}_{\left[ \tau \leq \sigma \right]}+H\left( X_{\sigma }\right) \mathbb{I}_{\left[ \tau >\sigma \right] }\right]$ while the inf-player selects a stopping time $\sigma $ with the aim of
minimising the same functional. Optimal stopping games are explained in more detail
in Section \ref{subsection diffusions} below. A variant of this optimal stopping
game was first studied by Dynkin \cite{Dynkin3} using martingale based methods. These
problems have also been approached via variational inequalities in \cite{BF2}
and \cite{BF3}. General conditions which ensure that such a saddle point exists have been studied in \cite{EP} and \cite{EV}. Optimal stopping games have been applied to solve problems in finance, see for example, \cite{BK}, \cite{Gapeev}, \cite{Kifer}, \cite{KK},
\cite{KKS} or \cite{Ky}.  

In \cite{Peskir2}, the superharmonic characterisation of the solution to optimal stopping problems has been extended to optimal stopping games where it is shown that the value function is a `semiharmonic' function. The value function is shown to be the smallest function which is superharmonic and dominates $G$ on the set $\{V< H\}$ as well as the largest function which is subharmonic and minorises $H$ on the set $\{V>G\}$. It is shown in \cite{Peskir2} that a single function with both these properties exists if and only if the optimal stopping game exhibits a saddle point. When the underlying Markov process is a one dimensional diffusion, these dual problems have been formulated in terms of an extension of the Legendre transform in \cite{Peskir}. 

The purpose of this paper is to establish the `semiharmonic characterisation' of the value function of optimal stopping games driven by one dimensional diffusions using purely analytical techniques. This requires us to modify some of the basic objects of convex analysis to account for the additional constraint on the domain of the smallest superharmonic majorant (resp. largest subharmonic minorant). The next section formally introduces optimal stopping games and sign-posts the contents of the rest of this paper. 

\section{Semiharmonic characterisation}

\label{subsection diffusions}

Consider a stochastic basis $\left( \Omega ,\mathcal{F},\mathbb{F},\left\{
P_{x}\right\} _{x\in I}\right) $ supporting a one dimensional diffusion $X=\left( X_{t}\right) _{t\geq 0}$ with $X_{0}=x\in I\subseteq \mathbb{R}$ under $P_{x}$. The level passage times of $X$ will be denoted by $T_{y}=\inf
\left\{ \left. t\geq 0\,\right\vert \,X_{t}=y\right\} $. The state space $I$
of $X$ is an open interval $I=\left( a,b\right) $ and
the boundaries of $I$ are natural, i.e. $P_{x}\left( T_{a}<\infty \right)
=P_{x}\left( T_{b}<\infty \right) =0$ for all $x\in I$. The diffusion $X$ is
also assumed to be regular in the sense that for all $x\in I$, $P_{x}\left( T_{y}<\infty
\right) >0$ for some $y\in(x,b)$.

For a given discount rate $r>0$, a Borel measurable function $U:I\rightarrow \mathbb{R}$ is $r$-superharmonic with respect to $X$ if
\begin{equation*}
U\left( x\right) \geq E_{x}\left[ e^{-r\tau} U\left( X_{\tau }\right) 
\mathbb{I}_{\left[ \tau <\infty \right] }\right]
\end{equation*}
for all $x\in I$ and all stopping times $\tau $. The function $U$ is $r$%
-subharmonic with respect to $X\,$if $U\left( x\right) \leq E_{x}\left[ e^{-r\tau }
U\left( X_{\tau }\right) \mathbb{I}_{\left[ \tau <\infty \right] }%
\right] $ for all $x\in I$ and all stopping times $\tau $. Moreover, $U$ is
referred to as $r$-harmonic, if it is both $r$-superharmonic and $r$%
-subharmonic.

The generator of $X$ is denoted $\mathbb{L}_{X}$ and under some additional
regularity conditions (see \cite{BS} Section 4.6) can be expressed as
\begin{equation*}
\mathbb{L}_{X}f\left( x\right) =\mu \left( x\right) \frac{d}{dx}f\left(
x\right) +D\left( x\right) \frac{d^{2}}{dx^{2}}f\left( x\right)
\end{equation*}
for $x\in I$ where $\mu $ is the drift and $D>0$ is the diffusion
coefficient of $X$. Take a constant $r>0$ and consider the ODE
\begin{equation}
\mathbb{L}_{X}f\left( x\right) =rf\left( x\right)  \label{ODE}
\end{equation}
for $x\in I$. This ODE has two linearly independent positive solutions, denoted $\varphi$
and $\psi$. As the boundaries of $I$ are natural, $\varphi $ and $\psi$
may be taken such that $\varphi$ is increasing and $\psi $ is decreasing
with
\begin{equation*}
\begin{array}{cc}
\varphi \left( a+\right) =\psi \left( b-\right) =0\qquad ,\qquad & \varphi
\left( b-\right) =\psi \left( a+\right) =+\infty .
\end{array}
\end{equation*}
The functions $\varphi$ and $\psi$ are shown in \cite{RW} V.50 to be
continuous, strictly monotone and strictly convex. Furthermore, for $y\in I$
, the process $( e^{-r\left( t\wedge T_{y}\right) }\varphi(
X_{t\wedge T_{y}})) _{t\geq 0}$ is a $( P_{x},\mathcal{F}_{t\wedge T_{y}})
$-martingale for all $y>x$ while $( e^{-r\left(
t\wedge T_{y}\right) }\psi (X_{t\wedge T_{y}}))_{t\geq 0}$
is a $( P_{x},\mathcal{F}_{t\wedge T_{y}})$-martingale for
all $y\leq x $. For $x,y\in I$ the Laplace transforms of the first passage
times may be expressed as
\begin{equation*}
E_{x}\left[ e^{-rT_{y}}\mathbb{I}_{[T_y < \infty]}\right] =\left\{
\begin{array}{cc}
\frac{\psi \left( x\right) }{\psi \left( y\right) }\smallskip & \mathrm{for}\,x>y, \\
\frac{\varphi \left( x\right) }{\varphi \left( y\right) } & \mathrm{for}\,x\leq y.
\end{array}
\right.
\end{equation*}

\begin{definition}
\label{def F-concave}Let $J:I\rightarrow \mathbb{R}$ be a monotone
function. A Borel measurable function $f:I\rightarrow \mathbb{R}$ is
$J$-concave if
\begin{equation*}
f\left( x\right) \geq f\left( y\right) \frac{J\left( z\right) -J\left(
x\right) }{J\left( z\right) -J\left( y\right) }+f\left( z\right) \frac{
J\left( x\right) -J\left( y\right) }{J\left( z\right) -J\left( y\right) }.
\end{equation*}
for $x\in \left[ y,z\right] \subseteq I$. The $J$-derivative of $f$ is
denoted
\begin{equation*}
\frac{d}{dJ}f\left( x\right) :=\lim_{y\downarrow x}\frac{f\left( y\right)
-f\left( x\right) }{J\left( y\right) -J\left( x\right) }.
\end{equation*}
\end{definition}

Define a pair of strictly increasing functions, $F:I \rightarrow (0,+\infty)$ and $\widetilde{F}:I \rightarrow (-\infty,0)$ using
\begin{equation}
F\left( x\right) =\frac{\varphi \left( x\right) }{\psi \left( x\right) }
\qquad ,\qquad  \widetilde{F}\left( x\right) =-\frac{\psi \left( x\right) }
{\varphi \left( x\right) }=-\frac{1}{F\left( x\right) }.
\label{F function}
\end{equation}
It is well-known (see \cite{Dyn} Theorem 16.4) that a Borel measurable
function $U$ is $r$-superharmonic if and only if $U/\psi $ is $F$-concave,
or equivalently, $U/\varphi $ is $\widetilde{F}$-concave. For $a<y\leq x\leq
z<b$ the Laplace transforms of the exit times from the open set $\left(
y,z\right) \subset I$ are
\begin{align}
E_{x}\left[ e^{-rT_{y}}\mathbb{I}_{\left[ T_{y}<T_{z}\right] }\right] &=
\frac{\psi \left( x\right) }{\psi \left( y\right) }\frac{F\left( z\right)
-F\left( x\right) }{F\left( z\right) -F\left( y\right) }=\frac{\varphi
\left( x\right) }{\varphi \left( y\right) }\frac{\widetilde{F}\left(
z\right) -\widetilde{F}\left( x\right) }{\widetilde{F}\left( z\right) -
\widetilde{F}\left( y\right) }, \notag \\
E_{x}\left[ e^{-rT_{z}}\mathbb{I}_{\left[ T_{y}>T_{z}\right] }\right] &=
\frac{\psi \left( x\right) }{\psi \left( z\right) }\frac{F\left( x\right)
-F\left( y\right) }{F\left( z\right) -F\left( y\right) }=\frac{\varphi
\left( x\right) }{\varphi \left( z\right) }\frac{\widetilde{F}\left(
x\right) -\widetilde{F}\left( y\right) }{\widetilde{F}\left( z\right) -
\widetilde{F}\left( y\right) }. \label{laplace}
\end{align}

Take a gains function $G: I \rightarrow \mathbb{R}$ which is upper semi-continuous
and such that for all $x \in I$,
\begin{equation}
P_x\left(\lim_{t\rightarrow\infty}e^{-rt}G\left(X_{t}\right)=c\right)=1\quad ,
\quad  E_{x}\left[ \sup_{t\geq 0}\left\vert e^{-rt} G\left( X_{t}\right)
\right\vert \right] <+\infty  \label{stopping assumptions}
\end{equation}
for some $c\in \mathbb{R}$. Section \ref{section stopping} examines the solution to the discounted optimal stopping problem
\begin{equation}
V\left( x\right) =\sup_{\tau }E_{x}\left[ e^{-r\tau} G\left( X_{\tau }\right) \right] \label{stopping problem - diffusion}
\end{equation}
for $x \in I$. Theorem 3.2 in \cite{Peskir} shows that when the value function $V$ defined in (\ref{stopping problem - diffusion}) coincides with the solution to the `dual problem'
\begin{equation}
\hat{V}=\inf_{U\in \mathrm{Sup}\left( G\right] }U  \label{dual stopping problem}
\end{equation}
where
\begin{equation*}
\mathrm{Sup}\left( G\right] :=\left\{ \left. U:I\rightarrow \left[ G,+\infty
\right) \,\right\vert \,U\text{ }\mathrm{is}\text{ }\mathrm{continuous}\text{
}\mathrm{and}\text{ }r\text{\textrm{-}}\mathrm{superharmonic}\right\}
\end{equation*}
the function $V$ can be expressed in terms of the concave biconjugate of $W:=(G/\psi)\circ F^{-1}$. The observation that when a finite optimal stopping time in (\ref{stopping problem - diffusion}) exists, the solutions to (\ref{stopping problem - diffusion}) and (\ref{dual stopping problem}) coincide is often referred to as the `superharmonic characterisation' of the value function (see Theorem 2.4 in \cite{PS} for a probabilistic proof). Theorem \ref{Theorem natural boundaries} is the converse of Theorem 3.2 in \cite{Peskir} in the sense that function $V$ is shown to be related to the concave biconjugate of $W$ without assuming that $V$ solves (\ref{dual stopping problem}). Furthermore, Theorem \ref{Theorem natural boundaries} provides a new proof that (\ref{stopping problem - diffusion}) and (\ref{dual stopping problem}) coincide.

The remainder of this section introduces optimal stopping games and the semiharmonic characterisation of the solution to such games as well as outlining the contents of the rest of this paper. Take a pair of gains functions $G$, $H$ that are continuous and for all $x\in I$ satisfy:
\begin{equation}
G\left(x\right) \leq H\left( x\right) \quad , \quad
E_{x}\left[ \sup_{t\geq 0}\left\vert e^{-rt} G\left( X_{t}\right)\right\vert
\right] <+\infty \quad , \quad E_{x}\left[ \sup_{t\geq 0}\left\vert e^{-rt}
H\left( X_{t}\right) \right\vert \right] <+\infty
\label{integ assumption}
\end{equation}
and
\begin{equation}
P_x\left(\lim_{t \rightarrow \infty} e^{-rt} G\left( X_{t} \right)  = \lim_{t
\rightarrow \infty} e^{-rt} H\left( X_{t} \right) \right) = 1.
\label{boundary assumption}
\end{equation}
Section \ref{section games} studies the solution to the infinite time horizon optimal stopping game with lower value $\underline{V}$ defined as
\begin{equation}
\underline{V}\left( x\right) =\sup_{\tau }\inf_{\sigma }E_{x}\left[ 
e^{-r\left( \tau \wedge \sigma \right) }\left(
G\left( X_{\tau }\right) \mathbb{I}_{\left[ \tau \leq \sigma \right]
}+H\left( X_{\sigma }\right) \mathbb{I}_{\left[ \tau >\sigma \right]
}\right) \right] ,
\label{lower value}
\end{equation}
and upper value $\overline{V}$ defined as
\begin{equation}
\overline{V}\left( x\right) =\inf_{\sigma }\sup_{\tau }E_{x}\left[ 
e^{-r\left( \tau \wedge \sigma \right) }\left(
G\left( X_{\tau }\right) \mathbb{I}_{\left[ \tau \leq \sigma \right]
}+H\left( X_{\sigma }\right) \mathbb{I}_{\left[ \tau >\sigma \right]
}\right) \right] .
\label{upper value}
\end{equation}
where $G$ and $H$ satisfy the assumptions outlined above. For ease of notation, the objective function is denoted
\begin{equation}
R_{x}\left( \tau ,\sigma \right) := E_{x}\left[ e^{-r\left( \tau
\wedge \sigma \right) }\left( G\left( X_{\tau
}\right) \mathbb{I}_{\left[ \tau \leq \sigma \right] }+H\left( X_{\sigma
}\right) \mathbb{I}_{\left[ \tau >\sigma \right] }\right) \right] . \label{game objective function}
\end{equation}
The optimal stopping game has a Stackelberg equilibrium if the upper and
lower values coincide, i.e. $\underline{V}(x) =\overline{V}(x) =:V(x)$ for all $x\in I$. In \cite{EP} and \cite{EV} it is shown via probabilistic means that this game exhibits a Stackelberg equilibrium when both (\ref{integ assumption}) and (\ref{boundary assumption})
hold. The assumptions in \cite{EV} are slightly more general, in which
case the Stackelberg equilibrium is determined by how the objective function
is specified at the natural boundaries.

A saddle point is a pair of stopping
times $\left( \tau ^{\ast },\sigma ^{\ast }\right) $ such that for any other
stopping times $\tau ,\sigma $
\begin{equation*}
R_{x}\left( \tau ,\sigma ^{\ast }\right) \leq R_{x}\left( \tau ^{\ast},
\sigma ^{\ast }\right) \leq R_{x}\left( \tau ^{\ast },\sigma \right)
\qquad \forall\, x \in I.
\end{equation*}
The optimal stopping game exhibits a Nash equilibrium if the game has saddle
point. In particular, existence of a Nash equilibrium implies a Stakelberg equilibrium
exists but the converse is not necessarily true. The result in \cite{EP} (which applies to more general processes) shows that, under the assumptions (\ref{integ assumption}) and
(\ref{boundary assumption}), the optimal stopping game described above has a Nash
equilibrium. 

Introduce a pair of dual problems
\begin{equation}
\hat{V}:=\inf_{U\in \mathrm{Sup}\left[ G,H\right) }U\qquad ,\qquad \check{V}
:=\sup_{U\in \mathrm{Sub}\left( G,H\right] }U  \label{dual problems}
\end{equation}
where the admissible sets of functions are
\begin{eqnarray*}
\mathrm{Sup}\left[ G,H\right) &=&\left\{ \left. U:I\rightarrow \left[ G,H%
\right] \,\right\vert \,U\text{ \textrm{is continuous and }}r\text{\textrm{-superharmonic on }}\left\{U<H\right\} \right\} , \\
\mathrm{Sub}\left( G,H\right] &=&\left\{ \left. U:I\rightarrow \left[ G,H\right] \,\right\vert \,U\text{ \textrm{is continuous and }}r\text{\textrm{-subharmonic on }}\left\{
U>G\right\} \right\} .
\end{eqnarray*}
Any function $U \in \mathrm{Sup}\left[G,H\right) \cup \mathrm{Sub}(G,H]$ is referred to as
a $r$-semiharmonic. It has been shown in \cite{Peskir2} Theorem 2.1 that when (\ref{integ assumption}) and (\ref{boundary assumption}) hold $\hat{V}=\check{V}$ and
that the value of the dual problems coincides with the value of the optimal
stopping game with upper and lower values (\ref{lower value})-(\ref{upper
value}) if and only if the optimal stopping game has a Nash equilibrium. The joint solution to the dual problems (\ref{dual problems}) is referred to in \cite{Peskir2} as the semiharmonic characterisation of the value function. 

When the driving process $X$ is a one dimensional diffusion absorbed upon exit from a compact set, it has been shown in \cite{Peskir} that the solutions to the dual problems (\ref{dual problems}) can be expressed in terms of an extension of the Legendre transform. It then follows from Theorem 2.1 in \cite{Peskir2} that the Nash and Stackelberg equilibrium of the optimal stopping game (\ref{lower value})-(\ref{upper value}) can be expressed in terms of this extension of the Legendre transform. 

The main contribution of Section \ref{section legendre transform} is to establish equality between the two extension of the Legendre transform introduced in \cite{Peskir} without reference to the probabilistic results in \cite{Peskir2} and \cite{EP}. In particular, we construct an extension of the concave biconjugate of $W^G:=(G/\psi)\circ F^{-1}$ of the form
\begin{equation*}
(W^G)_H^{\ast\ast}(y)= \inf_{z\in \mathbb{R}}\sup_{y\in \mathcal{A}^x_G(z)}(z(x-y)+W^G(y))
\end{equation*}
for $y>0$ where the set mapping $z\mapsto \mathcal{A}_G^x(z)$ is defined in such a way as to ensure that $(W^G)_H^{\ast\ast}(y)\leq (H/\psi)\circ F^{-1}(y)=: W_H(y)$ for all $y>0$. Similarly, we  construct an extension of the convex biconjugate of $W_H$ of the form
\begin{equation*}
(W_H)^G_{\ast\ast}(y)= \sup_{z\in \mathbb{R}}\inf_{y\in \mathcal{A}^H_x(z)}(z(x-y)+W_H(y))
\end{equation*}
for $y>0$ where the set mapping $z\mapsto \mathcal{A}^H_x(z)$ is defined in such a way as to ensure that $(W_H)^G_{\ast\ast}(y)\geq W^G(y)$ for all $y>0$. In particular, Theorem \ref{Theorem duality} is a converse to \cite{Peskir} Theorem 4.1 as it shows that $(W^G)_H^{\ast\ast}(y)=(W_H)^G_{\ast\ast}(y)$ for all $y>0$ using a purely analytical approach.  

Section \ref{section games} establishes that under the assumptions (\ref{integ assumption}) and (\ref{boundary assumption}) the infinite time horizon optimal stopping game (\ref{lower value})-(\ref{upper value}) has both a Stackelberg and Nash equilibrium. In Theorem \ref{Theorem Game has value} the optimal stopping game (\ref{lower value})-(\ref{upper value}) is shown to have a Stackelberg equilibrium which can be expressed in terms of the extensions of the Legendre transform introduced in Section \ref{section legendre transform}. It is then shown that the functions $(G)_H^{\ast\ast},(H)_{\ast,\ast}^{\ast\ast}: I \rightarrow [G,H]$
defined as 
\begin{equation*}
(G)_H^{\ast\ast}(x)=(W^G)_H^{\ast\ast}(F(x))\psi(x) \qquad , \qquad 
(H)^G_{\ast\ast}(x)=(W_H)^G_{\ast\ast}(F(x))\psi(x)
\end{equation*}
are $r$-semiharmonic and solve the dual problems (\ref{dual problems}). Consequently, it follows from the results in Section \ref{section legendre transform} that the solutions to the dual problems (\ref{dual problems}) coincide. Finally in Theorem \ref{Theorem Nash} it is shown that the $F$-concave/$F$-convex structure of $(W^G)_H^{\ast\ast}$ can be used to characterise a Nash equilibrium of the optimal stopping game (\ref{lower value})-(\ref{upper value}).

\section{Optimal stopping using the Legendre transformation}

\label{section stopping}

Before proceeding to solve the optimal stopping problem (\ref{stopping
problem - diffusion}), we shall first recall the definition and some
properties of the concave biconjugate. Let $f:\mathrm{dom}\left( f\right)
\rightarrow \mathbb{R}$ be a finite, measurable function on the domain
$\mathrm{dom}\left( f\right) \subseteq [-\infty,+\infty]$. The concave conjugate of
$f$, denoted $f_{\ast}$, is defined for $c\in \mathbb{R}$ as
\begin{equation*}
f_{\ast }\left( c\right) =\inf_{x\in \mathrm{dom}\left( f\right) }\left( cx
-f\left( x\right) \right) .
\end{equation*}
The concave biconjugate of $f$ is defined as
\begin{equation}
f_{\ast \ast }\left( x\right) =\inf_{y\in \mathrm{dom}\left( f^{\ast
}\right) }\left( xy -f_{\ast }\left( y\right) \right) =\inf_{y\in \mathrm{dom%
}\left( f^{\ast }\right) }\sup_{c\in \mathrm{dom}\left( f\right) }\left(
y\left( x-c\right) +f\left( c\right) \right) .  \label{concave biconjugate}
\end{equation}
The epigraph of a function $f$ is the set of all points above the graph of
$f$, that is
\begin{equation*}
\mathrm{epi}\left( f\right) :=\left\{ \left( x,\mu \right) \in \mathbb{R}
^{2}\,\left\vert \,\mu \geq f\left( x\right) \right. \right\} .
\end{equation*}
The convex hull of the set $\mathrm{epi}\left( f\right) $ is the
intersection of all convex sets containing $\mathrm{epi}\left( f\right)$
and is denoted $\mathrm{conv}\left( f\right) $. With a slight abuse of
notation, let 
\[
\mathrm{conv}\left( f\right) \left( x\right) :=\inf \left\{
\left. y\in \mathbb{R\,}\right\vert \,\left( x,y\right) \in \mathrm{conv}%
\left( f\right) \right\} ,
\]
then $-f_{\ast \ast }$ is the upper semi-continuous modification of $\mathrm{conv}\left( -f\right) \left(x\right) $, see \cite{Rock} Theorem 12.2 and Corollary 12.1.1. A constant $%
c\in \mathbb{R}$ is a subgradient of the function $f$ at $x\in \mathbb{R}$ if
\begin{equation*}
f\left( z\right) \geq f\left( x\right) +c \left( z-x\right) \qquad \forall
z\in \mathbb{R}.
\end{equation*}
The set of all such subgradients, denoted
\begin{equation}
\partial f\left( x\right) :=\left\{ c\in \mathbb{R}\,\left\vert \,f\left(
y\right) -f\left( x\right) \leq c(y-x) \quad\forall
y\in \mathbb{R}\right. \right\}  \label{superdiff}
\end{equation}
is referred to as the subdifferential. The function $f$ is referred to as
concave at $x\in \mathbb{R}$ when $\partial f\left( x\right) \neq \emptyset $. 
When $f$ is concave and differentiable at $x$ then $\partial f\left(
x\right) =\left\{ f^{\prime}\left( x\right) \right\}$.

The next lemma characterises the set upon which $f$ coincides with $f_{\ast
\ast }$. The proof illustrates how $x\mapsto f_{\ast \ast }\left( x\right) $
can be constructed using a spike variation.

\begin{lemma}
\label{lemma constructive} Any function $f:\mathrm{dom}(f)\rightarrow \mathbb{R}$ coincides with its concave biconjugate on
the set
\begin{equation}
\left\{ x\in \mathbb{R}\,\left\vert \,f_{\ast \ast }\left( x\right) =f\left(
x\right) \right. \right\} =\left\{ x\in \mathbb{R}\,\left\vert \,\partial
f\left( x\right) \neq \emptyset \right. \right\} .  \label{coincidence set}
\end{equation}
\end{lemma}

\begin{proof}
Suppose that $\partial f\left( x\right) \neq \emptyset $ and take $c\in
\partial f\left( x\right) $ then it follows from the definition (\ref{superdiff}) that
\begin{equation*}
f_{\ast }\left( c\right) :=\inf_{y\in \mathbb{R}}\left( cy -f\left( y\right)
\right) = cx -f\left( x\right) .
\end{equation*}
Consequently, $f_{\ast \ast }$ can be written as
\begin{equation*}
f_{\ast \ast }\left( x\right) =\inf_{c\in \mathbb{R}}\left(
c\left(x-x\right) +f\left( x\right) \right) =f\left( x\right) .
\end{equation*}
so $\left\{ x\in \mathbb{R}\,\left\vert \,\partial f\left( x\right) \neq
\emptyset \right. \right\} \subseteq \left\{ x\in \mathbb{R}\,\left\vert
\,f_{\ast \ast }\left( x\right) =f\left( x\right) \right. \right\} $. To
show this inclusion holds with equality assume that $\partial f\left(
x\right) =\emptyset $ and for a fixed $c\in \mathbb{R}$ let
\begin{equation*}
\varepsilon \left( x;c\right) :=\inf \left\{ \varepsilon \geq 0\,\left\vert
\,f\left( y\right) + c\left(x-y\right) \leq f\left( x\right) +\varepsilon
\quad \forall y\in \mathbb{R}\right. \right\} .
\end{equation*}
It follows that
\begin{equation*}
-f_{\ast }\left( c\right) =\sup_{y\in \mathbb{R}}\left( f\left( y\right) -cy
\right) =f\left( x\right) +\varepsilon \left( x;c\right) -cx ,
\end{equation*}
and $f_{\ast \ast }$ can be written as
\begin{equation}
f_{\ast \ast }\left( x\right) =\inf_{c\in \mathbb{R}}\left( cx -f_{\ast
}\left( c\right) \right) =\inf_{c\in \mathbb{R}} \left( f\left( x\right)
+\varepsilon \left( x;c\right) \right) .  \label{step 11}
\end{equation}
Let
\begin{equation}
\varepsilon^{\ast} \left( x\right) :=\inf \left\{ \varepsilon \geq
0\,\left\vert \,\exists\, c\in \mathbb{R} \quad \mathrm{s.t.} \quad f\left(
y\right) +c\left(x-y\right) \leq f\left( x\right) +\varepsilon \quad \forall
y\in \mathbb{R}\right. \right\} \label{pre epsilon star}
\end{equation}
so that $\left\{ x\in \mathbb{R}\,\left\vert \,\varepsilon ^{\ast }\left(
x\right) >0\right. \right\} =\left\{ x\in \mathbb{R}\,\left\vert \,\partial
f\left( x\right) =\emptyset \right. \right\} $. Hence we may conclude from
(\ref{step 11}) that when $\partial f\left( x\right) =\emptyset $, $f_{\ast
\ast }\left( x\right) =f\left( x\right) +\varepsilon^{\ast}\left( x\right)
>f\left( x\right) $.\medskip
\end{proof}

Let $W:(0,\infty) \rightarrow
\mathbb{R}$ be defined as
\begin{equation}
W\left( y\right) :=
\left( \frac{G}{\psi }\right) \circ \left( F\right) ^{-1}\left( y\right) \label{W_function}
\end{equation}
where $F$ was defined in (\ref{F function}).
Similarly, let $\widetilde{W}:(-\infty,0) \rightarrow \mathbb{R}$ be defined via
\begin{equation}
\widetilde{W}\left( y\right) := \left( \frac{G}{\varphi }\right) \circ ( \widetilde{F})
^{-1}\left( y\right) \label{W_tilde function}
\end{equation}
The next result is the main result in this section and it shows that the value
function $V$ defined in (\ref{stopping problem - diffusion}) is such that
$V/\psi$ coincides with $W_{\ast \ast }\circ F$ where $W_{\ast \ast }$ is 
the concave biconjugate of the function $W$. The value
function $V$ is also such that $V/\varphi $ coincides with $\widetilde{W}
_{\ast \ast }\circ \widetilde{F}$ where $\widetilde{W}_{\ast \ast }$ is the
concave biconjugate of the function $\widetilde{W}$. The case that both 
boundaries are absorbing has been handled by showing that $W_{\ast \ast }\left( F\left( x\right) \right)\psi \left( x\right) $ solves (\ref{dual stopping problem}) in \cite{Peskir} Theorem 3.2.

\begin{theorem}
\label{Theorem natural boundaries}Assume that $G:I \rightarrow \mathbb{R}$ is an
upper-semicontinuous function satisfying the assumption (\ref{stopping assumptions})
and such that $W\left(0+\right) =\widetilde{W}\left(0-\right) =0$. Consider the
stopping problem (\ref{stopping problem - diffusion}), then
\begin{equation*}
V\left( x\right) =W_{\ast \ast }\left( F\left( x\right) \right) \psi \left(
x\right) =\widetilde{W}_{\ast \ast}( \widetilde{F}(x)) \varphi \left( x\right) 
\end{equation*}
for all $x\in I$. The stopping time which attains the supremum in (\ref{stopping problem -
diffusion}) is $\tau ^{\ast }=T_{a^{\ast }}\wedge T_{b^{\ast }}$ where
\begin{eqnarray}
a^{\ast } &=&\sup \left\{ \left. c\in (a,x]\,\right\vert W\left( F\left(
c\right) \right) =W_{\ast \ast }\left( F\left( c\right) \right) \right\} ,
\label{location continuation region} \\
b^{\ast } &=&\inf \left\{ \left. c\in [x,b)\,\right\vert W\left( F\left(
c\right) \right) =W_{\ast \ast }\left( F\left( c\right) \right) \right\} ,
\notag
\end{eqnarray}
and we use the convention that $\sup \emptyset = a$ and $\inf \emptyset = b$.
Furthermore $V$ solves (\ref{dual stopping problem}).
\end{theorem}

\begin{proof}
It is assumed that $W(0+)=\widetilde{W}(0-)=0$ without loss of generality as when this is not true it can be achieved by subtracting the constant $c$ introduced in (\ref{stopping assumptions}) from the value function $V$. Let $\left( b_{n}\right) _{n\geq 1}$ be a sequence of real
numbers such $(a, b_{n}] \subset \left( a,b\right) $, $b_{n}\leq
b_{n+1}$ for all $n\geq 1$ and $(a,b_n] \rightarrow \left(
a,b\right) $ as $n\rightarrow +\infty $. Fix $x\in \left( a,b\right) $
then for some sufficiently large $N$, $x\in (a,b_n] $ for all $%
n\geq N$.
Let $X_{t}^{n}:=X_{t\wedge T_{b_{n}}}$ and consider the following family of
optimal stopping problems
\begin{equation}
V^{n}\left( x\right) =\sup_{\tau }E_{x}\left[ e^{-r\tau} G\left( X_{\tau }^{n}\right)
\right]  \label{stopping problem limit}
\end{equation}
where $G$ satisfies (\ref{stopping assumptions}). Consider the sets of functions 
\begin{equation*}
\mathrm{Sup}^n(G]=\{U:(a,b_n] \rightarrow [G,+\infty)\,\vert\, U\,\mathrm{is}\,\mathrm{continuous}\,\mathrm{and}\,r\mathrm{-superharmonic}\}
\end{equation*}
for each $n\geq 1$. Take any $U_1,U_2 \in \mathrm{Sup}^n(G]$, then $U_1/\psi$ and $U_2/\psi$ are $F$-concave on $(a,b_n]$ (see [8] Theorem 16.4). Since $U_1/\psi$ and $U_2/\psi$ are $F$-concave, the functions $W_1= (U_1/\psi) \circ F^{-1}$ and $W_2= (U_2/\psi) \circ F^{-1}$ are concave. Consequently, $W_3 = W_1 \wedge W_2$ is a concave function and reversing this argument implies $U_3 = U_1 \wedge U_2 \in \mathrm{Sup}^n(G]$. Hence the set $\mathrm{Sup}^n(G]$ is downwards directed so $\widetilde{V}^n:=\inf_{U\in\mathrm{Sup}^n(G]}U$ exists. In fact this infimum is attained since, when $X$ is a one dimensional diffusion, all $r$-superharmonic functions are continuous (see \cite{PS} Sections 9.3.3 and 9.3.4 replacing $S$ with $F$). This shows that the conditions of \cite{PS} Theorem 2.17 hold so if an optimal stopping time exists it is of the form $\tau_n^{\ast}=\inf\{t\geq 0\,\vert\, X_t \in \mathcal{D}^n\}$ where $\mathcal{D}^n:=\{x\in I\,\vert\,G(x)=\widetilde{V}^n(x)\}$. The set $\mathcal{D}$ is a closed set and $\mathcal{C}^n:=\{x\in I\,\vert\,G(x)<\widetilde{V}^n(x)\}$ is open since $G$ is upper-semicontinuous and $\widetilde{V}^n$ is continuous on $(a,b_n]$. Hence, for each $n\geq 1$ there is a collection of disjoint open intervals $((x_i,y_i))_{i\geq 1}$ such that $\mathcal{C}^n = \bigcup_{i\geq 1}(x_i,y_i)$. Suppose that $X_0 \in (x_i,y_i)$ for some $i\geq 1$, as the paths $t \mapsto X_t$ are continuous, $T_{x_i} \wedge T_{y_i}\leq T_{x_j} \wedge T_{y_j}$ for all $j\neq i$. Consequently we can restrict attention to candidate stopping times which are exit times from open intervals. 

If an optimal stopping time for (\ref{stopping problem limit}) exists we may write the optimal stopping time for (\ref{stopping problem limit}) as $\tau _{n}^{\ast }=T_{a^{\ast}_{n}}\wedge T_{b^{\ast }_{n}}$ for some $a\leq a^{\ast}_{n}\leq x\leq
b^{\ast}_{n}\leq b_n$. Furthermore, each $\tau _{n}^{\ast }\leq T_{b_n}$, $\tau
_{n}^{\ast }\leq \tau _{n+1}^{\ast }$ for all $n\geq 1$ and
$\lim_{n\rightarrow +\infty }\tau _{n}^{\ast }=\tau ^{\ast }$ $P_{x}$-a.s.
for all $x\in I$. For all $n\geq 1$, $X_{\tau _{n}^{\ast
}}^{n}=X_{\tau _{n}^{\ast }}$ and the process $X$ is left-continuous over
stopping times so $\lim_{n\rightarrow +\infty }X_{\tau _{n}^{\ast
}}^{n}=X_{\tau ^{\ast }}$ $P_{x}$-a.s for all $x\in I$. The assumption
(\ref{stopping assumptions}) implies that $\left(e^{-rt}G\left(X_{t}\right)\right) _{t\geq 0}$ is uniformly integrable so it
follows that
\begin{equation*}
\lim_{n\rightarrow \infty }V^{n}\left( x\right) =\lim_{n\rightarrow \infty
}E_{x}[ e^{-r\tau _{n}^{\ast }} G( X_{\tau _{n}^{\ast }}^{n}) ]
=V\left( x\right) .
\end{equation*}
The process $( e^{-r\left( t\wedge T_{b_{n}}\right) }\varphi(
X_{t}^{n}))_{t\geq 0}$ is a $( P_{x},\mathcal{F}_{t\wedge
T_{b_{n}}})$-martingale for each $b_{n}>x$ so applying the optional
sampling theorem yields $E_{x}\left[ e^{-r\tau} \varphi \left( X_{\tau}^{n}\right)
\right] =\varphi \left( x\right) $ for all $x\in (a,b_n]$ and
all $\tau \leq T_{b_{n}}$. Hence for an arbitrary $c\in \mathbb{R}$
\begin{equation*}
V^{n}\left( x\right) =\sup_{\tau }E_{x}\left[ e^{-r\tau } G\left( X_{\tau }^{n}\right)
\right] =c\varphi \left( x\right) +\sup_{\tau }E_{x}\left[ e^{-r\tau }\left(
G\left( X_{\tau }^{n}\right) -c\varphi \left( X_{\tau }^{n}\right) \right)
\right] .
\end{equation*}
Thus,
\begin{equation}
V^{n}\left( x\right) =\inf_{c\in \mathbb{R}}\left( c\varphi \left( x\right)
-\inf_{\tau }E_{x}\left[ e^{-r\tau }\left( c\varphi \left( X_{\tau }^{n}\right)
-G\left( X_{\tau }^{n}\right) \right) \right] \right) .
\label{Vn}
\end{equation}
The next step is to expand the right hand side of this expression in such a
way that it converges to the concave biconjugate of $W$ as $n\rightarrow
+\infty $. Using (\ref{laplace}), the inner infimum in (\ref{Vn}) can be written as
\begin{eqnarray*}
&&\inf_{a< y\leq x\leq z\leq b_n}E_{x}\left[ e^{-r\left( T_{y}\wedge T_{z}\right) }
( c\varphi ( X_{T_{y}\wedge T_{z}}^{n}) -G( X_{T_{y}\wedge T_{z}}^{n}))
\right]  \notag \\
&=&\inf_{a<y\leq x\leq z\leq b_n}\left( \frac{c\varphi \left(
y\right) -G\left( y\right) }{\psi \left( y\right) }\frac{F\left( z\right)
-F\left( x\right) }{F\left( z\right) -F\left( y\right) }+\frac{c\varphi
\left( z\right) -G\left( z\right) }{\psi \left( z\right) }\frac{F\left(
x\right) -F\left( y\right) }{F\left( z\right) -F\left( y\right) }\right)
\psi \left( x\right) .  
\end{eqnarray*}
We claim that the right hand side of this expression converges to
\begin{equation*}
\inf_{y\in \left( a,b\right) }\left( \frac{c\varphi \left( y\right)
-G\left( y\right) }{\psi \left( y\right) }\right) \psi \left( x\right)
=\inf_{y^{\prime }> 0}( cy^{\prime }-W\left( y^{\prime }\right))
\psi \left( x\right) =:W_{\ast }\left( c\right) \psi \left( x\right)
\end{equation*}
as $n\rightarrow \infty $. To this end, take
\begin{eqnarray*}
z^{+} &=&\sup \left\{ z\in I\,\,\left\vert \,W_{\ast }\left( c\right) \psi
\left( z\right) =c\varphi \left( z\right) -G\left( z\right) \right. \right\}
\\
y^{-} &=&\inf \left\{ y\in I\,\,\left\vert \,W_{\ast }\left( c\right) \psi
\left( y\right) =c\varphi \left( y\right) -G\left( y\right) \right. \right\}
\end{eqnarray*}
with the convention that $\sup \emptyset =a$ and $\inf \emptyset =b$. In
general for all $a\leq y\leq x\leq z\leq b_n$, $W_{\ast}(c)\leq \lambda(cy - W(y))+(1-\lambda)(cz-W(z))$ for all $\lambda \in [0,1]$. In particular, taking $\lambda = (F(z)-F(x))/(F(z)-F(y))$ we obtain
\begin{equation}
W_{\ast }\left( c\right) \leq \frac{c\varphi \left( y\right) -G\left(
y\right) }{\psi \left( y\right) }\frac{F\left( z\right) -F\left( x\right) }{%
F\left( z\right) -F\left( y\right) }+\frac{c\varphi \left( z\right) -G\left(
z\right) }{\psi \left( z\right) }\frac{F\left( x\right) -F\left( y\right) }{%
F\left( z\right) -F\left( y\right) }.  \label{step4}
\end{equation}
When $y^{-}\leq x\leq z^{+}$ (\ref{step4}) holds with equality for $y=y^{-}$
and $z=z^{+}$. When $y^{-}\geq x$, let $(y_m)_{m\geq 1}$ be such that $y_m\leq y^-$ and $\lim_{m\rightarrow \infty}y_m = a$. For $y=y_m$ and $z=y^-$ the inequality in (\ref{step4}) 
is strict, however, since we assumed $W(0+)=0$ it follows from the definition of $F$ that
the as $m\rightarrow \infty$ the right hand side of (\ref{step4}) converges to $W_{\ast}(c)$. 
When $z^{+}\leq x$ the inequality is strict as
\begin{equation}
\frac{c\varphi \left( y\right) -G\left( y\right) }{\psi \left( y\right) }
\frac{F\left( z^{+}\right) -F\left( x\right) }{F\left( z^{+}\right) -F\left(
y\right) }+\frac{c\varphi \left( z^{+}\right) -G\left( z^{+}\right) }{\psi
\left( z^{+}\right) }\frac{F\left( x\right) -F\left( y\right) }{F\left(
z^{+}\right) -F\left( y\right) }>\frac{c\varphi \left( z^{+}\right) -G\left(
z^{+}\right) }{\psi \left( z^{+}\right) } \label{here}
\end{equation}
for all $y\in \left[x,b_n \right] $. We may switch to the other ratio of the
fundamental solutions using (\ref{laplace}) on the left hand side of (\ref{here}) and use that $\varphi \left( y\right)\rightarrow +\infty $ as $y\uparrow b$ to conclude
\begin{equation*}
\lim_{n\rightarrow \infty }\inf_{y\in \left[x,b_{n}\right] }\left( \frac{
c\varphi \left( y\right) -G\left( y\right) }{\varphi \left( y\right) }\frac{
\widetilde{F}\left( z^{+}\right) -\widetilde{F}\left( x\right) }{\widetilde{F}\left( z^{+}\right)
-\widetilde{F}\left(y\right) }+\frac{c\varphi \left( z^{+}\right) -G\left( z^{+}\right) }{\varphi
\left( z^{+}\right) }\frac{\widetilde{F}\left( x\right) -\widetilde{F}\left( y\right) }
{\widetilde{F}\left(z^{+}\right) -\widetilde{F}\left( y\right) }\right) =W_{\ast }\left( c\right) .
\end{equation*}
Hence letting $n\rightarrow \infty$ on both sides of (\ref{Vn}) yields
\begin{equation*}
V\left( x\right) =\inf_{c\in \mathbb{R}}\left( c\frac{\varphi \left(
x\right) }{\psi \left( x\right) }+W_{\ast }\left( c\right) \right) \psi
\left( x\right) .
\end{equation*}
The argument using the other ratio of fundamental solutions follows
analogously. The statement about the optimal stopping time follows from the
form of the value function provided as the stopping region $\mathcal{D}$ and 
continuation region $\mathcal{C}$ for (\ref{stopping problem - diffusion}) are
\begin{eqnarray*}
\mathcal{C} &=&\left\{ \left. x\in I\,\right\vert \,V\left( x\right)
>G\left( x\right) \right\} =\left\{ x\in I\,\left\vert \,W_{\ast \ast
}\left( F\left( x\right) \right) \psi \left( x\right) >G\left( x\right)
\right. \right\} , \\
\mathcal{D} &=&\left\{ \left. x\in I\,\right\vert \,V\left( x\right)
=G\left( x\right) \right\} =\left\{ x\in I\,\left\vert \,W_{\ast \ast
}\left( F\left( x\right) \right) \psi \left( x\right) =G\left( x\right)
\right. \right\} .
\end{eqnarray*}
The stopping time $\tau^{\ast}:=\{t\geq 0\,\vert\,X_t \in \mathcal{D}\}=T_{a^{\ast}}\wedge T_{b^{\ast}}$ attains the supremum in (\ref{stopping problem - diffusion}) although it may occur that $P_x(\tau^{\ast}<+\infty)<1$. Finally, $W_{\ast \ast }=-\mathrm{cl}\left( \mathrm{conv}\left( -W\right)\right) $ is the smallest concave function dominating $W$, so $V\left(x\right) =W_{\ast \ast }\left( F\left( x\right) \right) \psi \left( x\right)
$ is the smallest $r$-superharmonic function dominating the gains function
$G\left( x\right)$ and hence $V\left( x\right) $ solves (\ref{dual stopping
problem}).
\end{proof}

\medskip

In Theorem \ref{Theorem natural boundaries} it has been shown that when $W(0+) =\widetilde{W}(0-)=0$ the value function of (\ref{stopping problem - diffusion}) coincides with the solution to (\ref{dual stopping problem}). Whether the optimal stopping time is obtained in finite time depends both on the function $G$ and the behaviour of the diffusion $X_t$ as $t\rightarrow \infty$. In Theorem \ref{Theorem natural boundaries} we can relax the assumptions (\ref{stopping assumptions}) and instead impose that $\vert W(0+)\vert < \infty$ and $\vert \widetilde{W}(0-)\vert < \infty$ which is the assumption used in \cite{DK}. The case that $W(0+)=+\infty$ is degenerate as $V(x)\geq \lim_{c\downarrow a}E_x[G(X_{T_c})e^{-rT_c}\mathbb{I}_{T_c <\infty} ] =\lim_{c\downarrow 0}W(c)\psi(x)=W(0+)\psi(x)=\infty$. 

\begin{remark}
\label{convention remark}
The assumption in Theorem \ref{Theorem natural boundaries} that $W(0+)=\widetilde{W}(0-)=0$ ensures that $V\left(a+\right)=V\left(b-\right)=0$. In \cite{DK} and \cite{EV} this condition is imposed by extending all Borel measurable functions $F$ on $I$ onto $[a,b]$ by setting $F\left(X_{\tau}\left(\omega\right)\right)=0$ on $\left\{\tau = +\infty \right\}$. If we use this convention and the weaker assumption that $\vert W(0+)\vert < \infty$ and $\vert \widetilde{W}(0-)\vert < \infty$ rather than (\ref{stopping assumptions}) the optimal stopping time may not exist. To see this, suppose that a sequence $(a_n)_{n\geq 1}$ such that $a_n \downarrow a$ as $n\rightarrow \infty$ is optimising in the sense that $E_x[G(X_{T_{a_n}})e^{-rT_{a_n}}]\leq E_x[G(X_{T_{a_{n+1}}})e^{-rT_{a_{n+1}}}]$ and $V(x)=\lim_{n\rightarrow \infty}E_x[G(X_{T_{a_n}})e^{-rT_{a_n}}]=G(a+)\psi(x)/\psi(a+)=W(0+)$. When $W(0+)>0$, this limit is not obtained as $\tau^{\ast}=\lim_{n\rightarrow \infty}T_{a_n}=\infty$ $P_x$-a.s. for all $x\in I$ so $E_x[G(X_{\tau^{\ast}})e^{-r\tau^{\ast}}]=0$.
\end{remark}

For fixed $x\in I$ let
\begin{eqnarray*}
l^{x}\left( c,p\right) &=&\sup \left\{ y\leq x\left\vert \,\left( \frac{G}{%
\psi }\right) \left( y\right) -\frac{p}{\psi \left( x\right) }=c\left(
F\left( y\right) -F\left( x\right) \right) \right. \right\} , \\
r^{x}\left( c,p\right) &=&\inf \left\{ y\geq x\,\left\vert \,\left( \frac{G}{%
\psi }\right) \left( y\right) -\frac{p}{\psi \left( x\right) }=c\left(
F\left( y\right) -F\left( x\right) \right) \right. \right\} ,
\end{eqnarray*}
with the usual convention that $\sup \emptyset =-\infty $ and $\inf
\emptyset =+\infty $. For a given $x\in I$ consider the sets
\begin{eqnarray}
A_{1}\left( x\right) &=&\left\{ p\in \left[ G\left( x\right) ,+\infty
\right) \,\left\vert \,\exists\,c\in \mathbb{R}\text{ }\mathrm{s.t.}\text{ }%
\left( l^{x}\left( c,p\right) ,r^{x}\left( c,p\right) \right) \subset
I\right. \right\} ,  \label{set A1} \\
A_{2}\left( x\right) &=&\left\{ p\in \left[ G\left( x\right) ,+\infty
\right) \,\left\vert \,\exists\,c\in \mathbb{R}\text{ }\mathrm{s.t.}\text{ }%
r^{x}\left( c,p\right) = -l^{x}\left( c,p\right) =\mathbb{+\infty }\right.
\right\} .  \notag
\end{eqnarray}
In Lemma \ref{lemma constructive} it was shown that $x\mapsto W_{\ast \ast
}\left( x\right) $ can be constructed by minimising over the `$F$-tangents'
of the form $y\mapsto p/\psi \left( x\right) +c\left( F\left( y\right)
-F\left( x\right) \right) $ which strictly dominate $W$ on $\mathrm{dom}%
\left( W\right) $, i.e. $W_{\ast \ast }\left( x\right) =\inf A_{2}\left(
x\right) $. The next corollary provides a short proof of the `dual
interpretation' provided in \cite{Peskir} which claims that $x\mapsto
W_{\ast \ast }\left( x\right) $ can be also constructed by maximising over
the $F$-tangents which intercept $W$ on both sides of the point $x$ in
$\mathrm{dom}\left( W\right) $, i.e. $W_{\ast \ast }\left( x\right)
=\sup A_{1}\left( x\right) $. The latter
formulation facilitates the construction of a maximising sequence of
stopping times.

\begin{corollary}
\label{corollary construction}Let
\begin{equation*}
\widehat{\varepsilon}\left( x;c\right) :=\sup \left\{ \varepsilon \geq
0\,\left\vert \, a \leq l^{x}\left(c,G\left(x\right)+\varepsilon\right) \leq
x\leq r^{x}\left(c,G\left(x\right)+\varepsilon\right) \leq b \right.\right\}
\end{equation*}
and set $A_{c}\left( x\right) :=\left[ F\left( l^{x}\left( c,G\left(
x\right) +\widehat{\varepsilon}\left( x;c\right) \right) \right) ,F\left(
r^{x}\left( c,G\left( x\right) +\widehat{\varepsilon}\left( x;c\right)
\right) \right) \right]$, then the value function of the optimal stopping
problem (\ref{stopping problem - diffusion}) can be represented as
\begin{equation*}
V\left( x\right) =\sup_{c\in \mathbb{R}}\sup_{y\in A_{c}\left( x\right)
}\left( c\left( F\left( x\right) -y\right) +W\left( y\right) \right) \psi
\left( x\right) .
\end{equation*}
Moreover, $V\left( x\right) =\sup A_{1}\left( x\right) $ for each $x\in I$
where $A_{1}\left( x\right) $ is the set defined in (\ref{set A1}).
\end{corollary}

\begin{figure}[h!]
\begin{center}
\includegraphics[width=9cm]{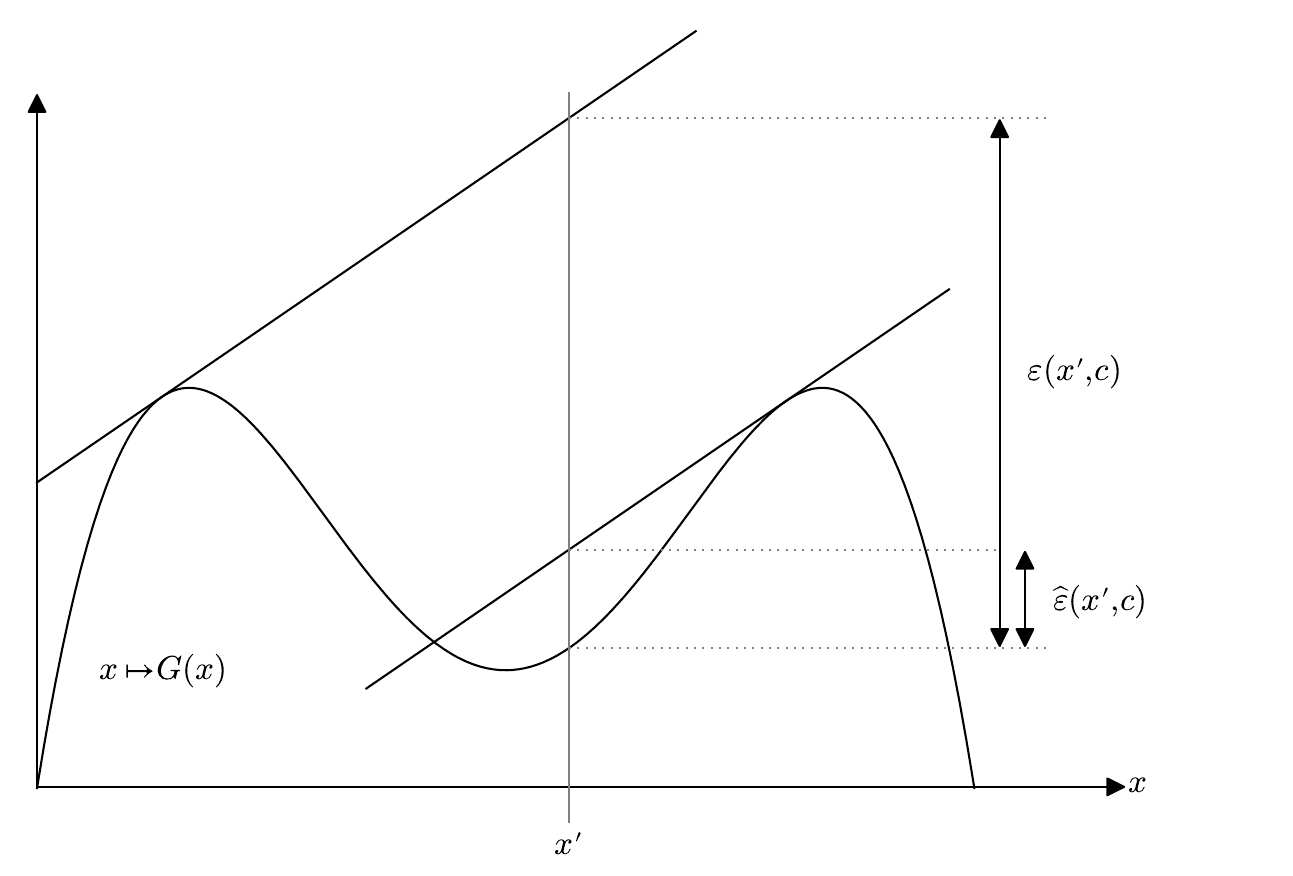}
\begin{minipage}{0.7\textwidth}
\caption[Definitions of $\widehat{\varepsilon}(x;c)$ and $\varepsilon(x;c)$]{This diagram illustrates $\widehat{\varepsilon}\left( x;c\right)$ and $\varepsilon\left( x;c\right)$ for a fixed $c$ in the case that $r=0$ and $X$ is a standard Brownian motion (i.e. $\psi \equiv 1$ and $\varphi(x)=x$ hence $F(x)=x$) 
\label{figure espilon star c} }
\end{minipage}
\end{center}

\end{figure} 

\begin{proof}
For the first statement, note that for each $c\in \mathbb{R}$, by definition
\begin{equation*}
\frac{G\left( x\right) +\widehat{\varepsilon }\left( x;c\right) }{\psi
\left( x\right) }+c\left( y-F\left( x\right) \right) \geq W\left( y\right)
\end{equation*}
for $y\in A_{c}\left( x\right) $ and equality holds for $y=F\left(
l^{x}\left( c,G\left( x\right) +\widehat{\varepsilon }\left( x;c\right)
\right) \right) $. Hence,
\begin{equation*}
\sup_{y\in A_{c}\left( x\right) }\left( W\left( y\right) -cy\right)
=W\left( F\left( x\right) \right) +\frac{\widehat{\varepsilon }\left(
x;c\right) }{\psi \left( x\right) }-cF\left( x\right)
\end{equation*}
and it follows that
\begin{equation*}
\sup_{c\in \mathbb{R}}\sup_{y\in A_{c}\left( x\right) }\left( c\left(
F\left( x\right) -y\right) +W\left( y\right) \right) \psi \left( x\right)
=\sup_{c\in \mathbb{R}}\left( G\left( x\right) +\widehat{\varepsilon}
\left( x;c\right) \right)
\end{equation*}
Let
\begin{equation*}
\varepsilon \left( x;c\right) :=\inf \left\{ \varepsilon \geq 0\,\left\vert
\,\left(\frac{G}{\psi}\right)(y) + c\left(F(x)-y\right) \leq \frac{G(x) +\varepsilon}{\psi(x)}
\quad \forall y\in \mathbb{R}\right. \right\} .
\end{equation*}
For each $c\in \mathbb{R}$, $\varepsilon \left( x;c\right) \geq \widehat{%
\varepsilon }\left( x;c\right) $ as illustrated in Figure \ref{figure espilon star c}
but to avoid contradicting their definitions $\inf_{c\in \mathbb{R}}\varepsilon
\left(x;c\right)=\sup_{c\in \mathbb{R}}\widehat{\varepsilon }\left( x;c\right)$
and hence the result follows from Lemma \ref{lemma constructive}.
Furthermore, it follows from the definition of the set $A_{1}\left( x\right)
$ that $\sup A_{1}\left( x\right) =G\left( x\right) +\sup_{c\in \mathbb{R}%
}\widehat{\varepsilon}\left( x;c\right)=W_{\ast\ast}(F(x))\psi(x)=V(x)$.\medskip
\end{proof}

The next corollary shows that Theorem \ref{Theorem natural
boundaries} is capable of handling the case when one (or more) of the
boundaries is absorbing.

\begin{corollary}
\label{corollary left absorbing}Take $x\in ( \alpha ,\beta )
\subset (a,b)$ and consider the stopped diffusion $%
X_{t}^{\alpha ,\beta }:=X_{t\wedge T_{\alpha ,\beta }}$ where $T_{\alpha
,\beta }=T_{\alpha }\wedge T_{\beta }\,$ and the corresponding optimal
stopping problem
\begin{equation}
V_{\alpha ,\beta }\left( x\right) =\sup_{\tau }E_{x}\left[ e^{-r\tau }
G(X_{\tau}^{\alpha ,\beta }) \right]
\label{stopping problem - absorbed left}
\end{equation}
for $G:\left( a,b\right) \rightarrow \mathbb{R}$ satisfying the assumptions
in Theorem \ref{Theorem natural boundaries}. Then
\begin{equation*}
V_{\alpha ,\beta }\left( x\right) =( W_{\left[ \alpha ,\beta \right]
})_{\ast \ast}\left( F\left( x\right) \right) \psi \left( x\right)
=(\widetilde{W}_{\left[ \alpha ,\beta \right] })_{\ast \ast}
(\widetilde{F}(x)) \varphi(x)
\end{equation*}
where
\begin{eqnarray*}
W_{\left[ \alpha ,\beta \right] }\left( y\right) &:=& W(y)
-\delta( y\,\vert\,F^{-1}\left( y\right) \in I\setminus \left[
\alpha ,\beta \right]) , \\
\widetilde{W}_{\left[ \alpha ,\beta \right] }\left( y\right) &:=& \widetilde{W}(y)
-\delta( y\,\vert\,\widetilde{F}^{-1}\left( y\right) \in
I\setminus \left[ \alpha ,\beta \right]) .
\end{eqnarray*}
Moreover, the stopping time which attains the supremum in (\ref{stopping
problem - diffusion}) is $\tau =T_{a^{\ast }}\wedge T_{b^{\ast }}$ where
\begin{align*}
a^{\ast} &=\sup \{ y\in \left[ \alpha ,x\right] \,\vert \,
W_{\left[ \alpha ,\beta \right] }\left( F\left( y\right) \right) =
(W_{\left[ \alpha ,\beta \right] })_{\ast \ast}\left( F\left( y\right)
\right) \} , \\
b^{\ast} &=\inf \{ y\in \left[ x,\beta \right] \,\vert\,
W_{\left[ \alpha ,\beta \right] }\left( F\left( y\right) \right) =
(W_{\left[ \alpha ,\beta \right] })_{\ast \ast }\left(F\left( y\right) \right)\} .
\end{align*}
\end{corollary}

\begin{proof}
An optimal stopping time for a problem of (\ref{stopping problem - absorbed
left}) exists (see \cite{PS} Chapter 1 Theorem 2.7) and is the exit time
from a open set containing $x$ so we may write $\tau^{\ast}=T_{a^{\ast
}}\wedge T_{b^{\ast }}$ for some $\alpha \leq a^{\ast }\leq x\leq b^{\ast
}\leq \beta $. Suppose that $\beta <b$ then as shown in the proof of
Theorem \ref{Theorem natural boundaries}, the value function of the optimal
stopping problem (\ref{stopping problem - absorbed left}) can be written as
\begin{equation}
V_{\alpha,\beta}\left( x\right) =\inf_{c\in \mathbb{R}}\left( c\varphi\left(
x\right) -R_{c}\left( x\right) \right) .  \label{step 7}
\end{equation}
where
\begin{equation*}
R_{c}\left( x\right) :=\inf_{a\leq y\leq x\leq z\leq b}\left(
\frac{c\varphi \left( y\right) -\widetilde{G}\left( y\right) }{\psi \left(
y\right) }\frac{F\left( z\right) -F\left( x\right) }{F\left( z\right)
-F\left( y\right) }+\frac{c\varphi \left( z\right) -\widetilde{G}\left(
z\right) }{\psi \left( z\right) }\frac{F\left( x\right) -F\left( y\right) }
{F\left( z\right) -F\left( y\right) }\right) \psi (x)
\end{equation*}
and
\begin{equation*}
\widetilde{G}\left( y\right) :=G\left( \alpha\right) \mathbb{I}_{\left( a,\alpha %
\right] }+G\left( y\right) \mathbb{I}_{\left( \alpha ,b\right) }+G\left(
\beta\right) \mathbb{I}_{\left[ \beta ,b\right) }.
\end{equation*}
As $\varphi $ is $r$-harmonic on $\left[y,b\right) $ for all $y>a$, it follows that for any $x\in I$
\begin{eqnarray*}
\frac{c\varphi \left( y\right) -\widetilde{G}\left( y\right) }{\psi \left(
y\right) }\psi \left( x\right) &=&E_{x}[ e^{-rT_{y}} ( c\varphi \left(
X_{T_{y}}\right) -\widetilde{G}\left( X_{T_{y}}\right)) ] \\
&=&c\varphi \left( x\right) -E_{x}[ e^{-rT_{y}} \widetilde{G}\left(X_{T_{y}}\right) ] \\
&>&c\varphi \left( x\right) -G\left( \alpha \right) E_{x}\left[
e^{-rT_{\alpha }}\right] =\frac{c\varphi \left( \alpha \right) -G\left(
\alpha \right) }{\psi \left( \alpha \right) }\psi \left( x\right)
\end{eqnarray*}
for all $y\in \left( a,\alpha \right)$. Likewise for $y\in \left( \beta ,b\right) $,
\begin{equation*}
\frac{c\varphi \left( y\right) -\widetilde{G}\left( y\right) }{\psi \left(
y\right) }\psi \left( x\right) >c\varphi \left( x\right) -G\left( \beta
\right) E_{x}\left[ e^{-rT_{\beta }}\right] =\frac{c\varphi \left( \alpha
\right) -G\left( \beta \right) }{\psi \left( \beta \right) }\psi \left(
x\right) .
\end{equation*}
Hence, an argument similar to that presented in Theorem \ref{Theorem natural
boundaries} can be used to show that
\begin{equation*}
R_{c}\left( x\right) =\inf_{y\in \left( a,b\right) }\left( \frac{c\varphi
\left( y\right) -\widetilde{G}\left( y\right) }{\psi \left( y\right) }
\right)\psi \left( x\right) =\inf_{y\in \left[ \alpha ,\beta \right]
}\left( \frac{c\varphi \left( y\right) -G\left( y\right) }{\psi \left(
y\right) }\right) \psi \left( x\right) =:W_{\ast }\left( c\right) .
\end{equation*} 
Thus (\ref{step 7}) reads
\begin{equation*}
V_{\alpha,\beta}\left( x\right) =\inf_{c\in \mathbb{R}}\left( c\varphi \left(
x\right) -\inf_{y\in \left[ \alpha ,\beta \right] }\left( \frac{c\varphi
\left( y\right) -G\left( y\right) }{\psi \left( y\right) }\psi \left(
x\right) \right) \right) =\left( W_{\left[ \alpha ,\beta \right] }\right)
_{\ast \ast }\left( F\left( x\right) \right) \psi \left( x\right) .
\end{equation*}
The case $\beta=b$ can be handled using an approximating sequence of domains as
in Theorem \ref{Theorem natural boundaries}. The argument using the other ratio of fundamental
solutions follows analogously.
\end{proof}

\medskip

This section concludes with a basic example, the perpetual American put option.

\begin{example}
\label{Example Legendre put}
Take $\sigma > 0$, $r>0$ and define $X$ to be the solution to
$dX_{t}=r X_{t}\,dt+\sigma X_{t}\,dW_{t}$ with $X_{0}=x> 0$ where $W$ is a
standard Brownian motion. Consider the perpetual American put option, the gains function of which is
$G(x)=(K-x)^{+}$ for a strike price $K>0$. The risk neutral value of this option is
\begin{equation}
V(x):=\sup_{\tau }E_{x}\left[ e^{-r\tau } \left(K-X_{\tau
}\right)^{+} \right] .  \label{problem - put}
\end{equation}
The infinitesimal generator associated with the geometric
Brownian motion $X$ is
\begin{equation*}
\mathbb{L}_{X}u:=r x\frac{du}{dx}+\frac{1}{2}\sigma ^{2}x^{2}\frac{d^{2}u}{%
dx^{2}}
\end{equation*}
and the ODE $\mathbb{L}_{X}u=ru$ has two fundamental solutions $\varphi \left( x\right) =x$
and $\psi \left(x\right) =x^{-2r/\sigma ^{2}}$. Let $-\widetilde{F}\left( x\right) =\left(
\psi /\varphi \right) \left( x\right) =x^{-1/\alpha}$ where $\alpha
:=1/\left( 1+2r/\sigma ^{2}\right) \in \left[ 0,1\right] $. The rescaled gains function
associated with the problem (\ref{problem - put}) is
\begin{equation*}
\widetilde{G}\left(y\right) :=
(Ky^{\alpha }-1)^+ .
\end{equation*}
The function $y\mapsto \widetilde{G}\left( y\right)$ has $\widetilde{G}^{\prime \prime }\left( y\right) \leq 0$ for all $y\in \mathbb{R}_+$ 
and hence to find $y\mapsto \widetilde{G}_{\ast \ast}\left(y\right)$ we need 
only to find $y\geq 0$ such that
\begin{equation}
\widetilde{G}\left( y\right) -\widetilde{G}^{\prime }\left( y\right)y=0.  \label{ODE put}
\end{equation}
The unique solution to (\ref{ODE put}) is $y^{\ast}=1/\left( K\left(
1-\alpha \right) \right) ^{1/\alpha }$ and
\begin{equation*}
\widetilde{G}_{\ast \ast }\left( y\right) =\left\{
\begin{array}{cc}
\widetilde{G}^{\prime}\left( y^{\ast}\right) y & \mathrm{for}\;y\in
\left[ 0,y^{\ast}\right],  \\
\widetilde{G}\left( y\right)  & \mathrm{for}\;y>y^{\ast}.
\end{array}
\right. 
\end{equation*}
Applying Theorem \ref{Theorem natural boundaries},
\begin{equation}
V\left( x\right) =\widetilde{G}_{\ast \ast }( -\widetilde{F}
\left( x\right) ) \varphi \left( x\right) =\left\{
\begin{array}{cc}
\frac{\sigma^{2}}{2r}\left( \frac{K}{\left( 1+\sigma ^{2}/2r\right) }
\right)^{1+2r/\sigma ^{2}}x^{-2r/\sigma^{2}} & \mathrm{for}\;x\geq x^{\ast} \\
K-x & \mathrm{for}\text{ }x\in \left[ 0,x^{\ast}\right]
\label{value function put}
\end{array}
\right.
\end{equation}
where $x^{\ast}=(y^{\ast})^{\alpha}=K/\left( 1+\sigma ^{2}/2r\right)$.
Figure \ref{figure put} illustrates the original and transformed payoff functions,
the concave biconjugate of the transformed payoff and the corresponding value function.

\begin{figure}[h!]
\begin{center}

\begin{tabular}{cc}
\includegraphics[width=7cm]{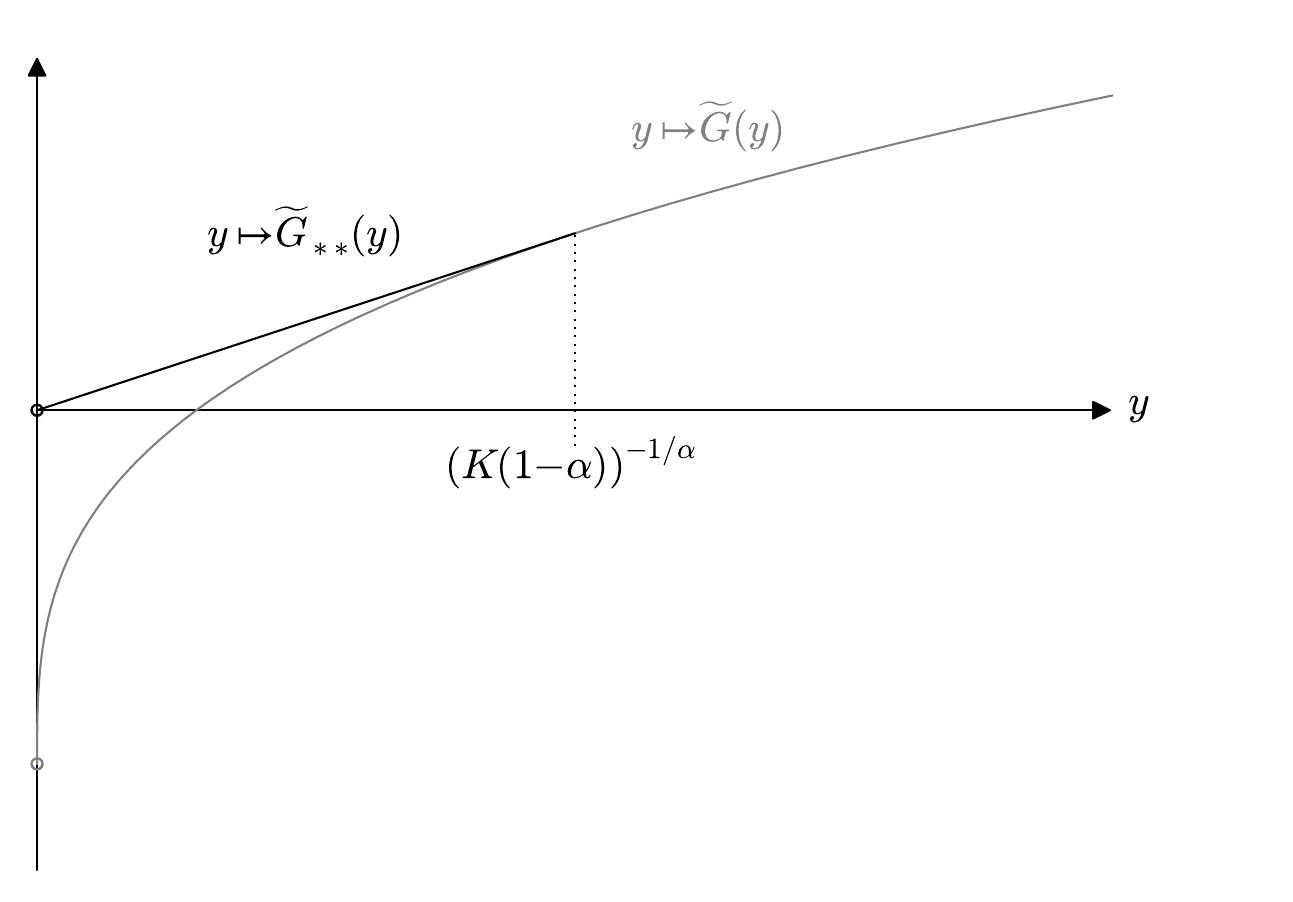} & \includegraphics[width=7cm]{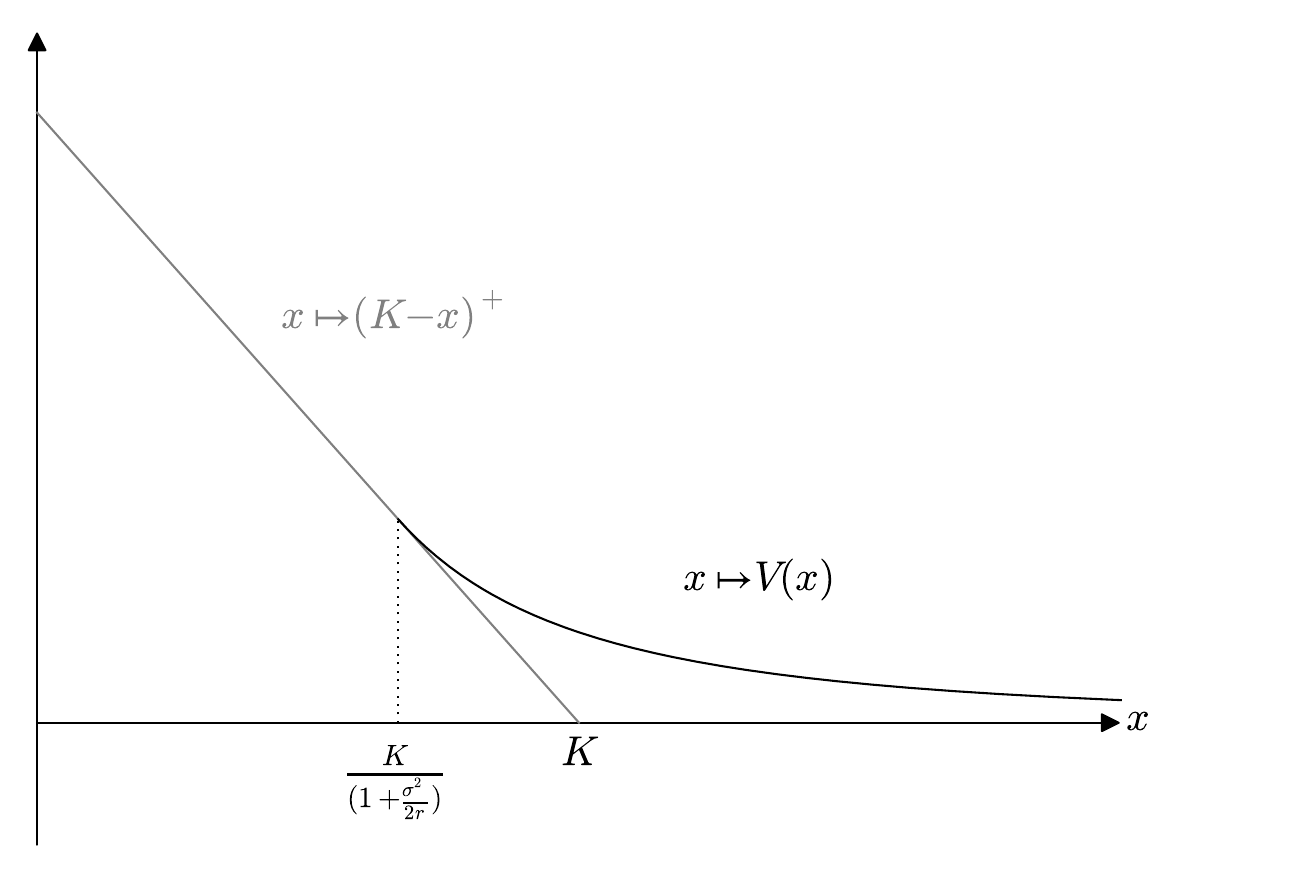} 
\end{tabular} 

\begin{minipage}{0.7\textwidth}
\caption[Transformed payoff of Perpetual put option]{In the left figure,
the transformed payoff function $\widetilde{G}$ is drawn along with its
concave biconjugate. The original payoff $G(x)=(K-x)^{+}$ and the corresponding
value function $V$ is in the figure on the right hand side.
\label{figure put}}
\end{minipage}
\end{center}
\end{figure}

Moreover, the continuation and stopping region for the problem (\ref{problem - put}) are
\begin{equation*}
\mathcal{D}:=\left\{ \left. x\in \mathbb{R}_{+}\,\right\vert \,V\left(
x\right) =G\left( x\right) \right\} =\left[ 0,x^{\ast}\right] \qquad
,\qquad \mathcal{C}:=\left\{ \left. x\in \mathbb{R}_{+}\,\right\vert
\,V\left( x\right) >G\left( x\right) \right\} =\left( x^{\ast},+\infty
\right) .
\end{equation*}
which coincides with the established solution (see for example \cite{PS} Section 25.1).
\end{example}

\section{An extension of the Legendre transform}

\label{section legendre transform}

The purpose of this section is to extend the Legendre transform so that the observations 
about optimal stopping in the previous section can be applied to optimal stopping games. 
The function $G$ represents the payoff of the maximising agent while the function $H$ represents the payoff to the minimising agent. The two
continuous gains functions $G,H$ are such that $G(x)\leq H(x)$ for all $x \in I$
and satisfy the assumptions (\ref{integ assumption}) and (\ref{boundary assumption}) unless otherwise stated.

Inspired by the transformation used in Theorem \ref{Theorem natural boundaries} introduce a pair of rescaled gains functions $W^{G}: (0,\infty) \rightarrow \mathbb{R}$ 
and $W_{H}: (0,\infty) \rightarrow \mathbb{R}$ defined via
\begin{equation}
W^{G}(y):= \left(\frac{G}{\psi}\right)\circ F^{-1}(y) \label{W^G function}
\end{equation}
and 
\begin{equation}
W_{H}(y):= \left(\frac{H}{\psi}\right)\circ F^{-1}(y). \label{W_H function}
\end{equation}
Under the assumptions (\ref{integ assumption}) and (\ref{boundary assumption}) we have
\begin{equation}
\lim_{x\downarrow 0}(W_H(x)-W^G(x))=\lim_{x \uparrow \infty}(W_H(x)-W^{G}(x))=0. \label{stuck together}
\end{equation}
This assumption can be relaxed a little as discussed in Remark \ref{Remark problem with gap} below. These definitions could equally well be formulated with respect to the other ratio of the fundamental solutions but for clarity we shall focus only on these two expressions.

The aim of this section is to define and describe a version of the 
convex biconjugate of the function $W_{H}$ which is modified to ensure that
it remains inside $\mathrm{epi}(W^G)$. At the same time, we define a version 
of the concave biconjugate of the function $W^{G}$ which is modified to ensure that
it remains inside $\mathrm{cl}(\mathbb{R}^2\setminus\mathrm{epi}(W_H))$. This is 
achieved by defining an extension of the $\varepsilon$-sub/superdifferential 
typically used in convex analysis. 

The main result in this section is Theorem \ref{Theorem duality} which
is the purely analytical version of \cite{Peskir} Theorem 4.1. Theorem \ref{Theorem
duality} shows that the convex biconjugate respecting the lower barrier $W_{H}$ and
the concave biconjugate respecting the upper barrier $W^{G}$ coincide. This `duality' result naturally follows from the spike-variations we use to construct these extensions of the Legendre transform. Moreover, the duality between these two extensions of the Legendre transform are used in Section \ref{section games} to provide a new (purely analytical) proof that the optimal stopping game (\ref{lower value})-(\ref{upper value}) exhibits both a Stackelberg and a Nash equilibrium.

For a given function $f:\mathbb{R}_+\rightarrow \mathbb{R}$, let
\begin{eqnarray*}
l_{f}^{x}(c,p) &=\sup \left\{ z\leq x\,\left\vert \,f(z)-p =c(z-x) \right. \right\} , \\
r_{f}^{x}(c,p) &=\inf \left\{ y\geq x\,\left\vert \,f(y)-p =c(y-x) \right. \right\} ,
\end{eqnarray*}
with the convention that $\sup \emptyset =0$ and $\inf
\emptyset =+\infty $. For ease of notation, let $l_{G}^{x}(c,p):=l_{W^{G}}^{x}(c,p)$ 
and $l_{H}^{x}(c,p):=l_{W_{H}}^{x}(c,p)$. The point $l_{G}^{x}(c,W_{H}(x))$
(resp. $r_{G}^{x}(c,W_{H}(x))$) is the last (resp. first) time the line passing through $(x,W_{H}(x))$ with slope $c$ intercepts $W^G$ before (resp. after) $x$.

Define the sperbdifferential of $W_H$ in the presence of the lower boundary 
$W^G$ as
\begin{equation}
\partial^{G}H(x) :=\left\{ c\in \mathbb{R}\,\left\vert \,W_H(y)-W_H(x)\geq c(y-x)\quad\forall
y\in (l_{G}^{x}(c,W_H(x)),r_{G}^{x}(c,W_H(x))) \right. \right\} .  \label{subdiff G}
\end{equation}
If the tangent $y\mapsto W_H(x) +c(y-x)$ minorises $y\mapsto W_H(y)$ prior to intercepting 
the lower boundary $y\mapsto W^G(y)$, then we refer to $c$ as a `supergradient of $W_H$ in the presence of 
$W^G$ at $x$', i.e. $c\in \partial^{G}H\left( x\right)$, as shown in Figure 
\ref{figure subsuperdiffs}. Similarly, the subdifferential of $W^G$ in the presence of the 
upper boundary $W_H$ is
\begin{equation}
\partial^{H}G(x) :=\left\{ c\in \mathbb{R}\,\left\vert\, W^G(y)-W^G(x) \leq c(y-x)\quad\forall
y\in ( l_{H}^{x}(c,W^G(x)),r_{H}^{x}(c,W^G(x))) \right. \right\} .  \label{Superdiff H}
\end{equation}
If the tangent $y\mapsto W^G(x) +c(y-x)$ majorises $y\mapsto W^G(y)$ prior to intercepting 
the upper boundary $y\mapsto W_H(y)$, then we refer to $c$ as a `subgradient of $G$ in the presence of 
$H\,$at $x$', i.e. $c\in \partial^{H}G(x)$, as illustrated in Figure \ref{figure subsuperdiffs}.

\begin{figure}[h!]
\begin{center}
\includegraphics[width=7cm]{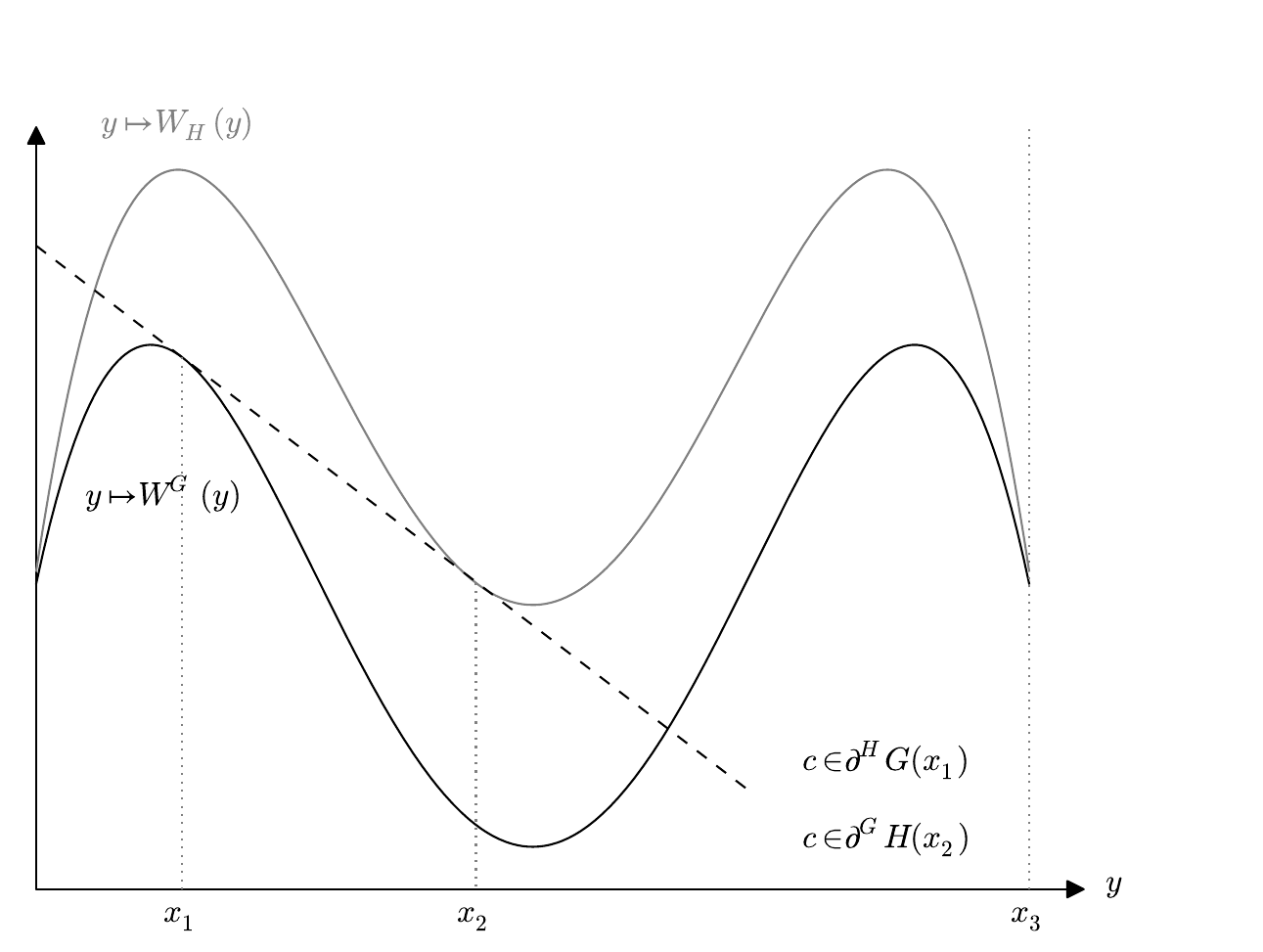}

\begin{minipage}{0.7\textwidth}
\caption[Definition modified sub-/superdifferentials]{In this figure $W^G(y)=W_H(y)$ for all $y\geq x_3$. The slope of the dashed black line, $c$, is a supergradient of $W^G$ in the presence of $W_H$ at $x_1$ and a subgradient of $W_H$ in the presence of $W^G$ at $x_2$.
\label{figure subsuperdiffs}}
\end{minipage}
\end{center}

\end{figure}

The $\delta$-superdifferential of $W_H$ in the presence of the lower boundary $W^G$ is 
defined as
\begin{align}
\partial^{G}_{\delta}H(x)&:=\left\{ c\in \mathbb{R}\,\left\vert \,W_H(y)-W_H(x)+\delta \geq c(y-x) \right.\right. \notag \\ 
& \qquad \qquad \left. \forall y\in (l_{G}^{x}(c,W_H(x)-\delta ),r_{G}^{x}(c,W_H(x)-\delta)) \right\} . \label{delta-sub}
\end{align}
When $c$ is a $\delta$-supergradient of $W_H$ in the presence of the lower boundary $W^G$, 
i.e. $c \in \partial^{G}_{\delta}H(x)$, it is possible to draw a line with slope $c$ through the 
point $(x,W_H(x)-\delta)$ which minorises $y\mapsto W_H(y)$ prior to intercepting 
the lower boundary $y\mapsto W^G(y)$. For example $0\in \partial^{G}_{\delta}G(x_2)$ in 
Figure \ref{figure subsuperdiffs2}. Similarly, the $\varepsilon$-subdifferential of $W^G$ 
in the presence of the upper boundary $W_H$ is defined as
\begin{align}
\partial^{H}_{\varepsilon}G(x) &:=\left\{ c\in \mathbb{R}\,\left\vert\, W^G(y)-W^G(x)-\varepsilon \leq c(y-x)\right.\right. \notag \\
& \qquad \qquad \left. \forall y\in ( l_{H}^{x}(c,W^G(x)+\varepsilon ),r_{H}^{x}(c,W^G(x)+\varepsilon)) \right\} \label{epsilon-super}.
\end{align}
When $c$ is an $\varepsilon$-subgradient of $W^G$ in the presence of the upper boundary $W_H$, 
i.e. $c \in \partial^{H}_{\varepsilon}G(x)$, it is possible to draw a line with slope $c$ through the 
point $(x,W^G(x)+\varepsilon)$ which dominates $y\mapsto W^G(y)$ prior to intercepting 
the upper boundary $y\mapsto W_H(y)$. For example $0 \in \partial^{H}_{\varepsilon}G(x_1)$ in 
Figure \ref{figure subsuperdiffs2}.

\begin{figure}[h!]
\begin{center}
\includegraphics[width=7cm]{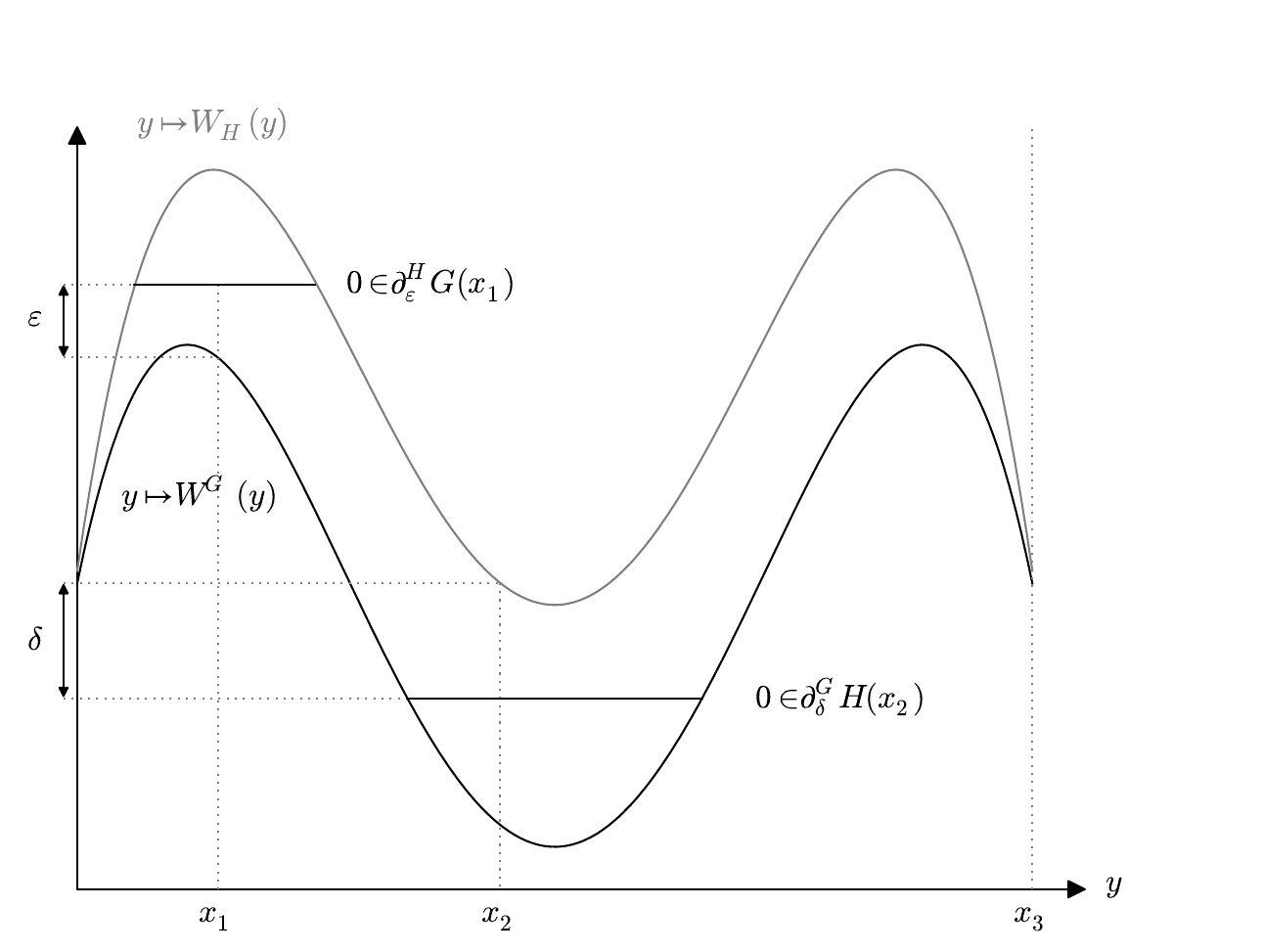}

\begin{minipage}{0.7\textwidth}
\caption[Definition modified $\delta$-sub-/$\varepsilon$-superdifferentials]{$c=0$
is a $\varepsilon$-subgradient of $W^G$ in the presence of $W_H$ at $x_1$ and
a $\delta$-supergradient of $W_H$ in the presence of $W^G$ at $x_2$.
\label{figure subsuperdiffs2}}
\end{minipage}
\end{center}
\end{figure}
 
Let 
\begin{equation}
\varepsilon_{c}^{\ast}(x) :=\inf \left\{ \left. \varepsilon \in [ 0,W_H(x) -W^G(x)] 
\,\right\vert \;c\in \partial^{H}_{\varepsilon}G(x) \right\} \label{epsilon-star}
\end{equation}
and 
\begin{equation}
\delta_{c}^{\ast}(x) :=\inf \left\{ \left. \delta \in [ 0,W_H(x) -W^G(x)] 
\,\right\vert \;c\in \partial^{G}_{\delta}H(x) \right\} \label{delta-star}
\end{equation}
which are the smallest spike variations that can be made in $W^G$, resp. $W_H$ at $x$ 
such that $c \in \partial^{H}_{\varepsilon}G(x)$, resp. $c \in \partial^{G}_{\delta}H(x)$.
The quantities (\ref{epsilon-star}) and (\ref{delta-star}) are illustrated in Figure \ref{Figure important properties}.

\begin{remark}  
\label{Remark properties of delta-star}
It follows from these definitions that for all $\delta> \delta_c^{\ast}(x)$ 
\begin{equation*}
W_H(y)> W_H(x)- \delta +c(y-x) \qquad \forall\,y\in [l_G^x(c,W_H(x)-\delta ),r_G^x(c, W_H(x)-\delta)] .
\end{equation*}
Whereas, for $\delta < \delta_c^{\ast}(x)$ there exists $x^{\prime} \in [l_G^x(c,W_H(x)-\delta ),r_G^x(c, W_H(x)-\delta)]$
such that $W_H(x^{\prime})< W_H(x)- \delta +c(x^{\prime}-x)$. These two statements imply that either: 
\begin{itemize}
\item[(a)] There exists $z\in [l_G^x(c,W_H(x)-\delta_c^{\ast}(x)),r_G^x(c, W_H(x)-\delta_c^{\ast}(x))]$ such that 
$c \in \partial^G H(z)$ and/or 
\item[(b)] $\delta_c^{\ast}(z)=W_H(z)-W^G(z)$ for $z=l_G^x(c,W_H(x)-\delta_c^{\ast}(x))$ and/or 
$z=r_G^x(c,W_H(x)-\delta_c^{\ast}(x))$. 
\end{itemize}
The top left panel of Figure \ref{Figure important properties} shows a point where only condition (a) holds, whereas
the bottom left panel illustrates a situation where only condition (b) holds. 
Similarly, for all $\varepsilon > \varepsilon_c^{\ast}(x)$
\begin{equation*}
W^G(y)> W^G(x)+\varepsilon +c(y-x) \qquad \forall\,y\in [l_H^x(c,W^G(x)+\varepsilon ),r_H^x(c, W^G(x)+\varepsilon)] .
\end{equation*}
Whereas, for $\varepsilon < \varepsilon_c^{\ast}(x)$ there exists $x^{\prime} \in [l_H^x(c,W^G(x)-
\varepsilon),r_H^x(c, W^G(x)-\varepsilon)]$ such that $W^G(x^{\prime})< W^G(x)+\varepsilon +c(x^{\prime}-x)$.
These two statements imply that either: 
\begin{itemize}
\item[(a)] There exists $z\in [l_H^x(c,W^G(x)+\varepsilon_c^{\ast}(x)),r_H^x(c,W^G(x)+\varepsilon_c^{\ast}(x))]$ 
such that $c \in \partial^H G(z)$ and/or 
\item[(b)] $\delta_c^{\ast}(z)=W_H(z)-W^G(z)$ for $z=l_G^x(c,W_H(x)-\delta_c^{\ast}(x))$ and/or 
$z=r_G^x(c,W_H(x)-\delta_c^{\ast}(x))$.
\end{itemize}
The top right panel of Figure \ref{Figure important properties} shows a point where only condition (a) holds, whereas
the bottom right panel illustrates a situation where only condition (b) holds.
\end{remark}

\begin{figure}[h!]
\begin{center}

\begin{tabular}{cc}
\includegraphics[width=7cm]{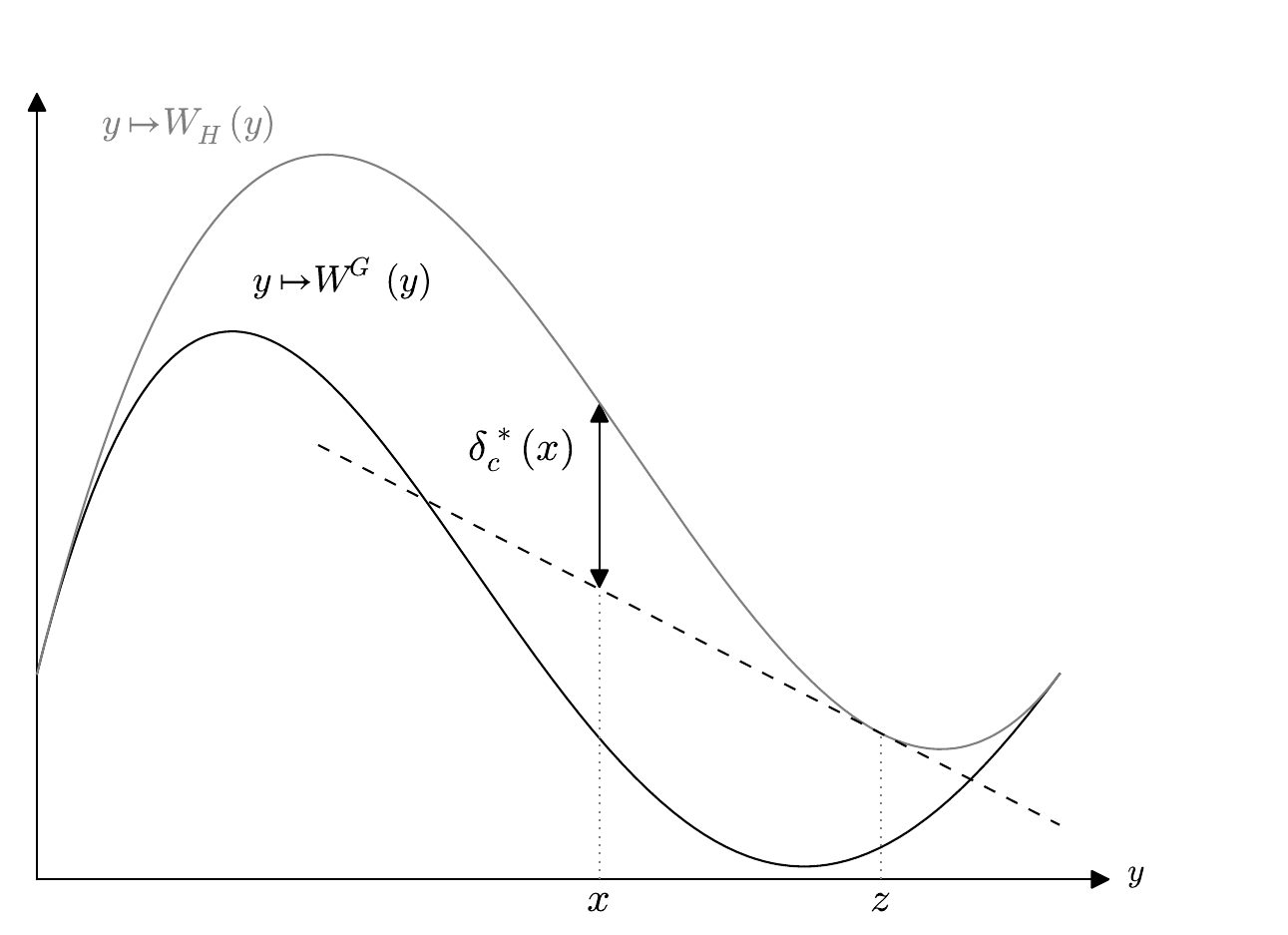} & \includegraphics[width=7cm]{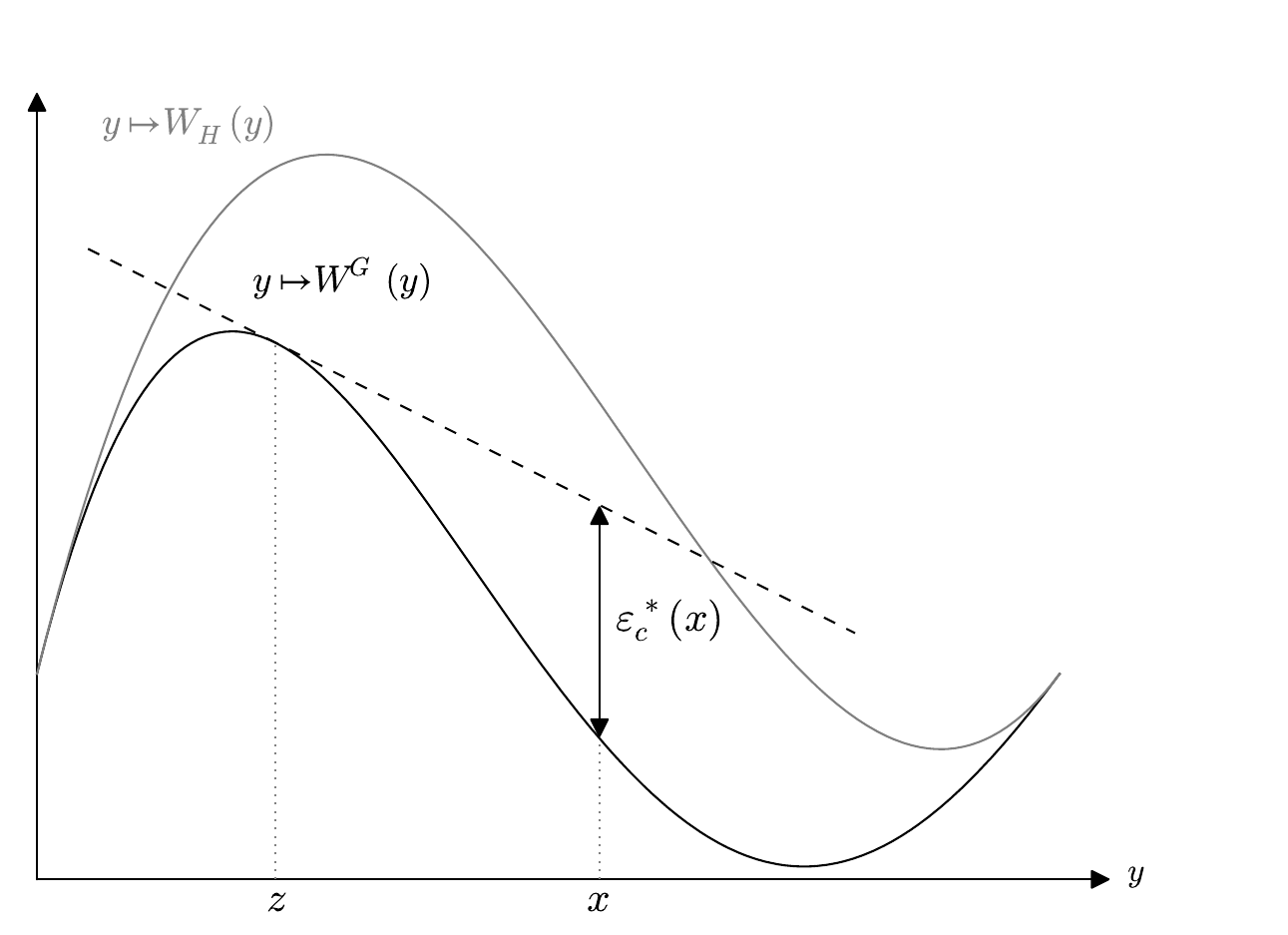} \\
\includegraphics[width=7cm]{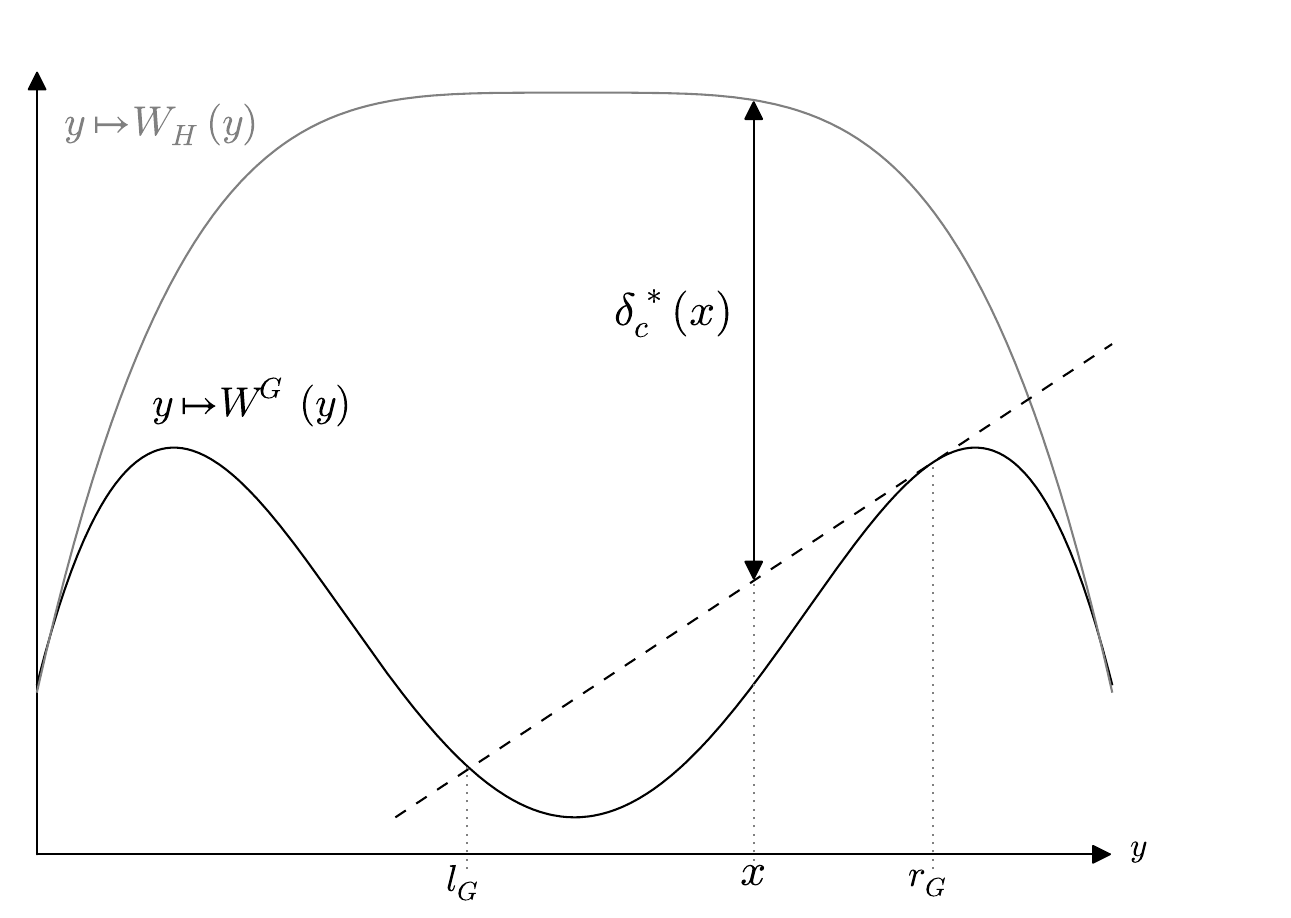} & \includegraphics[width=7cm]{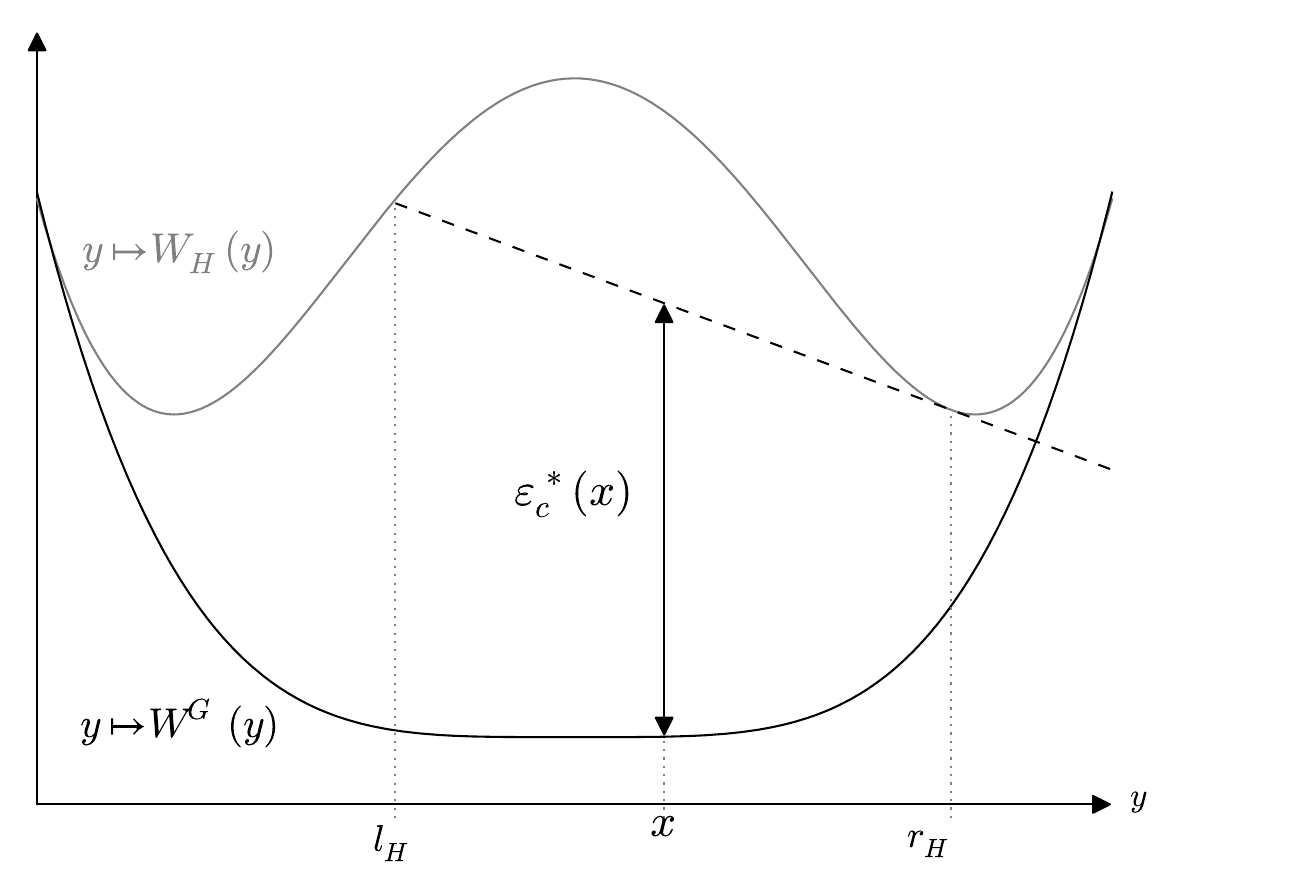}
\end{tabular} 

\begin{minipage}{0.7\textwidth}
\caption[Properties of $\delta_c^{\ast}(x)$ and $\varepsilon_c^{\ast}(x)$]{The two figures on the left illustrate $\delta_c^{\ast}(x)$ and the two figures on the right illustrate $\varepsilon_c^{\ast}(x)$. For ease of notation, in the bottom left figure $l_G:=l^x_G(c,W_H(x)-\delta^{\ast}_c(x))$ and $r_G:=r^x_G(c,W_H(x)-\delta^{\ast}_c(x))$ and in the bottom right figure $l_H:=l^x_H(c,W^G(x)+\varepsilon^{\ast}_c(x))$ and $r_H:=r^x_H(c,W^G(x)+\varepsilon^{\ast}_c(x))$.
\label{Figure important properties}}
\end{minipage}
\end{center}

\end{figure}

Furthermore, let 
\begin{equation}
\varepsilon^{\ast}(x):= \inf_{c\in \mathbb{R}}\varepsilon_{c}^{\ast}(x) \qquad , \qquad
\delta^{\ast}(x):= \inf_{c\in \mathbb{R}}\delta_{c}^{\ast}(x) \label{epsilon^ast}
\end{equation}
which are the smallest spike variation that can be made in $W^G$ (resp. $W_H$) at $x$ 
such that the set $\partial^{H}_{\varepsilon}G(x)$ (resp. $\partial^{G}_{\delta}H(x)$)
is non-empty. The first result in this section shows that there is `no gap' between 
the minimum spike variation in $W_H$ downwards admitting a $\delta$-supergradient 
and the minimum spike variation in $W^G$ upwards admitting a $\varepsilon$-subgradient.

\begin{lemma}
\label{lemma no gap}
Suppose that $W^G: (0,\infty) \rightarrow \mathbb{R}$ and $W_H: (0,\infty) \rightarrow \mathbb{R}$ are as defined in (\ref{W^G function}) and (\ref{W_H function}) and that (\ref{stuck together}) holds, then: for all $x> 0$ and $c\in \mathbb{R}$
\begin{equation}
W^G(x)+\varepsilon^{\ast}_c(x) \geq  W_H(x)-\delta^{\ast}_c(x) \label{pre no gap} 
\end{equation}
and 
\begin{equation}
W^G(x)+\varepsilon^{\ast}(x) = W_H(x)-\delta^{\ast}(x) . \label{no gap} 
\end{equation}
\end{lemma}

\begin{proof}
Fix $x> 0$, $c\in \mathbb{R}$ and take $\delta > \delta_c^{\ast}(x)$. It follows from 
the definition of $\delta_c^{\ast}(x)$ that 
\begin{equation*}
W_H(y)> W_H(x)- \delta +c(y-x) \qquad \forall\,y\in [l_G^x(c,W_H(x)-\delta ),r_G^x(c, W_H(x)-\delta)]
\end{equation*}
and 
\begin{equation}
[l_G^x(c,W_H(x)-\delta ),r_G^x(c,W_H(x)-\delta )] \subset [l_H^x(c,W_H(x)-\delta ),r_H^x(c,W_H(x)-\delta )] 
\label{inclusion}
\end{equation}
where the inequality and inclusion are strict to avoid contradicting that $\delta>\delta_c^{\ast}(x)$.  
Due to the properties of the line $y\mapsto W_H(x)-\delta_c^{\ast}(x)+c(y-x)$ described in Remark \ref{Remark 
properties of delta-star}, it follows from the inclusion (\ref{inclusion}) and the continuity of $G$ that 
\begin{equation}
\exists\,y \in [l_H^x(c,W_H(x)-\delta ),r_H^x(c,W_H(x)-\delta )] \quad \mathrm{s.t.} \quad
W_H(x)-\delta + c(y-x) < W^G(y) \label{pass below}
\end{equation}
Let $\varepsilon^{\prime}:= W_H(x)-W^G(x)-\delta$ so that (\ref{pass below}) is equivalent to
\begin{equation*}
\exists\,y \in [l_H^x(c,W_H(x)-\delta ),r_H^x(c,W_H(x)-\delta )] \quad \mathrm{s.t.} \quad 
W^G(x)-\varepsilon^{\prime} + c(y-x) < W^G(y) .
\end{equation*}
It follows from the definition of $\varepsilon_c^{\ast}(x)$ 
that $\varepsilon^{\prime}(x) <\varepsilon_c^{\ast}(x)$ or equivalently 
\begin{equation}
W_H(x)-\delta <W^G(x)+\varepsilon^{\ast}_c(x) \quad \forall\,\delta>\delta_c^{\ast}(x) . \label{pass ineq}
\end{equation} 
Take any sequence of real numbers $(\delta_n)_{n\geq 1}$ such that $\delta_n \in (\delta_c^{\ast}(x),W_H(x)-W^G(x))$ for all $n\geq 1$ 
and $\lim_{n \rightarrow \infty}\delta_n = \delta_c^{\ast}(x)$ then (\ref{pass ineq}) shows that 
$W_H(x)-\delta_n < W^G(x)+ \varepsilon_c^{\ast}(x)$ for all $n\geq 1$. Taking the limit as 
$n \rightarrow \infty$ of both sides of this inequality we obtain (\ref{pre no gap}).

Suppose that for some $c\in \mathbb{R}$ that $W^G(x)+\varepsilon^{\ast}_c(x) = W_H(x)-\delta^{\ast}_c(x)$. 
When the assumption (\ref{stuck together}) holds the line $g(y):=W^G(x)+\varepsilon^{\ast}_c(x) 
+ c(y-x)$ has both sets of properties described in Remark \ref{Remark properties of delta-star} so is 
the only line passing through $(x,p)$ for some $p\in (W^G(x),W_H(x))$ which both dominates $y\mapsto W^G(y)$ on 
$[l_H^x(c,p),r_H^x(c,p)]$ and minorises $y\mapsto W_H(y)$ on $[l_G^x(c,p),r_G^x(c,p)]$. 
When (\ref{stuck together}) fails, $y\mapsto g(y)$ may not be the only line with these properties and 
(\ref{no gap}) may not hold, this case is examined further in Remark 
\ref{Remark problem with gap} below. Consequently, it follows from the continuity of $G$ and $H$ and 
the assumption (\ref{stuck together}) that for any $c^{\prime} \neq c$ the line 
$y\mapsto h(y)=W^G(x)+\varepsilon_c^{\ast}(x)+c^{\prime}(y-x)$ must either satisfy both
\begin{align}
l_G^x(c^{\prime},W_H(x)-\delta_c^{\ast}(x)) &< l_H^x(c^{\prime},W_H(x)-\delta_c^{\ast}(x)), \notag \\
r_G^x(c^{\prime},W_H(x)-\delta_c^{\ast}(x)) &< r_H^x(c^{\prime},W_H(x)-\delta_c^{\ast}(x)), \label{condition(ii)}
\end{align}
or 
\begin{align}
l_G^x(c^{\prime},W_H(x)-\delta_c^{\ast}(x)) &> l_H^x(c^{\prime},W_H(x)-\delta_c^{\ast}(x)) , \notag \\
r_G^x(c^{\prime},W_H(x)-\delta_c^{\ast}(x)) &> r_H^x(c^{\prime},W_H(x)-\delta_c^{\ast}(x)) \label{condition(i)}.
\end{align}

\begin{figure}[h!]
\begin{center}

\begin{tabular}{cc}
\includegraphics[width=7cm]{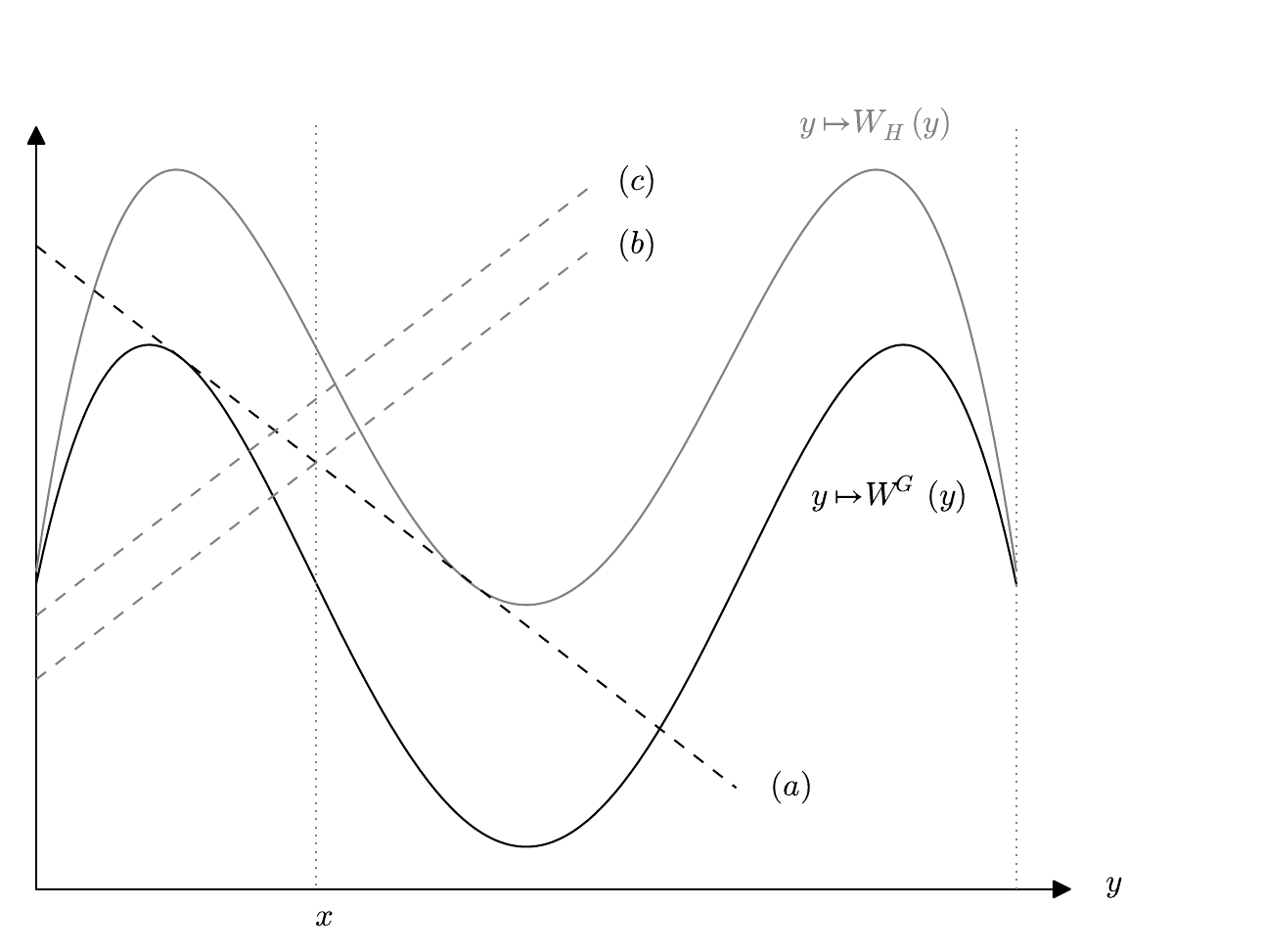}  & \includegraphics[width=7cm]{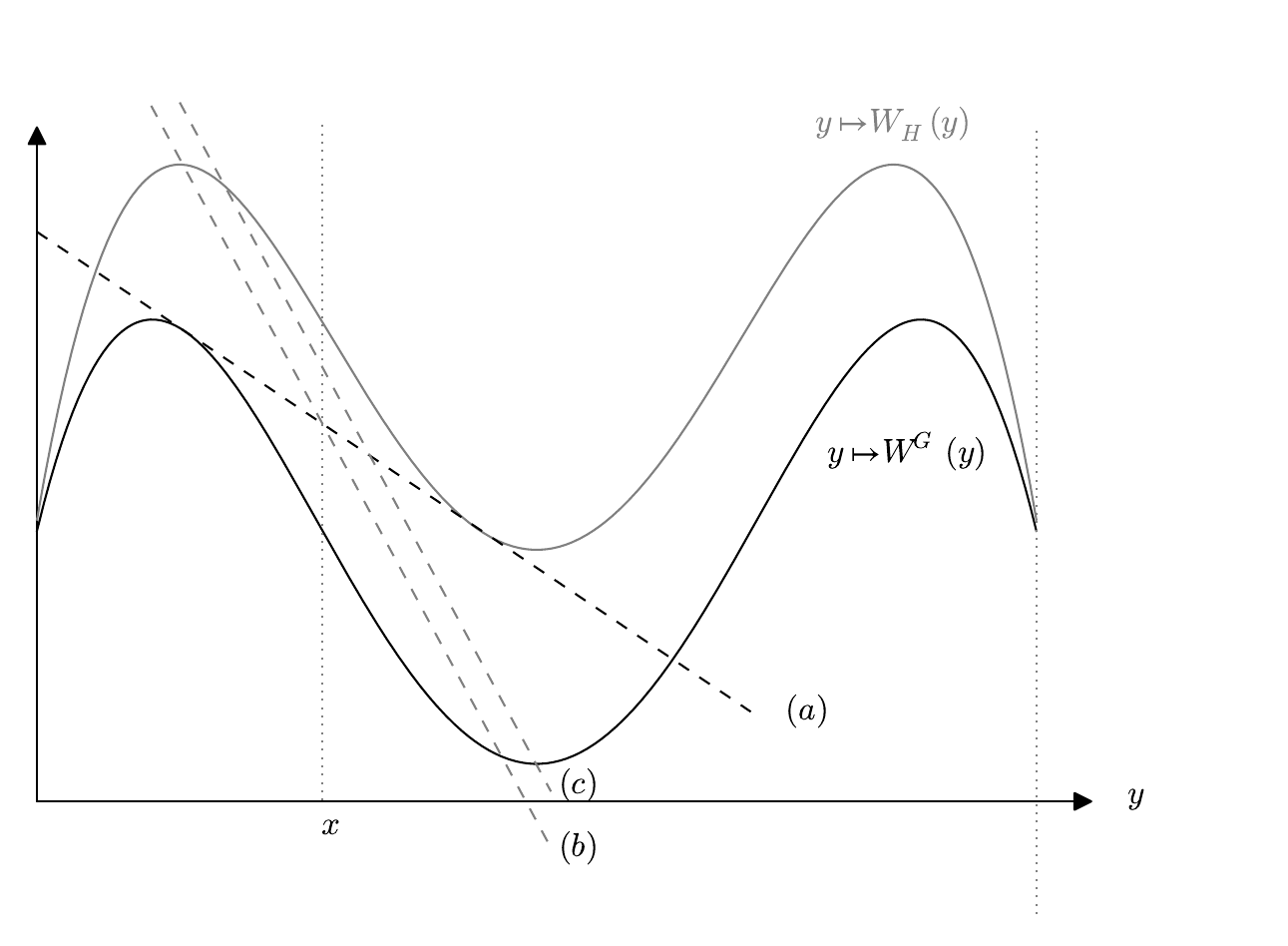}
\end{tabular} 

\begin{minipage}{0.7\textwidth}
\caption[Properties of $\delta_c^{\ast}(x)$ when $W^G(x)+\varepsilon_c^{\ast}(x)=W_H(x)-\delta_c^{\ast}(x)$]{These 
figures illustrate the properties used in the proof of Lemma \ref{lemma no gap}.
\label{Figure rotate}}
\end{minipage}
\end{center}

\end{figure}

In the right panel of Figure \ref{Figure rotate} the line marked (b) has the properties (\ref{condition(ii)}) whereas in the left panel the line marked (b) has the properties (\ref{condition(i)}). When (\ref{stuck together}) fails one (or more) of the inequalities in both of these statements can hold with equality which invalidates the following 
argument. Next the line $y\mapsto h(y)$ is shifted upwards so that it passes through $(x,W_H(x)-\delta)$ for 
some $\delta \in[0,\delta_c^{\ast}(x)]$. As is illustrated by the lines marked (c) in Figure \ref{Figure rotate}
the properties (\ref{condition(i)}) imply that
\begin{equation*}
l_H^x(c^{\prime},H(x)-\delta) \in [l_H^x(c^{\prime},H(x)-\delta_{c^{\prime}}(x)),x]
\end{equation*}
and hence $l_H^x(c^{\prime},H(x)-\delta)>l_G^x(c^{\prime},H(x)-\delta)$. Whereas if (\ref{condition(ii)}) holds then
\begin{equation*}
r_H^x(c^{\prime},H(x)-\delta) \in [x,r_H^x(c^{\prime},H(x)-\delta_{c^{\prime}}(x))] .
\end{equation*} 
which implies that $r_H^x(c^{\prime},H(x)-\delta)<r_G^x(c^{\prime},H(x)-\delta)$. In both cases $c^{\prime} 
\notin \partial_{\delta}^G H(x)$ for arbitrary $\delta < \delta_c^{\ast}(c)$. Moreover,
as $c^{\prime}$ was arbitrary we have shown that $\partial_{\delta}^G H(x)=\emptyset$ for all 
$\delta < \delta_c^{\ast}(c)$. We may conclude that when for some $c\in \mathbb{R}$ equality holds in 
(\ref{pre no gap}) that $\delta_c^{\ast}(x)= \delta^{\ast}(x)$. A symmetric argument can be used to show that 
when equality holds for some $c\in \mathbb{R}$ equality holds in (\ref{pre no gap}) that $\varepsilon_c^{\ast}(x)= 
\varepsilon^{\ast}(x)$ from which we deduce (\ref{no gap}) holds. 
\end{proof}

\begin{figure}[h!]
\begin{center}
\includegraphics[width=7cm]{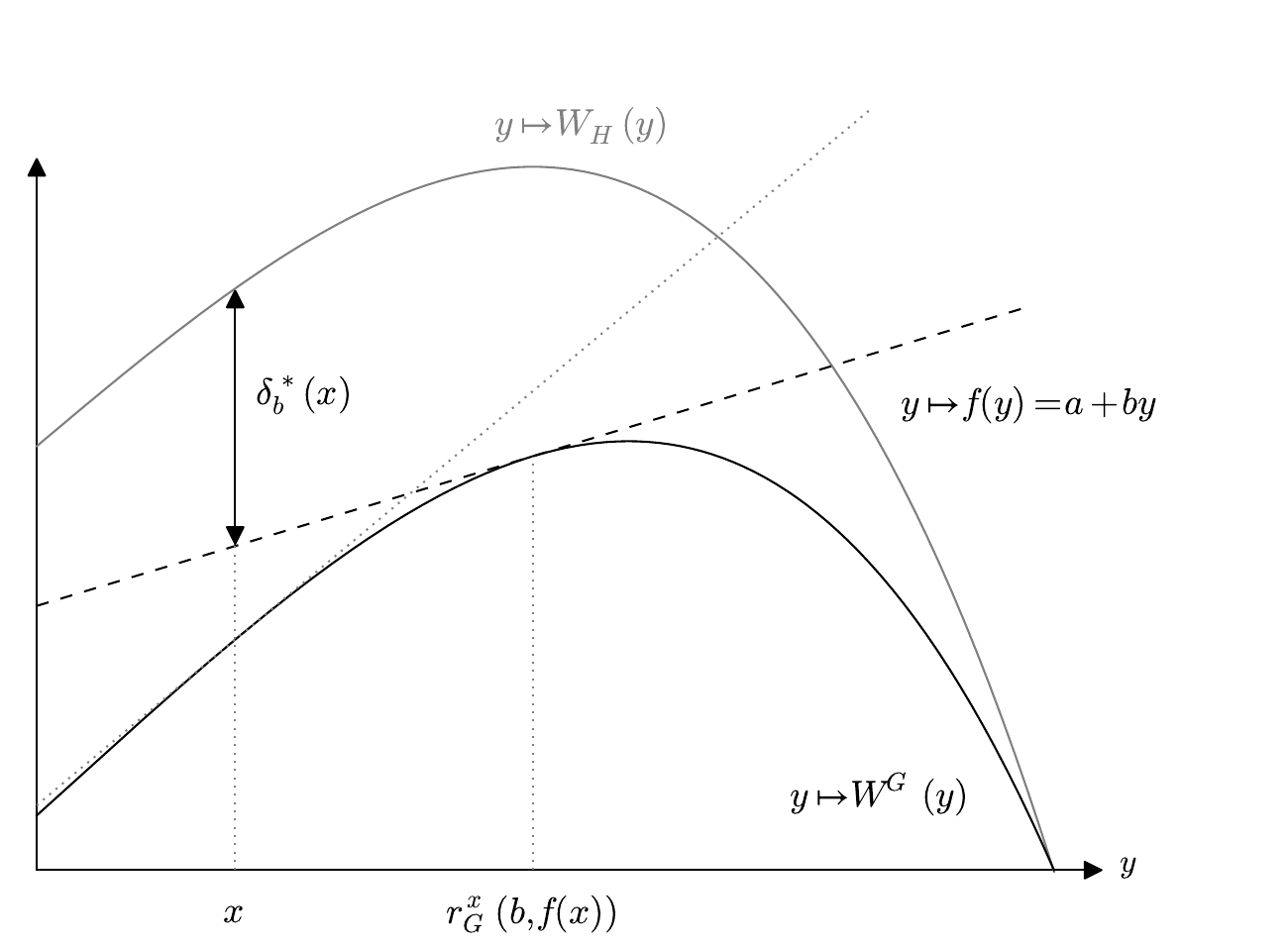}
\begin{minipage}{0.7\textwidth}
\caption[An example where $W_H(0+)>W^G(0+)$]{When $W_H(0+)>W^G(0+)$ it may be the case that condition (\ref{no gap})
in Lemma \ref{lemma no gap} fails.
\label{Figure problem with boundaries}}
\end{minipage}
\end{center}

\end{figure} 

\begin{remark}
\label{Remark problem with gap}
In the previous lemma it is essential to assume that (\ref{stuck together}) holds. When this is not the case 
it may be the case that $W^G(x)+\varepsilon^{\ast}(x)<W_H(x)+\delta^{\ast}(x)$ as is illustrated in Figure \ref{Figure problem with boundaries}. Suppose that as in the illustration $W^G(0+)<W_H(0+)$ and 
there is a $x>0$ such that for some $c^{\prime}\in \mathbb{R}$, $c^{\prime} \in \partial^H G(x)$, $l_G^x(c^{\prime},W^G(x))=l_H^x(c^{\prime},W^G(x))=0$ and 
$W^G(0+)<W^G(x)-c'x<W_H(0+)$. Take $a \in (W^G(x)-c^{\prime}x,W_H(0+))$ such that $\exists b<c^{\prime}$ with 
the properties: (i) $a+bx \in (W^G(x),W_H(x))$, (ii) $r_G^x(b,a+bx)<r_H^x(b,a+bx)$ and (iii) 
$b \in \partial^HG(r_G^x(b,a+bx))$. We have chosen $a$, $b$ and $x$ such that $\delta^{\ast}_b(x)=W_H(x)-(a+bx)$ hence $\delta^{\ast}(x)\leq W_H(x)-(a+bx) <
W_H(x)-W^{G}(x)$. However, by construction $\partial^H G(x)\neq \emptyset$ so $\varepsilon^{\ast}(x)=0$ which 
implies
\begin{equation*}
W^G(x)+ \varepsilon^{\ast}(x)< W_H(x)-\delta^{\ast}(x),
\end{equation*}
so in this case (\ref{no gap}) fails. 
\end{remark}

\begin{remark}
\label{Remark duality}
The assumption that (\ref{stuck together}) holds in Lemma \ref{lemma no gap} can be relaxed in the following way. When $W^G(0+)<W_H(0+)$, take any $w_0 \in [W^G(0+),W_H(0+)]$ and extend $W^G$ (respectively $W_H$) onto $[0,\infty)$ by allowing $W^G$ to be multivalued at zero taking all values $[W^G(0+),w_0]$ (respectively $[w_0,W_H(0+)]$). If we use this extension of $W^G$ and $W_H$ and the intervals in the definitions of $\partial^G_{\varepsilon}H(x)$ and $\partial_H^{\delta}G(x)$ are take to be closed (as apposed to open) then (\ref{no gap}) in Lemma \ref{lemma no gap} holds without the need to assume (\ref{stuck together}) as the problem discussed in Remark \ref{Remark problem with gap} can not occur. However, the choice of $w_0$ not only completely determines $\delta^{\ast}(0+)$ and $\varepsilon^{\ast}(0+)$ but the values of the functions $x \mapsto \delta^{\ast}(x)$ and $x \mapsto \varepsilon^{\ast}(x)$ on $(0,z)$ for some $z\geq 0$.
\end{remark}

\begin{figure}[h!] 
\begin{center}
\includegraphics[width=7cm]{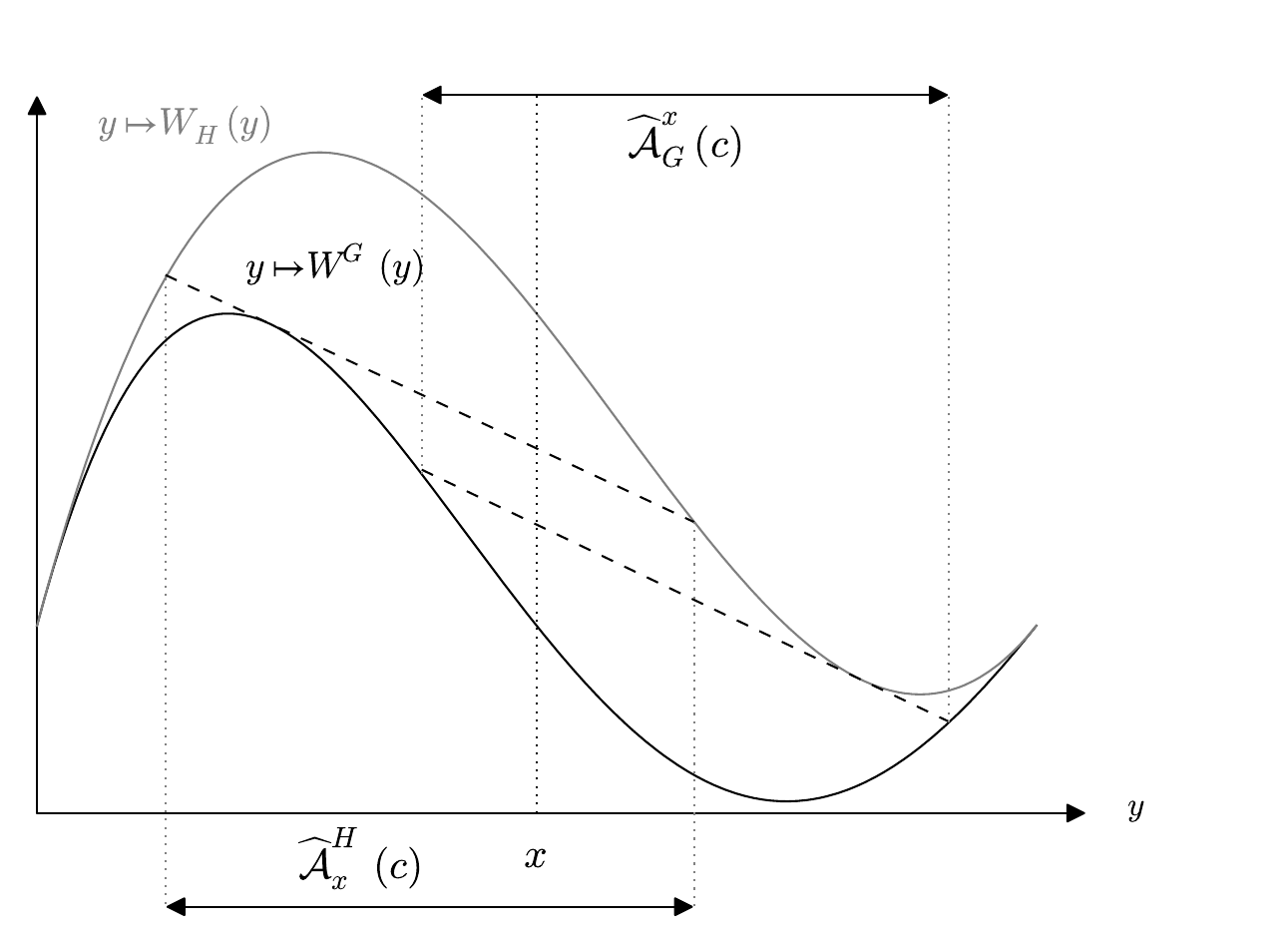}
\begin{minipage}{0.7\textwidth}
\caption[Definition of admissible neighborhoods $\widehat{\mathcal{A}}_{G}^{x}(c)$ and 
$\widehat{\mathcal{A}}_{x}^{H}(c)$]{The sets $\widehat{\mathcal{A}}_{G}^{x}(c)$ and 
$\widehat{\mathcal{A}}_{x}^{H}(c)$ are defined in (\ref{pre-strategy space G}) and (\ref{pre-strategy space H}) are marked when $c$ is the slope of the pair of parallel lines.
\label{Figure neighborhoods 1}}
\end{minipage}
\end{center}
\end{figure}

For a given $x> 0$, define two sets of admissible neighborhoods using
\begin{eqnarray}
\widehat{\mathcal{A}}_{G}^{x}(c) &:=&\bigcup\limits_{p\in [
W^G(x),W_H(x)] }\{ y> 0\,\vert
\,l_{H}^{x}(c,p) \leq l_{G}^{x}(c,p) \leq y\leq
r_{G}^{x}(c,p) \leq r_{H}^{x}(c,p)\} , \label{pre-strategy space G} \\
\widehat{\mathcal{A}}_{x}^{H}(c) &:=&\bigcup\limits_{p\in [
W^G(x),W_H(x)]}\{ y> 0\,\vert\,l_{G}^{x}(c,p) \leq l_{H}^{x}(c,p) \leq y\leq
r_{H}^{x}(c,p) \leq r_{G}^{x}(c,p)\} . \label{pre-strategy space H}
\end{eqnarray}
These admissible neighborhoods are illustrated in Figure \ref{Figure neighborhoods 1}. The next result 
describes a pair of functions which coincide with $y\mapsto W^G(y)+\varepsilon^{\ast}(y)$ and 
$y\mapsto W_H(y)-\delta^{\ast}(y)$.

\begin{proposition}
\label{proposition supsup}
For $x> 0$, define a pair of functions
\begin{eqnarray}
(W^G)_H^{\ast \ast}(x) &:=& \sup_{z\in \mathbb{R}}\sup_{y\in 
\widehat{\mathcal{A}}^{x}_{G}(z)}\left(z(x-y) +W^G(y)\right) ,  \label{pre-concave biconjugate H} \\
(W_H)_{\ast\ast}^G(x) &:=& \inf_{z\in \mathbb{R}}\inf_{y\in \widehat{\mathcal{A}}^{H}_{x}(z)}
\left( z(x-y) +W_H(y)\right) ,  \label{pre-convex biconjugate G}
\end{eqnarray}
then $(W^G)_H^{\ast \ast}(x)=W_H(x)-\delta^{\ast}(x)$ and $(W_H)^G_{\ast \ast}(x)=
W^G(x)+\varepsilon^{\ast}(x)$.
\end{proposition}

\begin{proof}
Fix a $x> 0$ and $c\in \mathbb{R}$, it follows from the definition of $\delta_c^{\ast}(x)$ 
and the continuity of $G$ that the set (\ref{pre-strategy space G}) satisfies 
\begin{equation*}
\widehat{\mathcal{A}}_{G}^{x}(c) = [l_G^x(c,W_H(x)-\delta_c^{\ast}(x)),r_G^x(c,W_H(x)-\delta_c^{\ast}(x))] .
\end{equation*}
By definition $c \in \partial_{\delta}^G H(x)$ for all $\delta> \delta_c^{\ast}(x)$, which is equivalent to
\begin{equation}
W_H(x) - \delta_c^{\ast}(x)+ c(y-x) \geq W^G(y) \quad \forall\,y\in \widehat{\mathcal{A}}_G^x(c) \label{step sup1}
\end{equation}
and equality holds in (\ref{step sup1}) for $y=l_G^x(c,W_H(x)-\delta_c^{\ast}(x))$ and 
$y=r_G^x(c,W_H(x)-\delta_c^{\ast}(x))$. Hence,
\begin{equation*}
\sup_{y\in \widehat{\mathcal{A}}_{G}^{x}(c)} (W^G(y)-cy) = W^G(r_G^x(c,W_H(x)-\delta_c^{\ast}(x)))
-c r_G^x(c,W_H(x)-\delta_c^{\ast}(x))
\end{equation*}
and 
\begin{align*}
(W^G)_H^{\ast \ast}(x) &:= \sup_{c\in \mathbb{R}}\sup_{y\in \widehat{\mathcal{A}}^{x}_{G}(c)}
\left(c(x-y) +W^G(y)\right) \\
&=\sup_{c\in \mathbb{R}} \left( c(x-r_G^x(c,W_H(x)-\delta_c^{\ast}(x)))+W^G(r_G^x(c,W_H(x)-
\delta_c^{\ast}(x)))\right) \\
&=\sup_{c\in \mathbb{R}} \left( W_H(x)-\delta_c^{\ast}(x)\right) = W_H(x)-\delta^{\ast}(x).
\end{align*}
Thus we have shown that $(W^G)_H^{\ast \ast}(x)=W_H(x)-\delta^{\ast}(x)$ for all $x>0$.
A symmetric argument can be used to show that $(W_H)^G_{\ast \ast}(x)=W^G(x)+\varepsilon^{\ast}(x)$. \medskip
\end{proof}

The previous result holds without the need to assume that (\ref{stuck together}) holds. At first glance, the notation used in the previous result may appear to be the wrong way around but when (\ref{stuck together}) holds Lemma \ref{lemma no gap} implies that $(W^G)_H^{\ast \ast}(x)=W^G(x)+\varepsilon^{\ast}(x)$ and 
$(W_H)^G_{\ast \ast}(x)=W_H(x)-\delta^{\ast}(x)$.

\begin{remark} 
\label{Remark Goran's duality}For a given $x> 0$ consider the sets
\begin{eqnarray*}
A_{1}(x) &:=&\left\{ p\in [W^G(x),W_H(x)] \,\left\vert\,\exists\,c\in\mathbb{R}\;\mathrm{s.t.}\;
[l_{G}^{x}(c,p),r_{G}^{x}(c,p)] \subseteq \widehat{\mathcal{A}}_{G}^{x}(c)\right.\right\} , \\
A_{2}(x) &:=&\left\{ p\in [W^G(x),W_H(x)] \,\left\vert\,\exists\,c\in\mathbb{R}\;\mathrm{s.t.}\;
[l_{H}^{x}(c,p),r_{H}^{x}(c,p)] \subseteq \widehat{\mathcal{A}}_{x}^{H}(c)\right.\right\} .
\end{eqnarray*}
The `dual interpretation' provided in \cite{Peskir} illustrates that 
$(W^G)_{H}^{\ast\ast}(x)$ can be constructed by maximising over functions of the form 
$y\mapsto p+c(y-x)$ which hit $W^G$ before $W_H$, i.e. $(W^G)_{H}^{\ast\ast}(x)
=\sup A_{1}(x)$ as shown in Proposition \ref{proposition supsup} or by minimising over the 
functions of the form $y\mapsto p+c(y-x)$ which hit $W_H$ before $W^G$ i.e. 
$(W^G)_{H}^{\ast\ast}(x) =\inf A_{2}(x)$.
\end{remark}

\begin{figure}[h!]
\begin{center}

\begin{tabular}{cc}
\includegraphics[width=7cm]{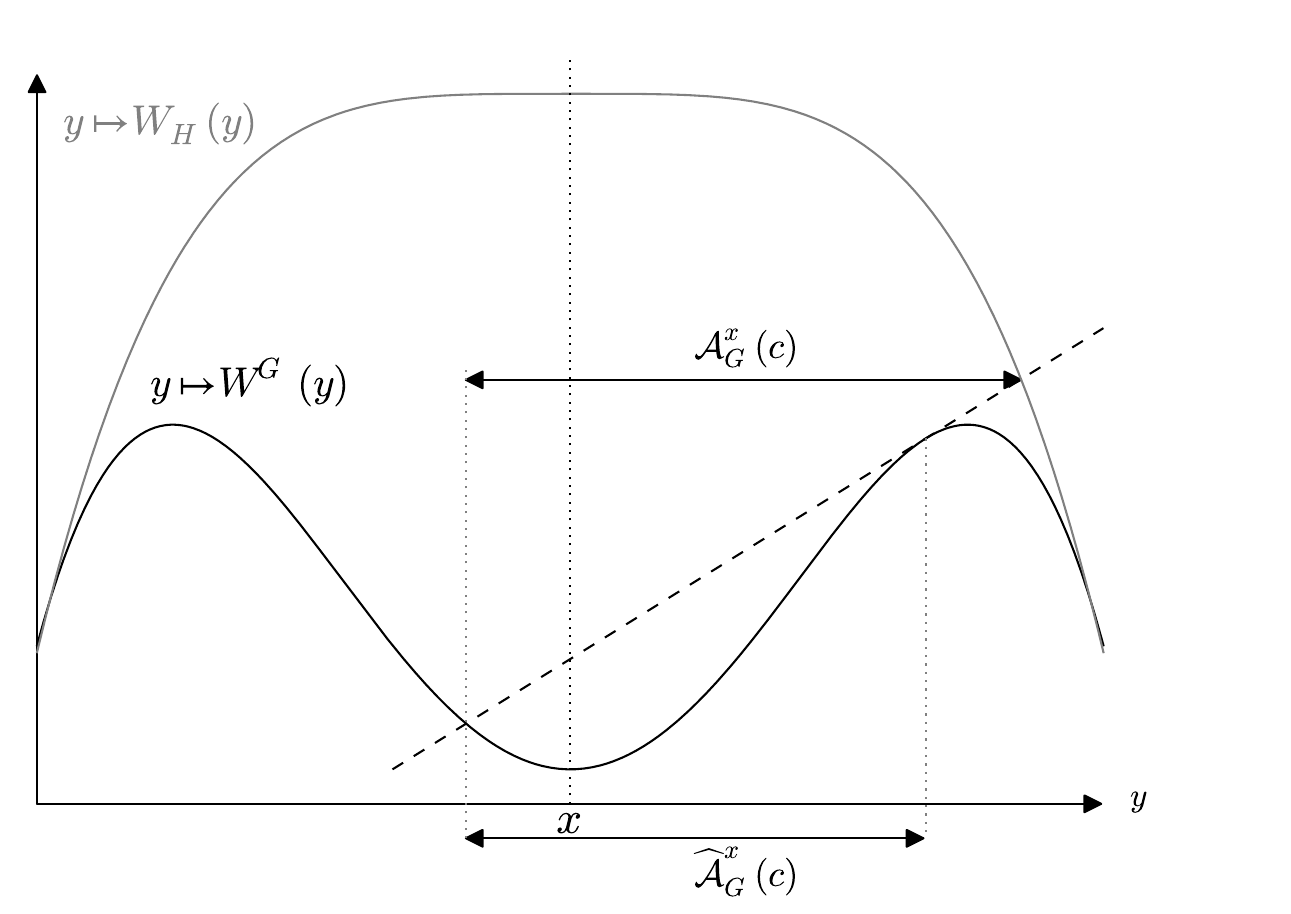} & \includegraphics[width=7cm]{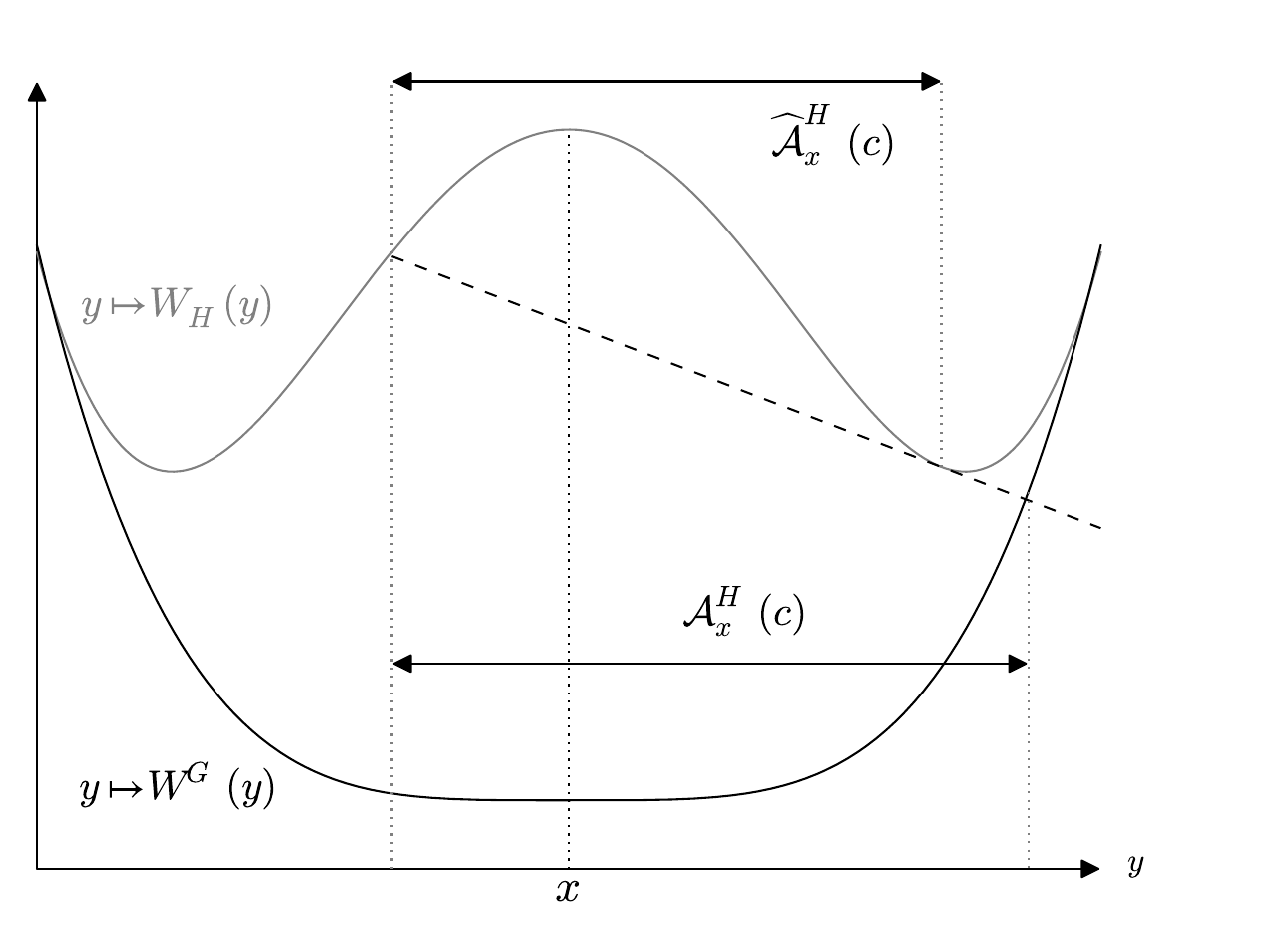}
\end{tabular} 

\begin{minipage}{0.75\textwidth}
\caption[Definition of admissible neighborhoods $\mathcal{A}_{G}^{x}(c)$ and 
$\mathcal{A}_{x}^{H}(c)$]{The left figure shows a situation when 
$\widehat{\mathcal{A}}_{G}^{x}(c) \subset \mathcal{A}_{G}^{x}(c)$ whereas the 
figure on the right shows a situation when $\widehat{\mathcal{A}}^{H}_{x}(c) 
\subset \mathcal{A}^{H}_{x}(c)$.
\label{Figure neighborhoods 2}}
\end{minipage}
\end{center}
\end{figure}

In view of the previous remark, Proposition \ref{proposition supsup} can be viewed as 
a preliminary version of the next result which shows that the functions 
(\ref{pre-concave biconjugate H}) and (\ref{pre-convex biconjugate G})
can be viewed as an extension of the concave biconjugate of $y \mapsto W^G(y)$ and 
the convex biconjugate of $y \mapsto W_H(y)$. To this end, let 
\begin{eqnarray*}
L^{x}(c,p) &=& \sup \{ z\leq x\,\vert\, p+c(z-x) \notin \mathrm{cl}(
\mathrm{epi}(W^G)\setminus \mathrm{epi}(W_H)) \} , \\
R^{x}(c,p) &=& \inf \{ y\geq x\,\vert\, p+c(y-x) \notin \mathrm{cl}(
\mathrm{epi}(W^G)\setminus \mathrm{epi}(W_H)) \} .
\end{eqnarray*}
For a given $x> 0$, let $p(x)=W_H(x)-W^G(x)$ and define two sets of admissible neighborhoods which are larger than (\ref{pre-strategy space G})
and (\ref{pre-strategy space H}) using   
\begin{eqnarray}
\mathcal{A}_{G}^{x}(c):=\bigcup\limits_{\delta\in[0,p(x)]}\bigcup\limits_{c\in\partial
^{G}_{\delta}H(x)} \{ y\in (0,\infty)\,\vert\,L^{x}(c,W_H(x)-\delta ) \leq y\leq R^{x}(c,W_H(x)-\delta)\} ,  \label{strategy G} \\
\mathcal{A}_{x}^{H}(c) :=\bigcup\limits_{\varepsilon\in[0,p(x)]}\bigcup
\limits_{c\in \partial^{H}_{\varepsilon}G(x)} \{ y\in(0,\infty)\,\vert
\,L^{x}(c,W^G(x)+\varepsilon) \leq y\leq R^{x}(c,W^G(x)+\varepsilon ) \} .
\label{strategy H}
\end{eqnarray}
These admissible neighborhoods are illustrated in Figure \ref{Figure neighborhoods 2}. These
neighborhoods coincide with those used in \cite{Peskir} and are used in the next result which is 
an analytical complement to \cite{Peskir} Theorem 4.1.

\begin{theorem}
\label{Theorem duality}
Suppose that (\ref{stuck together}) holds, then the functions $x \mapsto (W^G)_H^{\ast \ast}(x)$
and $x \mapsto (W_H)_{\ast\ast}^G(x)$ defined in (\ref{pre-concave biconjugate H}) and 
(\ref{pre-convex biconjugate G}) satisfy
\begin{eqnarray}
(W^G)_H^{\ast \ast}(x) =\inf_{z\in \mathbb{R}}\sup_{y\in \mathcal{A}^{x}_G(z)}
\left(z(x-y) +W^G(y)\right) ,  \label{concave biconjugate H} \\
(W_H)_{\ast\ast}^G(x) =\sup_{z\in \mathbb{R}}\inf_{y\in \mathcal{A}^{H}_{x}(z)}
\left( z(x-y) +W_H(y)\right) .  \label{convex biconjugate G} 
\end{eqnarray}
for all $x>0$. Moreover, $(W^G)_H^{\ast \ast}(x)=(W_H)_{\ast\ast}^G(x)$ for all $x> 0$ and 
\begin{eqnarray*}
\left\{x\in(0,\infty)\,\left\vert\, (W^G)_H^{\ast\ast}(x) =W^G(x)\right.\right\} 
=\left\{ x\in(0,\infty)\,\left\vert\,\partial^{H}G(x) \neq \emptyset \right. \right\} , \notag \\
\left\{ x\in(0,\infty)\,\left\vert\, (W^G)^{\ast\ast}_H(x) =W_H(x)\right.\right\} 
=\left\{ x\in(0,\infty)\,\left\vert\,\partial^{G}H(x) \neq \emptyset \right. \right\} \notag .
\end{eqnarray*}
\end{theorem}

\begin{proof}
Fix $x> 0$ and $c\in \mathbb{R}$ and define
\begin{equation*}
(W^G)_H^{\ast}(c) := \inf_{y\in \mathcal{A}_x^H(c)}(cy-W^G(y)) . \label{duality conjugate}
\end{equation*}
By definition, the line $f(y):=W^G(x)+\varepsilon_c^{\ast}(x)+c(y-x)$ dominates $y\mapsto
W^G(y)$ on $\widehat{\mathcal{A}}_x^H(c)=[l_H^x(c,W^G(x)),r_H^x(c,W^G(x))] \subseteq 
\mathcal{A}_x^H(c)$. If there exists $z\in \mathcal{A}_x^H(c) \setminus \widehat{\mathcal{A}}
_x^H(c)$ such that $f(z)=W^G(z)$ then either $c\in \partial^H G(z)$ or $z$ is on the boundary 
of $\mathcal{A}_x^H(c)$, i.e. $z=L_H^x(c,W^G(x))$ or $z=R_H^x(c,W^G(x))$. Thus we may 
conclude that $(W^G)_H^{\ast}(c)=cx^{\prime}-W^G(x^{\prime})$ for some $x^{\prime}$
such that 
\begin{equation*}
c(y-x^{\prime})+W^G(x^{\prime}) = W^G(x)+\varepsilon_c^{\ast}(x). 
\end{equation*}
Hence 
\begin{equation*}
\inf_{z\in \mathbb{R}}\sup_{y\in \mathcal{A}^{x}_{G}(z)}\left(z(x-y) +W^G(y)\right)
=\inf_{z\in \mathbb{R}}(W^G(x)-\varepsilon_z^{\ast}(x))
\end{equation*}
so it follows from the definition of $\varepsilon^{\ast}(x)$ and Proposition 
\ref{proposition supsup} that 
\begin{equation*}
\inf_{z\in \mathbb{R}}\sup_{y\in \mathcal{A}^{x}_{G}(z)}\left(z(x-y) +W^G(y)\right)
=W^G(x)+\varepsilon^{\ast}(x)= (W^G)_H^{\ast\ast}(x) .
\end{equation*}
A symmetric argument can be used to show that 
\begin{equation*}
\inf_{z\in \mathbb{R}}\sup_{y\in \mathcal{A}^{x}_{G}(z)}\left(z(x-y) +W^G(y)\right)
=W_H(x)-\delta^{\ast}(x)= (W_H)^G_{\ast\ast}(x) .
\end{equation*}
It follows from Lemma \ref{lemma no gap} that $(W^G)_H^{\ast \ast}(x)=(W_H)^G_{\ast \ast}(x)$ and
the final statement follows from the definition of $\varepsilon^{\ast}(x)$ and $\delta^{\ast}(x)$.
\medskip
\end{proof}

Theorem \ref{Theorem duality} can be reformulated in terms of the $F$-concavity (resp. $F$-convexity) of 
the functions $G$ and $H$ determining $W^G$ and $W_H$. To see this observe that it follows from the definitions of 
$y \mapsto W^G(y)$ and $y\mapsto W_H(y)$ in (\ref{W^G function}) and (\ref{W_H function}) that
\begin{align}
(G)_H^{\ast\ast}(x) &:= (W^G)_H^{\ast \ast}(F(x))\psi(x) =\inf_{z\in \mathbb{R}}\sup_{F(y)\in \mathcal{A}^{x}_G(z)}
\left(z(F(x)-F(y)) + \left(\frac{G}{\psi}\right)(y)\right)\psi(x) , \label{F-concave H}  \\
(H)_{\ast\ast}^G(x) &:= (W_H)_{\ast\ast}^G(F(x))\psi(x) =\sup_{z\in \mathbb{R}}\inf_{F(y)\in \mathcal{A}^{H}_{x}(z)}
\left( z(F(x)-F(y)) +\left(\frac{H}{\psi}\right)(y)\right)\psi(x) \label{F-convex G},
\end{align}
which are the modified $F$-concave/$F$-convex biconjugates used in \cite{Peskir}. It is shown in the next section that (\ref{F-concave H})-(\ref{F-convex G}) are $r$-semiharmonic functions which solve the dual problems (\ref{dual problems}).

We began constructing an extension of the convex/concave biconjugates of $W^G$ and $W_H$ due to the fact that $G$ is $r$-superharmonic with respect to the diffusion $X$ if and only if $G/\psi$ is $F$-concave or equivalently $W^G$ is concave. For studying optimal stopping games, we shall also use that $G$ is $r$-superharmonic with respect to the diffusion $X$ if and only if $G/\varphi$ is $\widetilde{F}$-concave or equivalently $(G/\varphi)\circ (-\widetilde{F})^{-1}$ is concave on $(0,\infty)$. With this in mind,  Theorem (\ref{Theorem duality}) can be formulated with respect to the other ratio of the fundamental solutions.

\begin{corollary}
\label{corollary general case}
Let $\widetilde{W}^G:=(G/\varphi)\circ (-\widetilde{F})^{-1}$ and $\widetilde{W}_H:=(H/\varphi)\circ (-\widetilde{F})^{-1}$. Assuming that (\ref{integ assumption})-(\ref{boundary assumption}) hold. Then for all $x>0$
\begin{equation*}
(\widetilde{W}^G)_H^{\ast \ast}(x) :=\inf_{z\in \mathbb{R}}\sup_{y\in \widetilde{\mathcal{A}}^{x}_G(z)}
\left(z(x-y) +\widetilde{W}^G(y)\right) = \sup_{z\in \mathbb{R}}\inf_{y\in \widetilde{\mathcal{A}}^{H}_{x}(z)}
\left( z(x-y) +\widetilde{W}_H(y)\right) =:(\widetilde{W}_H)_{\ast\ast}^G(x).
\end{equation*}
where the sets 
$\widetilde{\mathcal{A}}^{x}_G(z)$, $\widetilde{\mathcal{A}}^{H}_{x}(z)$ are defined as in (\ref{strategy G}) and (\ref{strategy H}) replacing $W^G$, $W_H$ with $\widetilde{W}^G$, $\widetilde{W}_G$. 
\end{corollary}

\begin{proof}
In this section we have not used any properties of the diffusion $X$, nor the specific form of the functions $W^G,W_H$. Consequently, the previous Theorem holds for any two continuous functions $g,h: (0,\infty) \rightarrow \mathbb{R}$ such that $g\leq h$ and $\lim_{x\downarrow 0}g(x)-h(x)=\lim_{x\uparrow \infty}g(x)-h(x)=0$. In, particular, replacing $W^G$, $W_H$ with $\widetilde{W}^G$, $\widetilde{W}_G$ does not alter the result. 
\end{proof}

\medskip

It was observed in Remark \ref{Remark duality} that the assumption (\ref{boundary assumption}) may be relaxed by extending the functions $W^G$, $W_H$ onto $\mathbb{R}_+$ in a manner which ensures that $W^G(x)+\varepsilon^{\ast}(x)=W_H(x)-\delta^{\ast}(x)$ for all $x>0$. This approach can be used to relax the assumptions of Theorem \ref{Theorem duality}.

\begin{corollary}
\label{corollary extend}
Take $w_0 \in [W^G(0+),W_H(0+)]$ and extend the functions $W^G$, $W_H$ onto $\mathbb{R}_+$ by setting 
$W^G(0)=[W^G(0+),w_0]$ and $W_H(0)=[w_0,W_H(0+)]$. Suppose that the intervals in the definitions (\ref{delta-sub})-(\ref{epsilon-super}) are take to be closed (as apposed to open) then
\begin{equation*}
(W^G)_H^{\ast \ast}(x) =\inf_{z\in \mathbb{R}}\sup_{y\in \mathcal{A}^{x}_G(z)}
\left(z(x-y) +W^G(y)\right) = \sup_{z\in \mathbb{R}}\inf_{y\in \mathcal{A}^{H}_{x}(z)}
\left( z(x-y) +W_H(y)\right) = (W_H)_{\ast\ast}^G(x)
\end{equation*}
for all $x\geq 0$.
\end{corollary}

\begin{proof}
Under these assumptions, Remark \ref{Remark duality} shows that Lemma \ref{lemma no gap} holds. Furthermore, under these assumptions the functions $(W^G)^{\ast\ast}_H$ and $(W_H)_{\ast\ast}^G$ defined in (\ref{pre-convex biconjugate G})-(\ref{pre-concave biconjugate H}) are defined on $\mathbb{R}_+$ and $(W^G)^{\ast\ast}_H(0)=(W_H)_{\ast\ast}^G(0)=w_0$. Thus with this extended definition, the corollary follows from Theorem \ref{Theorem duality}.
\end{proof}

\begin{remark}
The function $y \mapsto (W^G)_H^{\ast \ast}(y)$ can be viewed as the `concave biconjugate of $W^G$ in the presence of $y\mapsto W_H(y)$' (or constrained to remain within $\mathrm{cl}(\mathbb{R}^2\setminus \mathrm{epi}(W_H))$) for the following three reasons: 
\begin{enumerate}

\item Theorem \ref{Theorem duality} illustrates that the function $y \mapsto (W^G)_H^{\ast \ast}(y)$ can be expressed as 
\begin{equation*}
(W^G)_H^{\ast \ast}(x) = \inf_{z\in \mathbb{R}}\left( zx -\inf_{y\in \mathbb{R}} \left( zy -W^G(y)+
\delta(y\,\vert\,\mathbb{R}_+\setminus\mathcal{A}^{x}_G(z))\right)\right) , \label{biconjugate again}
\end{equation*} 
where $\delta(y\,\vert\, \mathbb{R}_+\setminus\mathcal{A}^{x}_G(z))$ is the `characteristic function' of the set $\mathbb{R}_+\setminus\mathcal{A}^{x}_G(z)$, 
i.e. $\delta(y\,\vert\,\mathbb{R}_+\setminus\mathcal{A}^{x}_G(z))=0$ if $y\in \mathcal{A}^{x}_G(z)$ and 
$\delta(y\,\vert\,\mathbb{R}_+\setminus\mathcal{A}^{x}_G(z))=\infty$ otherwise. The difference with the standard concave biconjugate
is that this characteristic function depends on the choice of $z$ and $x$. However, this dependence seems essential to ensure that $\mathrm{epi}(W_H)\subseteq \mathrm{epi}((W^G)_H^{\ast \ast})$.

\item  For any measurable function $f:(0,\infty) \rightarrow \mathbb{R}$, the superdifferential of $f$ is as defined in (\ref{superdiff}) and the $\varepsilon$-superdifferential is defined as
\begin{equation*}
\partial_{\varepsilon}f(x)= \{ c\in\mathbb{R}\,\vert\, f(y)-f(x)-\varepsilon \geq c(y-x) \; \forall y\in (0,\infty) \}.
\end{equation*}
Hence $f_{\ast\ast}(x)-f(x)=\inf\{\varepsilon\geq 0\,\vert\,\partial_{\varepsilon}f(x)\neq \emptyset\}$.
This definition of a $\varepsilon$-superdifferential coincides with the definition (\ref{epsilon-super}) 
when $W_H$ is taken to be the multivalued function 
\begin{equation*}
W_{H}(y)=\left\{
\begin{array}{cc}
\left[ W^G(0),+\infty \right] & y=0, \\
\infty & y\in (0,+\infty), \\
\left[ W^G(\infty),+\infty \right] & y=+\infty .
\end{array}
\right. 
\end{equation*}
Moreover, when (\ref{stuck together}) holds, it follows from Theorem \ref{Theorem duality} that 
$(W^G)_H^{\ast \ast}(x)-W^G(x)=\varepsilon^{\ast}(x)$ (where $\varepsilon^{\ast}(x)$ is as defined in 
(\ref{epsilon^ast})) so $y \mapsto (W^G)_H^{\ast \ast}(y)$ and $y \mapsto f_{\ast\ast}(y)$ can both be 
characterised in terms of a smallest spike variation. 

\item In Corollary \ref{corollary construction} it was shown that the concave biconjugate of $y \mapsto (G/\psi) 
\circ F^{-1}(y)$ is such that $((G/\psi) \circ F^{-1})_{\ast\ast}(x)=\sup A_1(x)$ where the set $A_1(x)$ is defined
in (\ref{set A1}). This characterisation of the concave biconjugate uses the same method of construction as
is used in Proposition \ref{proposition supsup} to define $y \mapsto (W^G)_H^{\ast \ast}(y)$.

\end{enumerate}
The function $y \mapsto (W_H)^G_{\ast\ast}(y)$ will be referred to as the `convex biconjugate of $W_G$
in the presence of $y\mapsto W^H(y)$ (or constrained to remain within $\mathrm{epi}(W^G)$) for directly similar reasons. 
\end{remark}

\section{Nash equilibrium in optimal stopping games}
\label{section games}

In this section Theorem \ref{Theorem duality} is applied to show that
the infinite time horizon optimal stopping game (\ref{lower value})-(\ref{upper value}) exhibits both a Stackelberg and a Nash equilibrium. The first Theorem of this Section shows that under the assumptions (\ref{integ assumption}) and (\ref{boundary assumption}) upper and lower values of the optimal stopping game (\ref{lower value}) and (\ref{upper value}) satisfy $\underline{V}\geq (G)_H^{\ast\ast}$ and $\overline{V}\leq (H)^G_{\ast\ast}$ respectively. Consequently, it follows from Theorem \ref{Theorem duality} that the game has a value, namely $\overline{V}=\underline{V}=(G)^{\ast\ast}_H$ hence the game has a Stackelberg equilibrium. 
In fact the game has a value without the need to assume the boundary condition (\ref{boundary assumption}) holds, but we need to account for the fact that when both minimising and maximising agent select the same stopping time $R_{x}(\tau,\tau) = E_{x}[ e^{-r\tau} G( X_{\tau})]$. 

\begin{theorem}
\label{Theorem Game has value}Suppose that (\ref{integ assumption}) holds and consider the optimal stopping game (\ref{lower value})-(\ref{upper value}). The optimal
stopping game has a Stackelberg equilibrium , i.e. the games has a value $V$ which can be represented as
\begin{equation*}
V(x)=\left( G\right) _{H}^{\ast \ast }(x)=(W^G)^{\ast\ast}_H(F(x))\psi(x)
\end{equation*}
for all $x\in I$. Moreover, when (\ref{boundary assumption}) hold, it follows that 
$V=\left( G\right) _{H}^{\ast \ast }=\left( H\right) _{\ast \ast }^{G}$ without the need to extend $W^G$ and $W_H$ onto $\mathbb{R}_+$.
\end{theorem}

\begin{proof}
To accomodate the asymmetry in the objective function discussed above, extend the functions $W^G$, $W_H$ onto $\mathbb{R}_+$ by setting $W^G(0):=W^G(0+)$ and $W_H(0)=[W^G(0+),W_H(0+)]$. As noted in Corollary \ref{corollary extend}, these extended definitions of $W^G$, $W_H$ ensure that $(W_H)^{G}_{\ast\ast}(0)=
(W^G)_{H}^{\ast\ast}(0)=W^G(0+)=G(a+)/\varphi(a+)$. Let $F(a):=F(a+)=0$ and $\varphi(a):=\varphi(a+)=+\infty$ then
\begin{equation*}
(H)^G_{\ast\ast}(x):=(W_H)^{G}_{\ast\ast}(F(x))\varphi(x) \qquad , \qquad
(G)_H^{\ast\ast}(x):=(W^G)_{H}^{\ast\ast}(F(x))\varphi(x) ,
\end{equation*}
for all $x\in (a,b)$ and $(H)^G_{\ast\ast}(a)=(G)_H^{\ast\ast}(a)=G(a+)$. It follows from Corollary \ref{corollary extend} (or Theorem \ref{Theorem duality} if (\ref{boundary assumption}) holds) that it is sufficient to show that
\begin{equation}
(H) _{\ast \ast }^{G} \leq \underline{V}\leq \overline{V}\leq
\left( G\right) _{H}^{\ast \ast } \label{step 3.15}
\end{equation}
for all $x\in I$ where $\underline{V}$ and $\overline{V}$ were defined in (\ref{lower value})
and (\ref{upper value}) respectively. The second inequality follows from the definitions of $\underline{V}$ and $\overline{V}$.

For the first inequality in (\ref{step 3.15}) let $\widehat{\mathcal{A}}%
_{x}=\bigcup_{z\in \mathbb{Z}}\widehat{\mathcal{A}}_{G}^{x}\left( z\right) $
and $\widehat{\tau}=\inf\{ \left. t\geq 0\,\right\vert
\,X_{t}\notin \widehat{\mathcal{A}}_{x}\} $. Let $a'=\inf\widehat{\mathcal{A}}_x$ 
and $b'=\sup\widehat{\mathcal{A}}_x$ and suppose that $b'<b$. It follows that 
\begin{equation*}
\underline{V}\left( x\right) \geq \inf_{\sigma \leq \widehat{\tau}}E_{x}\left[ e^{-r\sigma} \widehat{H}\left( X_{\sigma }\right) \right] =:V^{-}\left( x\right) .
\end{equation*}
where $\widehat{H}\left( y\right):=H\left( y\right)\mathbb{I}_{\left(
a^{\prime },b^{\prime }\right) }+G\left( y\right) \mathbb{I}_{%
\left[ y=a^{\prime }\right] \cup \left[ y=b^{\prime }\right] }$.
The stopping time which attains this infimum is of the form $\widehat{\sigma}=T_{a^{\ast}}\wedge
T_{b^{\ast }}$ for some $a^{\prime }\leq a^{\ast }\leq x\leq b^{\ast
}\leq b^{\prime}$. As we have assumed that $\sup\widehat{\mathcal{A}%
}_{x}<b$ the process 
$\left( e^{-r\left( t\wedge \widehat{\sigma }\right) }\varphi \left(
X_{t\wedge \widehat{\sigma }}\right) \right) _{t\geq 0}$ is a $\left( P_{x},%
\mathcal{F}_{t\wedge \widehat{\sigma }}\right)$-martingale. 
The optional sampling theorem implies that $E_{y}\left[ e^{-r\sigma} 
\varphi \left( X_{\sigma }\right) \right] =\varphi \left( y\right) $ for all $y\in \widehat{\mathcal{A}}_x$ 
and all $\sigma \leq \widehat{\tau}$. Hence for arbitrary $c\in \mathbb{R}$,
\begin{equation}
V^{-}\left( x\right) =\inf_{\sigma \leq \widehat{\tau}}E_{x}\left[ e^{-r\sigma } \widehat{H}\left(X_{\sigma }\right) \right] =c\varphi \left( x\right)
+\inf_{\sigma \leq \widehat{\tau}}E_{x}\left[ e^{-r\sigma}( \widehat{H}\left( X_{\sigma }\right)-c\varphi \left( X_{\sigma }\right) ) \right] .
\label{step 3.6}
\end{equation}
Since (\ref{step 3.6}) holds for all $c\in \mathbb{R}$, it follows that
\begin{equation}
V^{-}\left( x\right) =\sup_{c\in \mathbb{R}}\left( c\varphi \left( x\right)
+\inf_{\sigma \leq \widehat{\tau}}E_{x}\left[ e^{-r\sigma}( \widehat{H}\left( X_{\sigma }\right)-c\varphi \left( X_{\sigma }\right) ) \right] \right) .
\label{step 3.7}
\end{equation}
The optimal stopping problem on the right hand side of (\ref{step 3.7}) can
be expressed as
\begin{eqnarray*}
R_{c}\left( x\right)  &:=&\inf_{\sigma \leq \widehat{\tau}}E_{x}\left[ e^{-r\sigma }
( \widehat{H}\left( X_{\sigma } ) -c\varphi \left( X_{\sigma }\right) \right) \right]  \\
&=&\inf_{a^{\prime }\leq y\leq x\leq z\leq b^{\prime}}\left(
\frac{\widehat{H}\left( y\right) -c\varphi \left( y\right) }{\psi \left(
y\right) }\frac{F\left( z\right) -F\left( x\right) }{F\left( z\right)
-F\left( y\right) }+\frac{\widehat{H}\left( z\right) -c\varphi \left(
z\right) }{\psi \left( z\right) }\frac{F\left( x\right) -F\left( y\right) }{%
F\left( z\right) -F\left( y\right) }\right) \psi \left( x\right)  \\
&\geq &\inf_{y\in \widehat{\mathcal{A}}_{x}}\left( \frac{\widehat{H}%
\left( y\right) -c\varphi \left( y\right) }{\psi \left( y\right) }\right)
\psi \left( x\right) .
\end{eqnarray*}
Hence, from (\ref{step 3.7}) we obtain
\begin{equation*}
V^{-}\left( x\right) \geq \sup_{c\in \mathbb{R}}\left( c\frac{\varphi \left( x\right)}{\psi(x)} +\inf_{y\in
\widehat{\mathcal{A}}_{x}}\left( \frac{\widehat{H}\left( y\right)
-c\varphi \left( y\right) }{\psi \left( y\right) }\right)  \right)\psi(x). 
\end{equation*}
Let
\begin{equation*}
\underline{c}(x)=\inf\{c\in\ \mathbb{R}\,\vert\,\inf\widehat{\mathcal{A}}_x = \inf\widehat{\mathcal{A}}_G^x(c)\}
\quad , \quad
\overline{c}(x)=\sup\{c\in\ \mathbb{R}\,\vert\,\sup\widehat{\mathcal{A}}_x = \sup\widehat{\mathcal{A}}_G^x(c)\} .
\end{equation*}
Recall that $\delta_c^{\ast}(x)$ and $\delta^{\ast}(x)$ were defined in (\ref{delta-star}) and (\ref{epsilon^ast}). 
Due to the definition of $\widehat{\mathcal{A}}_{x}$ it follows that for all $c\in \mathbb{R}$
\begin{equation*}
cF(x)+\inf_{y\in \widehat{\mathcal{A}}_{x}}\left( \frac{\widehat{H}(
y) -c\varphi(y) }{\psi(y) }\right) \leq \frac{H(x)}{\psi(x)} - \delta_c^{\ast}(F(x)) .
\end{equation*}
and equality holds for all $c\in [\underline{c}(x),\overline{c}(x)]$. 
In particular, it follows from the definition of $\widehat{\mathcal{A}}_x$ 
that $\delta_c^{\ast}(F(x))=\delta^{\ast}(F(x))$ for some $c\in [\underline{c}(x),\overline{c}(x)]$
so we may conclude that
\begin{equation*}
V^{-}(x)  \geq H(x)-\delta^{\ast }(F(x))\psi(x) = (W_{H})^G_{\ast \ast }(F(x))\psi(x) = (H)^G_{\ast \ast }(x)
\end{equation*}
which is the first inequality in (\ref{step 3.15}). The case that $b'=b$ can be handled by 
using an approximating sequence of domains as in Theorem \ref{Theorem natural boundaries}.

For the final inequality in (\ref{step 3.15}), let $\mathcal{A}_{x}=
\bigcup_{z\in \mathbb{Z}}\widehat{\mathcal{A}}_{x}^{H}\left( z\right)$
and let $\widehat{\sigma}=\inf \left\{ \left. t\geq 0\,\right\vert
\,X_{t}\notin \mathcal{A}_{x}\right\} $. Let $a'=\inf\mathcal{A}_x$ 
and $b'=\sup\mathcal{A}_x$ and suppose that $b'<b$.
There is a natural asymmetry in the 
value function defined in (\ref{game objective function}) as 
$R_{x}( \tau ,\tau) = E_{x}[ e^{-r\tau} G( X_{\tau})]$. Thus
\begin{eqnarray*}
\overline{V}\left( x\right)  &=&\inf_{\sigma }\sup_{\tau }E_{x}\left[ 
e^{-r\left( \tau \wedge \sigma \right) } \left(
G\left( X_{\tau }\right) \mathbb{I}_{\left[ \tau \leq \sigma \right]
}+H\left( X_{\sigma }\right) \mathbb{I}_{\left[ \tau >\sigma \right]
}\right) \right]  \\ &\leq &\sup_{\tau \leq \widehat{\sigma}}E_{x}\left[ e^{-r\tau }
G\left( X_{\tau}\right) \right] =:V^{+}\left( x\right) .
\end{eqnarray*}
A similar argument to that used above can be used to show that
%\begin{equation*}
%V^{+}(x)  \leq \inf_{c\in \mathbb{R}}\left( c\frac{\varphi \left( x\right)}{\psi(x)} +\sup_{y\in
%\mathcal{A}_{x}}\left( \frac{G(y)-c\varphi(y) }{\psi(y) }\right)  \right)\psi(x).
%\end{equation*}
%when $b<b'$. Let
%\begin{equation*}
%\underline{c}(x)=\inf\{c\in\ \mathbb{R}\,\vert\,\sup\mathcal{A}_x = \sup\widehat{\mathcal{A}}^H_x(c)\}
%\quad , \quad
%\overline{c}(x)=\sup\{c\in\ \mathbb{R}\,\vert\,\inf\mathcal{A}_x = \inf\widehat{\mathcal{A}}^H_x(c)\} .
%\end{equation*}
%Define $\varepsilon_c^{\ast}(x)$ and $\varepsilon^{\ast}(x)$ as in (\ref{epsilon-star}) and (\ref{epsilon^ast}) with $H$ replaced by $\widetilde{H}$ where necessary. Due to the definition of $\mathcal{A}_{x}$ it follows that for all $c\in \mathbb{R}$
%\begin{equation*}
%cF(x)+\sup_{y\in\mathcal{A}_{x}}\left( \frac{G(y) -c\varphi(y) }{\psi(y) }\right) \geq \frac{G(x)}{\psi(x)} + \varepsilon_c^{\ast}(F(x)) .
%\end{equation*}
%and equality holds for for all $c\in [\underline{c}(x),\overline{c}(x)]$. In particular, it follows from the definition of $\mathcal{A}_x$ 
%that $\varepsilon_c^{\ast}(F(x))=\varepsilon^{\ast}(F(x))$ for some $c\in [\underline{c}(x),\overline{c}(x)]$
%so we may conclude that
\begin{equation*}
V^{+}(x)  \leq G(x)-\varepsilon^{\ast}(F(x))\psi(x) = (W^G)_{H}^{\ast \ast }(F(x))\psi(x) = (G)^{H}_{\ast \ast }(x). 
\end{equation*}
which is the final inequality in (\ref{step 3.15}). The case that $b^{\prime}=b$ can be handled by using an approximating sequence of domains as in Theorem \ref{Theorem  natural boundaries}.

The extension of $W^G$ and $W_H$ onto $\mathbb{R}_+$ mentioned at the beginning of proof is only necessary to exclude the case that $W^G(0+)=(W^G)_H^{\ast\ast}(0+)<(W_H)^G_{\ast\ast}(0+)=W_H(0+)$. 
When $(W^G)_H^{\ast\ast}(0+)<(W_H)^G_{\ast\ast}(0+)$ it follows from the definitions (\ref{F-concave H})-(\ref{F-convex G}) that $(G)_H^{\ast\ast}(a+)<(H)^G_{\ast\ast}(a+)$. However, when (\ref{boundary assumption}) holds $W^G(0+)=W_H(0+)$ so this case can not occur.
\end{proof}

\medskip

\begin{remark}
\label{remark integ}
In the previous Theorem the assumption (\ref{integ assumption}) has only been used
to rule out the degenerate case that $\limsup_{y\downarrow a}(G/\psi)(y)=+\infty$ 
(resp. $\limsup_{y\uparrow b}(G/\psi)(y)=+\infty$) as in this case the optimal stopping 
time for the maximising agent is not finite and $\overline{V}(x)=\underline{V}(x)=+\infty$
for all $x \in I$. In this case the functions $(G)_H^{\ast\ast}$ and $(H)^G_{\ast\ast}$
are no longer related to the value function of the optimal stopping game.
\end{remark}

In Theorem \ref{Theorem Game has value} the functions $H$ and $G$ are defined on the entire 
domain of the one dimensional diffusion $X$. Theorem \ref{Theorem Game has value} can handle 
the case that $X$ is absorbed upon exit from a set $[c,d]\in I$ by setting $G(x)=H(x)$ for all 
$x \in I\setminus [c,d]$. The value function in Theorem \ref{Theorem Game has value} can be thought of as an elastic cord which is tied to $\limsup_{y\downarrow 0} W^G(y)$ and pulled towards infinity between the two obstacles $W^G$ and $W_H$. On a compact domain it is the
shortest path between these two obstacles as shown in \cite{Peskir}. 

The next two lemmata examine some properties of the value of the game identified in Theorem \ref{Theorem Game has value} which will be used to prove that the game with upper- and lower-value (\ref{lower value})-(\ref{upper value}) has a Nash equilibrium. The first lemma show
that the value of the game is $r$-semiharmonic, that is $\left( G\right) _{H}^{\ast \ast }\in \mathrm{Sup}\left[
G,H\right) \cap \mathrm{Sub}\left( G,H\right] $ where
\begin{eqnarray*}
\mathrm{Sup}\left[ G,H\right) &=&\left\{ \left. U:I\rightarrow \left[ G,H%
\right] \,\right\vert \,U\text{ \textrm{is continuous and }}r\text{\textrm{-superharmonic on }}\left\{U<H\right\} \right\} , \\
\mathrm{Sub}\left( G,H\right] &=&\left\{ \left. U:I\rightarrow \left[ G,H\right] \,\right\vert \,U\text{ \textrm{is continuous and }}r\text{\textrm{-subharmonic on }}\left\{
U>G\right\} \right\} ,
\end{eqnarray*}
are the admissible sets in the dual problems defined in (\ref{dual problems}).
The second lemma shows that the $r$-harmonic functions $\left( G\right) _{H}^{\ast \ast }$ and $\left( H\right)_{\ast \ast }^{G}$ solve the dual problems (\ref{dual problems}).

\begin{lemma}
\label{lemma sets}The concave biconjugate of $W^G$ in the presence of the
upper barrier $W_H$ defined in (\ref{concave biconjugate H}) is such that the associated function $(G)_H^{\ast\ast}$ defined in (\ref{F-concave H}) is in the
admissible sets for the dual problems, i.e.
\begin{equation}
\left( G\right) _{H}^{\ast \ast }\in \mathrm{Sup}\left[ G,H\right) \cap
\mathrm{Sub}\left( G,H\right] .  \label{step 3.18}
\end{equation}
Moreover, the convex biconjugate of $W_H$ in the presence of the
lower barrier $W_G$ defined in (\ref{convex biconjugate G}) is such that the associated function $(H)^G_{\ast\ast}$ defined in (\ref{F-convex G}) is in the
admissible sets for the dual problems, i.e.
\begin{equation*}
\left( H\right) ^{G}_{\ast \ast }\in \mathrm{Sup}\left[ G,H\right) \cap
\mathrm{Sub}\left( G,H\right] .
\end{equation*}
\end{lemma}

\begin{proof}
First we show that $\left( G\right) _{H}^{\ast \ast }\in \mathrm{%
Sup} \left[ G,H\right)$. Take $x\in I$ such that $(G)_{H}^{\ast \ast}(x) <H(x)$ and let
\begin{equation}
z_{-}\left( x\right) :=\sup \left\{ \left. y\leq x\,\right\vert \,\left(
G\right) _{H}^{\ast \ast }\left( y\right) =H\left( y\right) \right\} \quad
;\quad z_{+}\left( x\right) :=\inf \left\{ \left. y\geq x\,\right\vert
\,\left( G\right) _{H}^{\ast \ast }\left( y\right) =H\left( y\right)
\right\} .
\label{define z}
\end{equation}
Let $\widetilde{G}(y):=G(y) \mathbb{I}_{(z_{-}(x)
,z_{+}(x))}+H(x)\mathbb{I}_{[y=z_{-}(x)] \cup [y=z_{+}(x)]}$
and define a function $\widetilde{W}$ using
\begin{equation*}
\widetilde{W}(y,x) :=\left(\frac{\widetilde{G}}{\psi}\right) \circ F^{-1}(y) .
\end{equation*}
For $z\in [F(z_-(x)),F(z_+(x))]=:J^x$ and $c\in \mathbb{R}$ let
\begin{equation*}
\widehat{\varepsilon}_{c}(z) :=\inf \left\{ \varepsilon \geq
0\,\left\vert \, \widetilde{W}(y,x)\leq W^G(z)+\varepsilon +c(y-z) \quad \forall y\in J^x\right. \right\} .
\end{equation*}
For fixed $z\in J^x$, $c\mapsto \widehat{\varepsilon}_{c}(z)$ this continuous. Define $\widehat{\varepsilon}(z):=\inf_{c\in \mathbb{R}}\widehat{\varepsilon}_{c}(z)$ which is attained at some $\widehat{c}\in \mathbb{R}$. We claim that $\widehat{\varepsilon }(z)=\varepsilon^{\ast}(z)$ for all $z\in J^x$ where $\varepsilon^{\ast }(z)$ is defined in (\ref{epsilon-star}). To
prove this first suppose that
\begin{equation}
\widehat{\varepsilon }(z) =\widehat{\varepsilon}_{\widehat{c}}(z) 
<\inf_{c\in \mathbb{R}}\varepsilon _{c}^{\ast }(z)=\varepsilon^*(z) .   \label{claim}
\end{equation}
By definition
\begin{equation*}
h(y,z) := W^G(z)+ \widehat{\varepsilon}(z)+\widehat{c}(y-z) \geq \widetilde{W}(y,x) \qquad \forall y\in J^x
\end{equation*}
and in particular $h(F(z_{-}(x)),F(x)) \psi(z_{-}(x)) \geq H(z_{-}(x))$ and
$h(F(z_{+}(x)),F(x)) \psi (z_{+}(x)) \geq H(z_{+}(x))$. Thus the continuity of $%
y\mapsto \left( H/\psi \right) \left( y\right) $ implies that for all $z \in J^x$
\begin{equation*}
[ l_{H}^{z}(\widehat{c},W^G(z) +\widehat{\varepsilon}(z)) ,r_{H}^{z}(\widehat{c},W^G(x)+
\widehat{\varepsilon}(z))] \subseteq [z_{-}(x),z_{+}(x)] =J^x   
\end{equation*}
hence $\widehat{c}\in \partial^{H}G_{\widehat{\varepsilon}(z)}(z)$ which 
contradicts (\ref{claim}) so $\widehat{\varepsilon }(z) \geq \varepsilon^{\ast }(z)$ for all $z \in J^x$.

Suppose that $\widehat{\varepsilon }(z)>\varepsilon^{\ast}(z)$ and take $c^{\prime}\in
\partial^{H}_{\varepsilon^{\ast}(z)}G(z)$ so that by assumption
\begin{equation*}
h(y):=W^G(z) + \varepsilon^{\ast}(z)+c^{\prime}(y-z) < W^G(y)
\end{equation*}
for some $y\in J^x \setminus [ l_{H}^{z}(\widehat{c},W^G(z) +\varepsilon^{\ast}(z)) ,r_{H}^{z}(\widehat{c},W^G(x)+
\varepsilon^{\ast}(z))]$. Let $\hat{c}= \inf\{c\in\mathbb{R}\,\vert\,c \in\partial^G H(F(z_+(x)))\}$
then either: (i)
\begin{equation*}
W_H(r_H^z(c^{\prime},W^G(z)+\varepsilon^{\ast}(z))) \geq W_H(F(z_+(x))) + 
\hat{c}(r_H^z(c^{\prime},W^G(z)+\varepsilon^{\ast}(z))-F(z_+(x)))
\end{equation*}
which implies that 
\begin{equation*}
h(y)>W_H(F(z_+(x))) + \hat{c}(y+\varepsilon^{\ast}(z))-F(z_+(x))) \geq W^G(y)
\end{equation*} 
for all $y \in [r_H^z(c^{\prime},W^G(z)+\varepsilon^{\ast}(z)),F(z_{+}(z))]$; or (ii) there exists $\hat{y}\in [r_H^z(c^{\prime},W^G(z)+\varepsilon^{\ast}(z)),F(z_{+}(z))]$ and $b\in \mathbb{R}$ such that 
\begin{equation*}
r_H^z(c^{\prime},W^G(z)+\varepsilon^{\ast}(z)) = l_H^{\hat{y}}(b,W^G(\hat{y}))
\end{equation*}
and hence $h(y)>W^G(\hat{y})+b(y-\hat{y})\geq W^G(y)$ for all $y \in [r_H^z(c^{\prime},W^G(z)+\varepsilon^{\ast}(z)),F(z_{+}(z))]$. 
These two cases are illustrated in Figure \ref{figure lemma 22}. We may conclude that $h(y)\geq W^G(y)$
for all $y\in [r_H^z(c^{\prime},W^G(z)+\varepsilon^{\ast}(z)),F(z_{+}(z))]$. A similar argument 
can be used to show that $h(y)\geq W^G(y)$ for all $y\in [F(z_{-}(z)),l_H^z(c^{\prime},W^G(z)+\varepsilon^{\ast}(z))]$. 
We may conclude that $\widehat{\varepsilon }(z)=\varepsilon^{\ast}(z)$ for all $z\in J^x$.

\begin{figure}[h!]
\begin{center}
\begin{tabular}{cc}
\includegraphics[width=7cm]{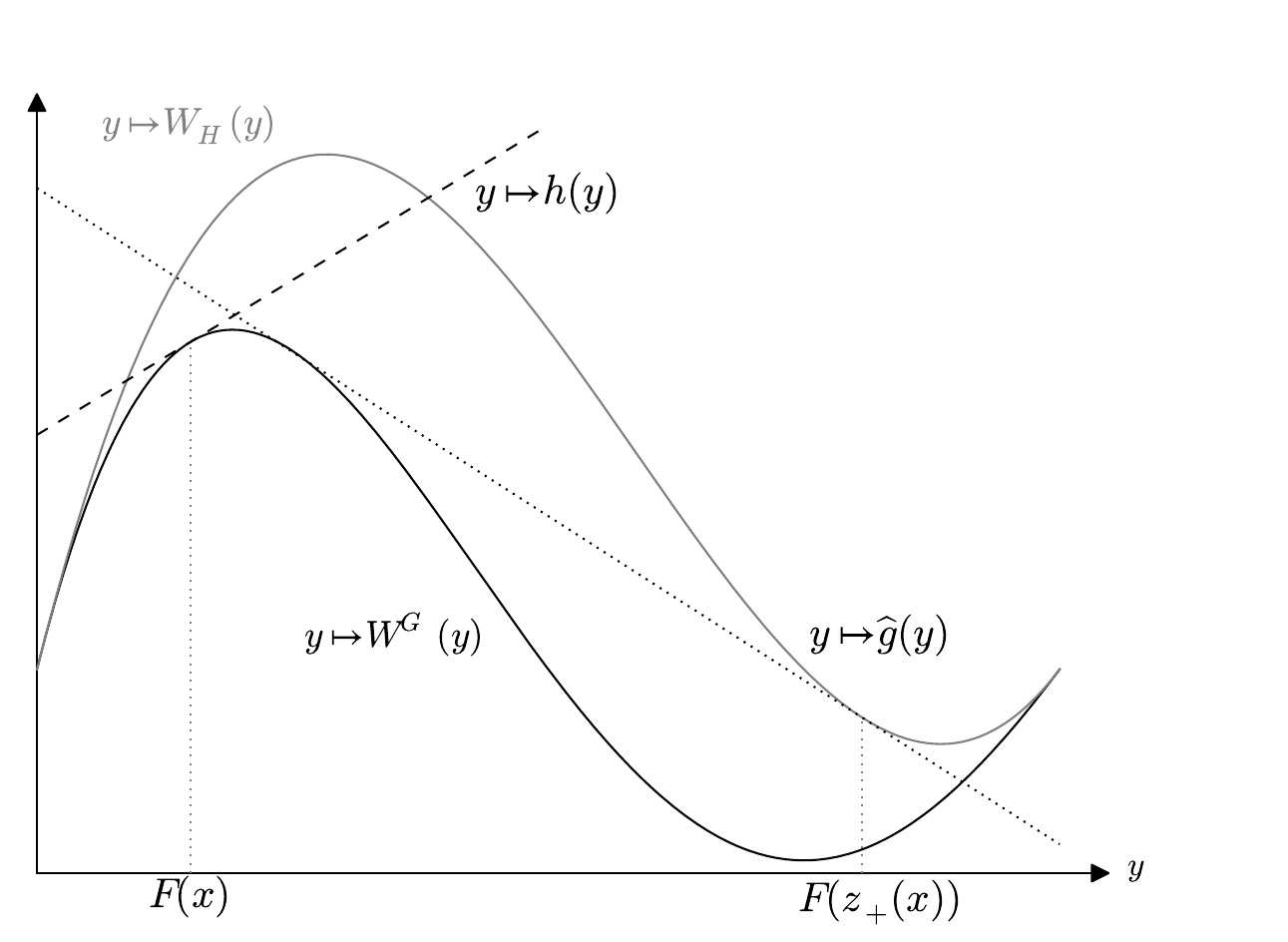} & \includegraphics[width=7cm]{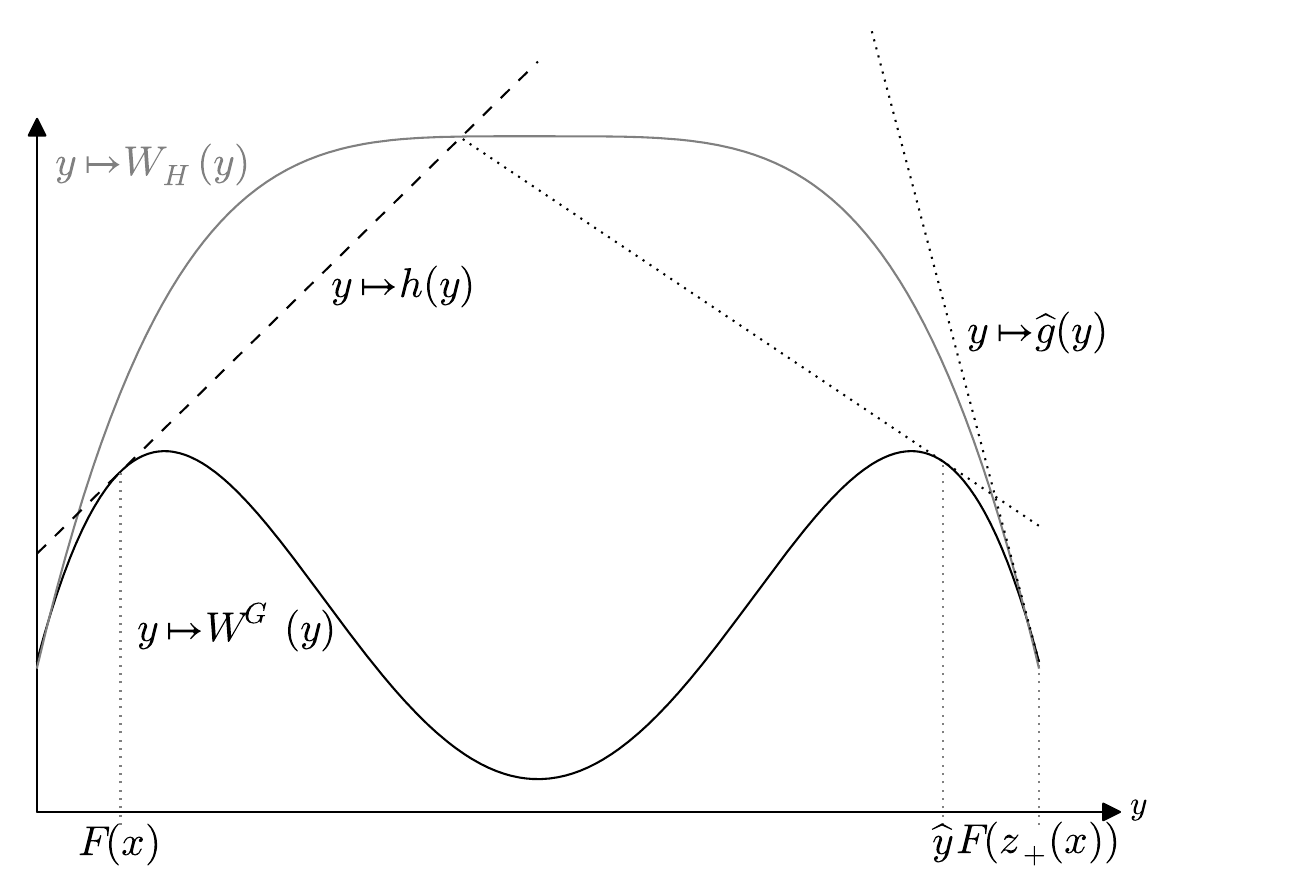}
\end{tabular} 

\begin{minipage}{0.7\textwidth}
\caption[Figure illustrating proof that $\widehat{\varepsilon}(x)=\varepsilon^{\ast}(x)$]
{In the figure on the left the line $h$ intercepts $W_H$ at a point above the line $\widehat{g}$ whereas in the the figure on the right this is not the case. However, in the figure on the right $h$ intercepts $W_H$ at the same point as the tangent to $W^G$ at $\widehat{y}$ intercepts $W_H$.
\label{figure lemma 22}} 
\end{minipage}
\end{center}
\end{figure}

Using Lemma \ref{lemma constructive}, for all $z\in J^x$ the concave biconjugate of $\widetilde{W}$ may be written as
\begin{eqnarray*}
\widetilde{W}_{\ast \ast}( z;x) 
&=& \inf_{c\in \mathbb{R}}\left( W^G(z) +\widehat{\varepsilon}_{c}(z)\right) \\
&=& W^G(z) +\varepsilon^{\ast}(z)=(W^G)_{H}^{\ast\ast}(z) .
\end{eqnarray*}
By definition $z\mapsto \widetilde{W}_{\ast \ast}( z;x) $ is concave on $J^x$ which
implies that $x\mapsto (G)_{H}^{\ast \ast}(x)$ is
$F$-concave on $[z_{-}(x),z_{+}(x)]$. Consequently, $(G)_{H}^{\ast \ast}$ 
is $F$-superharmonic on $[z_{-}(x),z_{+}(x)] $ and since $x$ was arbitrarily chosen 
$(G)_{H}^{\ast \ast }$ is $F$-superharmonic on each connected section of $\{
x\in I\,\vert\,(G)_{H}^{\ast \ast }(x)<H(x)\} $. The concave biconjugate $y\mapsto
\widetilde{W}_{\ast \ast }(y;x) $ is continuous on $J^x$ (see \cite{Rock} Theorem 10.1) and
\begin{equation*}
\widetilde{W}_{\ast \ast }(F(z_{-}(x);x)\psi( z_{-}(x)) = H(z_{-}(x)) \quad , \quad 
\widetilde{W}_{\ast \ast }(F(z_{+}(x);x)\psi( z_{+}(x)) = H(z_{+}(x)) 
\end{equation*}
so the continuity of $x\mapsto H(x)$ implies that $x\mapsto
(G)_{H}^{\ast \ast}(x) $ is continuous. We conclude that $(G)_{H}^{\ast \ast }\in \mathrm{Sup}\left[ G,H\right)$ as required.

Let $\widetilde{H}(x)=H(x)\mathbb{I}_{(a,b)}+G(x)
\mathbb{I}_{[x=a]\cup[x=b]}$ then by reversing the roles of $G$ and $H$ and replacing $H$ with $\widetilde{H}$
in the arguments above it can be shown that 
$(\widetilde{H})^{G}_{\ast \ast }\in \mathrm{Sub}( G,\widetilde{H}]
\subset \mathrm{Sub}\left( G,H\right]$ .
The first part of the lemma then follows from Theorem \ref{Theorem duality} as
$(\widetilde{H})^{G}_{\ast \ast }=\left( G\right) _{\widetilde{H}}^{\ast \ast }
=\left( G\right)_{H}^{\ast \ast }$.
The second part can be derived using a symmetric argument.
\end{proof}

\medskip
 
As $\left( G\right) _{H}^{\ast \ast }\in \mathrm{Sup}\left( G,H\right] $
from the definition of the dual problems $\hat{V}$ $\leq \left( G\right)
_{H}^{\ast \ast }$ and $\left( H\right) _{\ast \ast }^{G}\in \mathrm{Sub}%
\left[ G,H\right) $ implies $\left( H\right) _{\ast \ast }^{G}\leq \check{V}$. The next lemma shows that the reverse inequalities hold and in particular when (\ref{integ assumption}) and (\ref{boundary assumption}) hold the joint value of the dual problems is $\left(G\right) _{H}^{\ast\ast}$.

\begin{lemma}
\label{lemma coincides with dual}The concave biconjugate of $G$ in the
presence of the upper barrier $H$ defined in (\ref{F-concave H})
coincides with the value of one of the dual problems defined in (\ref{dual
problems}) and the convex biconjugate of $H$ in the
presence of the lower barrier $G$ defined in (\ref{F-convex G})
coincides with the other dual problem, i.e.
\begin{equation*}
\left( G\right)_{H}^{\ast \ast }=\hat{V} \qquad , \qquad
\left( H\right)^{G}_{\ast \ast }=\check{V}.
\end{equation*}
\end{lemma}

\begin{proof}
For the first statement it is sufficient to show that $\hat{V} \geq \left( G\right)_{H}^{\ast \ast }$. 
Fix $x'\in I$ such that $(G)_{H}^{\ast \ast }(x') <H(x')$ and suppose 
that $\hat{V}(x')<(G)_{H}^{\ast \ast }(x')=(W^G)_{H}^{\ast \ast }(F(x'))\psi(x')$. 
Define another set 
\begin{align*}
\widetilde{\mathrm{S}}\mathrm{up}[G,H) &:= \left\{ U:\mathbb{R}_+ \rightarrow [W^G,W_H]\,\vert\, U\,\mathrm{is}\,\mathrm{continuous}\,\mathrm{and}\,\mathrm{concave}\,\mathrm{on}\,\{U<W_H\}\right\}
\end{align*}
Let $x=F(x')$ then by assumption there exists $f \in \widetilde{\mathrm{S}}\mathrm{up}[G,H)$ 
such that $f(x)=W^G(x)+\varepsilon$ for some $\varepsilon <\varepsilon ^{\ast}(x)$ 
where $\varepsilon^{\ast}$ is as defined in (\ref{epsilon-star}). Let
\begin{equation*}
l_{f}(x) :=\sup \{ y\leq x\,\vert \,f(y) =W^G(y)\} \qquad ,\qquad  
r_{f}(x) :=\inf \{ y\geq x\,\vert \,f(y) =W^G(y)\}
\end{equation*}
so that the function $f$ is concave on $[ l_{f}(x),r_{f}(x)]$. A
tangent to $f$ at $x$ can be expressed as
\begin{equation*}
u(y;x) :=f(x)+c^{\prime}(y-x) \quad \mathrm{for} \quad c^{\prime}\in \left[ \frac{d^{+}}{dy}f(x),\frac{d^{-}}{dy}f(x)\right].
\end{equation*}
Consider
\begin{equation*}
l_{\eta}^{\prime}:= l_{G}^{x}(c^{\prime},f(x)+\eta) \vee l_{H}^{x}(c^{\prime},f(x)+\eta) \quad ,\quad
r_{\eta}^{\prime}:= r_{G}^{x}(c^{\prime},f(x)+\eta) \wedge r_{H}^{x}(c^{\prime},f(x)+\eta).
\end{equation*}
As $\partial^{H}f(x) \subseteq \partial^{H}_{\varepsilon}G(x) =\emptyset$ 
it follows that either $l_0^{\prime}=l_{G}^{x}(c^{\prime},f(x))$ and/or $r_0^{\prime}=
r_{G}^{x}(c^{\prime},f(x))$. Consequently, for sufficiently small $\eta >0$, either
$l_{\eta}^{\prime}=l_{G}^{x}(c^{\prime},f(x)+\eta)$ and/or $r_{\eta}^{\prime}=r_{G}^{x}(c^{\prime},f(x)+\eta)$.
Thus $f(y) =f(x) +\eta +c^{\prime}(y-x)$ for some $y\in [ l_{f}(x) ,r_f(x)] $ which
contradicts that $f$ is concave on $[ l_{f}(x),r_{f}(x)]$. 
We may conclude that it is not possible to construct a function $f\in \widetilde{\mathrm{S}}\mathrm{up}(G,H]$
passing through $p\in[W^G(x),W^G(x)+\varepsilon^{\ast}(x))$ or equivalently, it is not possible to construct
a function $\tilde{f}\in \mathrm{Sup}(G,H]$ passing though $\tilde{p}\in[G(x'),G(x')+\varepsilon^{\ast}(F(x')))$
for arbitrary $x'\in I$. Hence $(G)_{H}^{\ast \ast }\geq \hat{V}$ and a symmetric argument can be used to show 
that assuming $(H)_{\ast \ast}^{G}<\check{V}$ leads to a similar contradiction.
\end{proof}

\medskip

The next theorem is the main result in this section and shows when
(\ref{integ assumption}) and (\ref{boundary assumption}) hold, $\hat{V}=\check{V}$
implies the existence of a Nash equilibrium in the optimal stopping game defined in (\ref{lower value})-(\ref{upper value}).
As such it is a purely analytical version of one direction of Theorem 2.1 in \cite{Peskir2}.
Take
\begin{equation}
\mathcal{D}^{+}:=\left\{ \left. x\in I\,\right\vert \,\partial ^{H}G\left(
F(x)\right) \neq \emptyset \right\} \quad ,\quad  \mathcal{D}^{-}:=\left\{
\left. x\in I\,\right\vert \,\partial^{G}H\left( F(x)\right) \neq \emptyset
\right\}
\label{stopping regions}
\end{equation}
as the candidate stopping regions.

\begin{theorem}
\label{Theorem Nash}
Let
\begin{equation}
\tau^{\ast}=\inf \{ t\geq 0\,\vert\,X_{t}\in \mathcal{D}^{+}\} 
\quad ,\quad \sigma^{\ast}=\inf\{ t\geq
0\,\vert \,X_{t}\in \mathcal{D}^{-}\} \label{stopping times}
\end{equation}
and suppose that assumptions (\ref{integ assumption}) and (\ref{boundary assumption}) hold.
The optimal stopping game (\ref{lower value})-(\ref{upper value})
has a Nash equilibrium and $(\tau ^{\ast },\sigma ^{\ast })$ is a saddle point,
i.e. for any $\tau$, $\sigma$
\begin{equation}
R_{x}\left( \tau ,\sigma ^{\ast }\right) \leq R_{x}\left( \tau ^{\ast
},\sigma ^{\ast }\right) \leq R_{x}\left( \tau ^{\ast },\sigma \right) .
\label{saddle point}
\end{equation}
for all $x\in I$.
\end{theorem}

\begin{proof}
When (\ref{integ assumption}) and (\ref{boundary assumption}) hold, it was
shown in Theorem \ref{Theorem Game has value} that the
optimal stopping game has a value and $V=(G)_{H}^{\ast \ast}=(H)^G_{\ast\ast}$. The 
candidate stopping regions (\ref{stopping regions}) suggest the following candidates for the optimal stopping times
\begin{equation*}
\tau^{\ast}=\inf\{t\geq 0\,\vert\,(G)_H^{\ast\ast}(X_t)=G(X_t)\} \quad , \quad
\sigma^{\ast}=\inf\{t\geq 0\,\vert\,(G)_H^{\ast\ast}(X_t)=H(X_t)\} .
\end{equation*}
It follows from Lemmata \ref{lemma sets} and \ref{lemma coincides with dual} 
that $V=\check{V}=\hat{V}$ is $r$-harmonic on 
$I\setminus (\mathcal{D}^+ \cup \mathcal{D}^-)$ and $(G)_{H}^{\ast \ast}(x)
=R_{x}(\tau^{\ast},\sigma^{\ast})$. Let
\begin{equation}
x_{+}=\inf\{ z\geq x\,\vert\,\partial^{H}G(F(z)) \neq \emptyset \} \quad ,\quad x_{-}=\sup \{ 
y\leq x\,\vert\,\partial^{H}G( F(y)) \neq \emptyset \} \label{tau interval}
\end{equation}
and
\begin{equation}
y_{+}=\inf \{ z\geq x\,\vert\,\partial^{G}H(F(z)) \neq \emptyset\} \quad ,\quad 
y_{-}=\sup \{ y\leq x\,\vert \,\partial^{G}H(F(y)) \neq \emptyset \} . \label{sigma interval}
\end{equation}
so that $\tau^{\ast }=T_{x_{+}}\wedge T_{x_{-}}$ and $\sigma ^{\ast}=T_{y_{+}}\wedge T_{y_{-}}$. Take any $a\leq z_+ \leq x\leq z_- \leq b$, set $\sigma =T_{z+}\wedge T_{z_{-}}$ and consider the set $A_{x}:=\left[x_{-},x_{+}\right] \cap \left[ z_{-},z_{+}\right] $, the following four cases are illustrated in Figure \ref{figure Nash}. 
First suppose that $A_{x}=\left[ x_{-},x_{+}\right] $ then
\begin{equation*}
h(y) =\left( \frac{G}{\psi}\right)(x_{-}) +\frac{\left( \frac{G}{\psi }\right)(x_{+}) -\left(\frac{G}{\psi}
\right)(x_{-})}{F(x_{+}) -F(x_{-})}(y-F(x_{-})) \geq \frac{(G)_{H}^{\ast\ast}(y)}{\psi(y)}
\end{equation*}
for all $y\in A_{x}$, moreover, $h$ is the largest $F$-convex function with $h(x_+)=(G/\psi)(x_+)$ and
$h(x_-)=(G/\psi)(x_-)$ so $R_{x}(\tau^{\ast },\sigma) =h(F(x))\psi(x) \geq R_{x}(\tau^{\ast},\sigma^{\ast})$. 
Secondly suppose that $A_{x}=\left[ z_{-},z_{+}\right] $ and let
\begin{equation*}
g(y) =\left( \frac{H}{\psi}\right)(z_{-})+\frac{\left(\frac{H}{\psi}\right)(z_{+})-\left(\frac{H}{\psi}
\right)(z_{-})}{F(z_{+})-F(z_{-})}(y-F(z_{-}))
\end{equation*}
for $y\in A_{x}$. Define $\widetilde{H}(y):=H(y)\mathbb{I}_{(x_{-},x_{+})} +G(y)
\mathbb{I}_{[y=x_{-}] \cup [y=x_{+}]}$ for $y\in [x_{+},x_{-}] $, then the convex biconjugate of
$\widetilde{W}_H(y):=(\widetilde{H}/\psi)\circ F^{-1}(y)$ satisfies $g(y) \geq \widetilde{W}_H^{\ast\ast}(y)$
for all $y\in A_{x}$. Moreover, Lemma \ref{lemma coincides with dual} implies that 
$\widetilde{W}_H^{\ast\ast}(y)=(W^G)_H^{\ast\ast}(y)$ so $R_{x}(\tau^{\ast},\sigma)=g(F(x))\psi(x)
\geq R_{x}(\tau^{\ast},\sigma^{\ast})$. Thirdly, let $A_{x}=[x_{-},z_{+}]$ then
\begin{equation*}
f(y) =\left(\frac{G}{\psi}\right)(x_{-})+\frac{\left(\frac{H}{\psi}\right)(z_{+})-\left(\frac{G}{\psi}
\right)(x_{-})}{F(z_{+}) -F(x_{-})}(y-F(x_{-}))
\end{equation*}
and as in the previous case $f(y)\geq W_H^{\ast\ast}(y)$ and $R_{x}(\tau^{\ast},\sigma)=f(F(x))\psi(
x) \geq (G)_{H}^{\ast\ast}(x) =R_{x}(\tau^{\ast},\sigma^{\ast})$. The final case that $A_{x}=[z_{-},x_{+}]$ 
follows by a symmetric argument. Thus we have shown that the second inequality in (\ref{saddle point}) holds.
The first inequality in (\ref{saddle point}) can be shown using a similar argument and 
thus the stopping times (\ref{stopping times}) are a saddle point.
\end{proof}

\begin{figure}[h!]
\begin{center}
\begin{tabular}{cc}
\includegraphics[width=7cm]{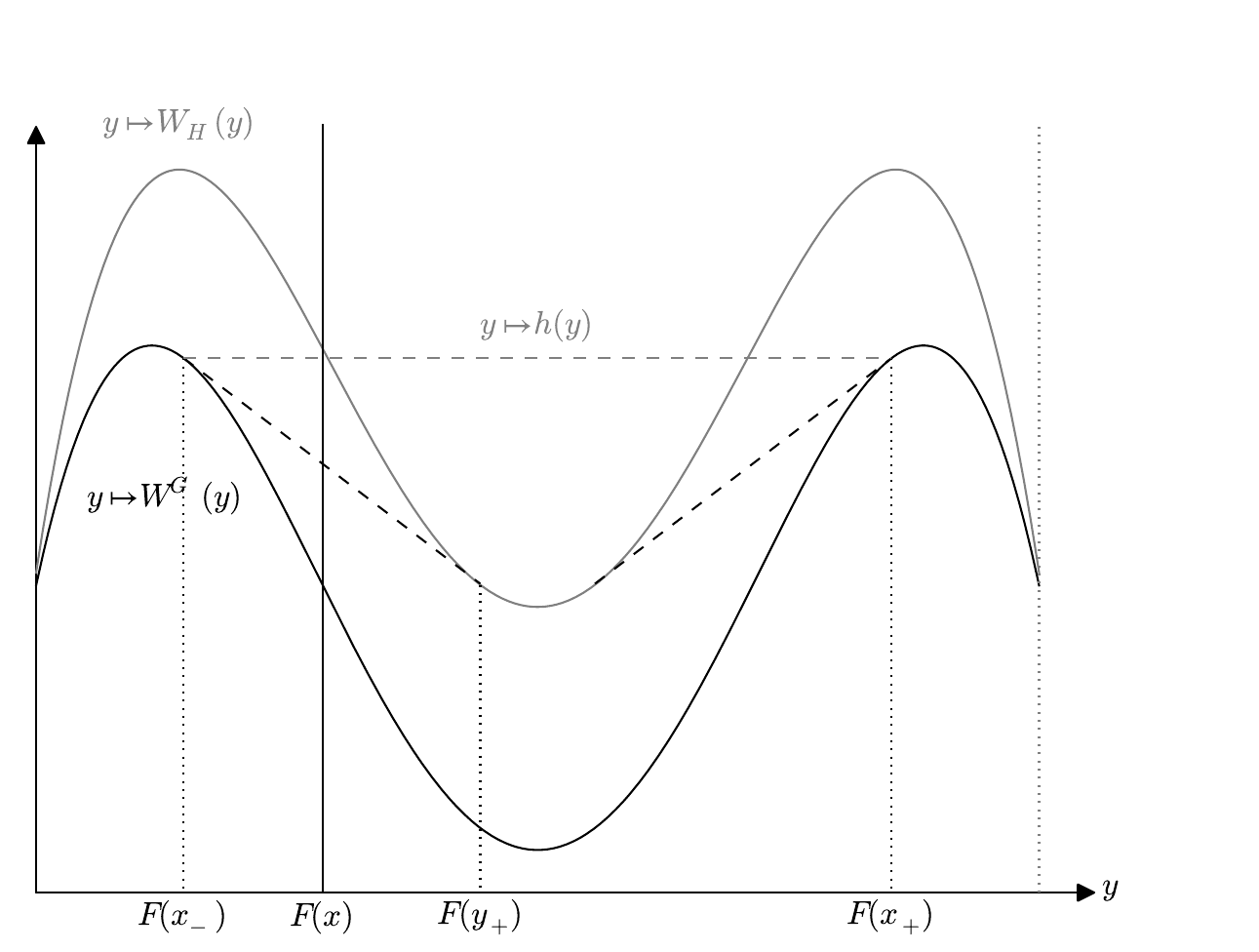} & \includegraphics[width=7cm]{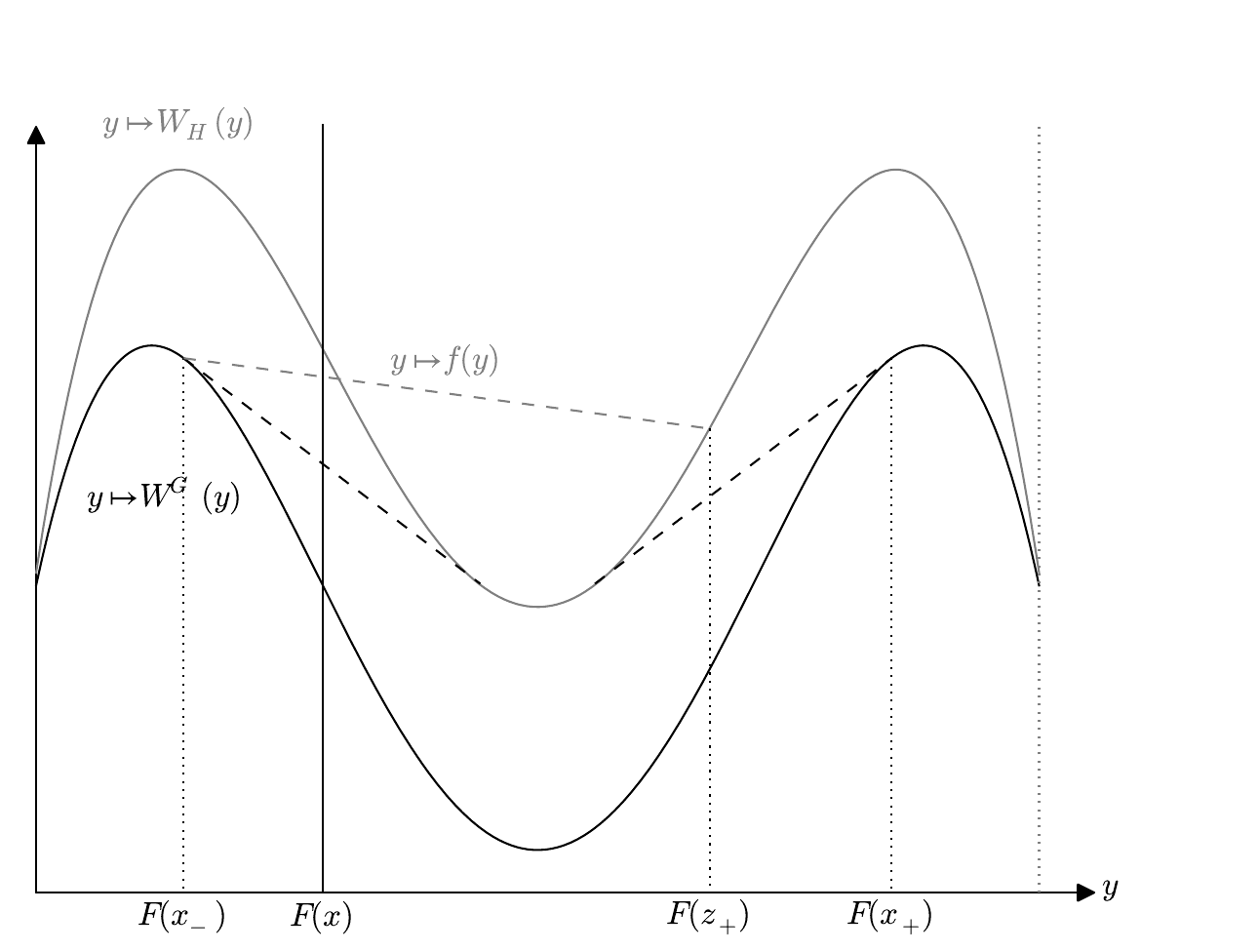} \\
\includegraphics[width=7cm]{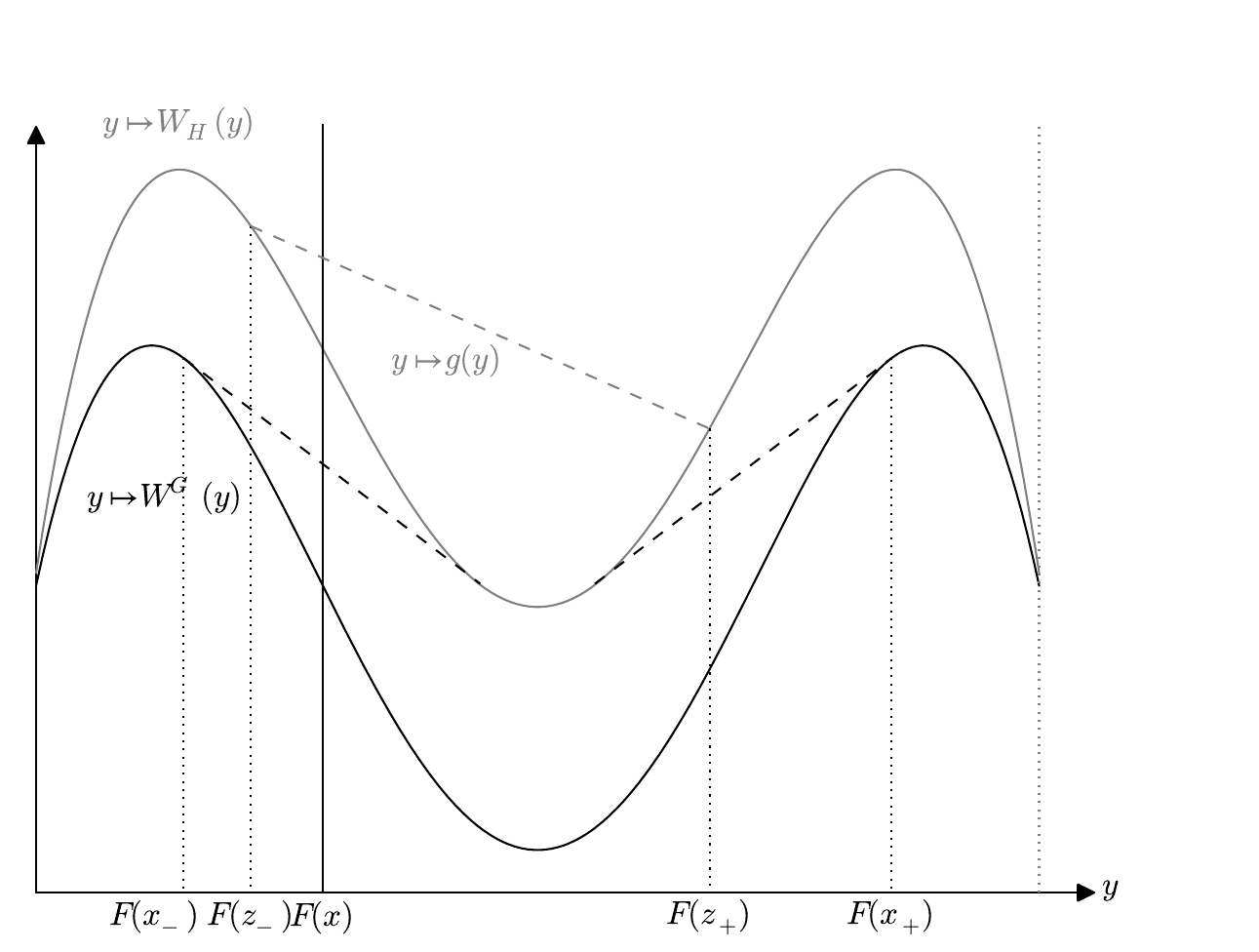} & \includegraphics[width=7cm]{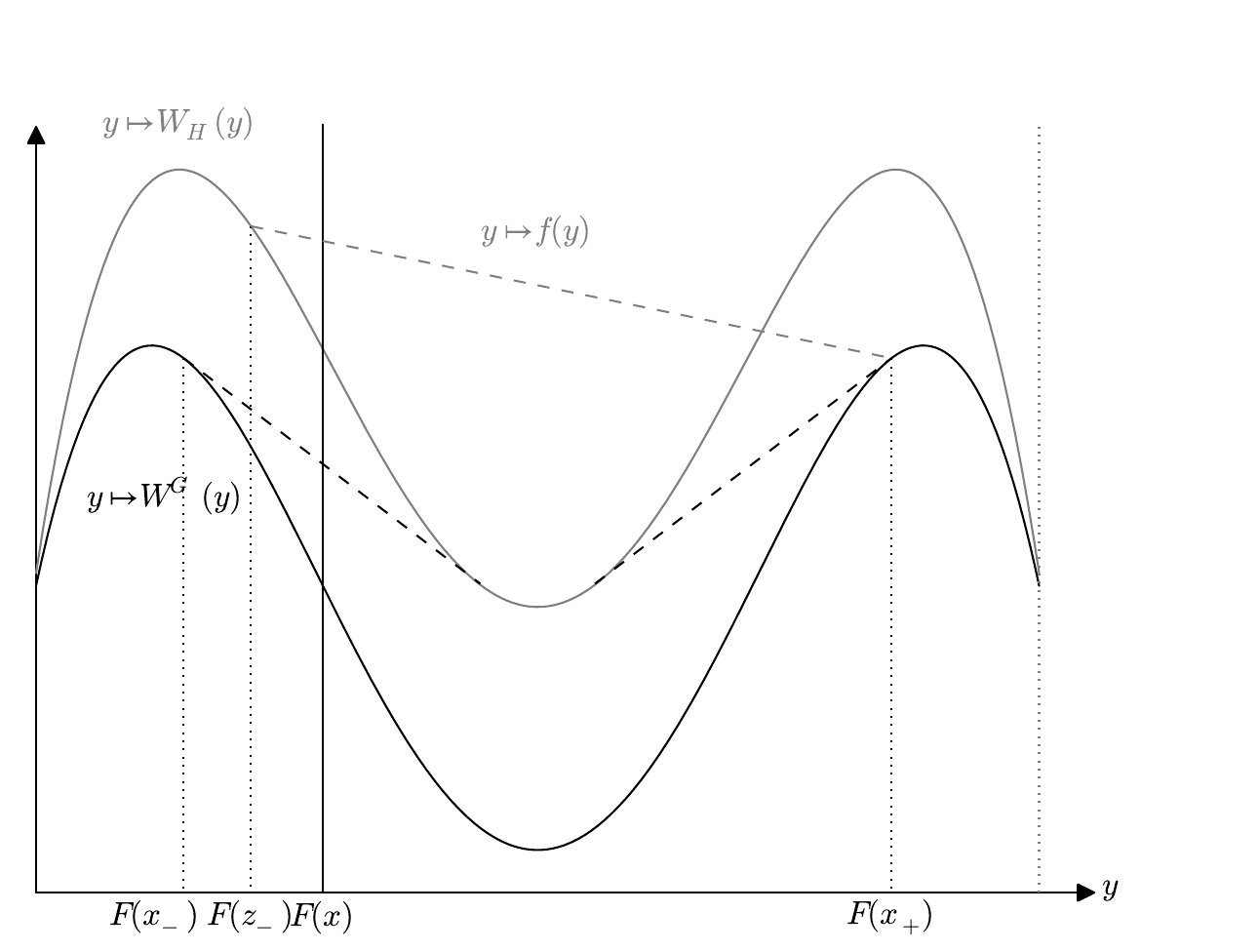}
\end{tabular} 

\begin{minipage}{0.7\textwidth}
\caption[Figure illustrating the three cases used in Theorem \ref{Theorem Nash}]
{For particular choices of $[z_-,z_+]$, these four figures illustrate that the line $f$, $g$ or $h$ dominates the black dotted line $y\mapsto (W^G)_H^{\ast\ast}(y)$.
\label{figure Nash}} 
\end{minipage}
\end{center}
\end{figure}

\medskip

The previous Theorem shows that when assumptions (\ref{integ assumption}) and (\ref{boundary assumption}) hold, the value function $V$ derived in Theorem \ref{Theorem Game has value} can be used to characterise a Nash equilibrium of the optimal stopping game. On the other hand, if we assume a Nash equilbrium exists, Theorem 2.1 in \cite{Peskir2} shows that the semiharmonic characterisation holds. Consequently, it follows from Lemma \ref{lemma coincides with dual} that the equality between the two representations of the generalised Legendre transform in Theorem \ref{Theorem duality} holds (which is the argument presented in \cite{Peskir}). However, in Section \ref{section legendre transform} it has been shown that the two representations of the extension of the Legendre transform in Theorem \ref{Theorem duality} coincide without the need to appeal to the result in \cite{Peskir2}. 

The next corollary relaxes the assumptions (\ref{integ assumption}) and (\ref{boundary assumption})
and is a purely analytical version of \cite{EV} Theorems 4.4 and 4.5. 

\begin{corollary}
\label{corollary Nash}
Define $U(X_{\tau}(\omega))=0$ on the event $\{\tau =+\infty\}$ for any Borel measurable
function $U$. Take $G\leq H$ such that (\ref{integ assumption}) holds, and consider the optimal stopping game (\ref{lower value})-(\ref{upper value}). Let
\begin{equation}
l_{a}:=\lim_{x\downarrow a}\frac{G(x)\vee 0}{\psi(x)} \qquad , \qquad
l_{b}:=\lim_{x\uparrow b}\frac{G(x)\vee 0}{\varphi(x)} .
\end{equation}
and assume that $l_a \in [W^G(0+),W_H(0+)]$ and $l_b=0$.
\begin{enumerate}

\item Extend $W^G$ and $W_H$ onto $\mathbb{R}_+$ by setting $W^G(0)=[W^G(0+),l_a]$ and $W_H(0)=[l_a,W_H(0+)]$. The value of the optimal stopping game (\ref{lower value})-(\ref{upper value}) is
\begin{equation*}
V(x)=\left( G\right) _{H}^{\ast \ast }(x)=(W^G)^{\ast\ast}_H(F(x))\psi(x)
\end{equation*}
for all $x\in I$ where $(W^G)^{\ast\ast}_H$ is as defined in Corollary \ref{corollary extend} with $w_0 =l_a$.

\item When $l_{a}=0$ the optimal stopping game (\ref{lower value})-(\ref{upper value})
has a Nash equilibrium and $(\tau^{\ast},\sigma^{\ast})$ is a saddle point.

\item When $l_{a}>0$ and there exists $l>a$ such that $(G)_{H}^{\ast\ast}(x)>G(x)$ for all $x\in(a,l)$
then the optimal stopping game (\ref{lower value})-(\ref{upper value}) does not have a Nash equilibrium.

\item When $l_{a}>0$ and there exists $l>a$ such that $(G)_{H}^{\ast \ast}(x)=G(x)$ and/or
$(G)_{H}^{\ast \ast}(x)=H(x)$ for all $x \in (a,l)$ then the optimal
stopping game (\ref{lower value})-(\ref{upper value}) has a Nash equilibrium
and $(\tau^{\ast},\sigma^{\ast})$ is a saddle point.

\end{enumerate}
\end{corollary}

\begin{proof}
(1) In the case that $l_a = W^G(0+)$ the first part was shown in Theorem \ref{Theorem Game has value}. When $l_a > W^G(0+)$, $l_a = 0$ and $G(a+)=-\infty$. In this case the functions $W^{G}$ and $W_H$ are extended onto $\mathbb{R}_+$ by setting $W^{G}(0)=[W^G(0+),l_a]$ and $W_H(0)=[[l_a,W_H(0+)]$. This ensures that the function $(W^{G})_{H}^{\ast\ast}$ introduced in Corollary \ref{corollary extend} has $(W^{G})_{H}^{\ast\ast}(0)=(W^{G})_{H}^{\ast\ast}(0+) = 0$. Furthermore, it can be shown using the same approach as in the proof of Theorem \ref{Theorem Game has value} that the value of this optimal stopping game is $V(x)=(W^G)^{\ast\ast}_H(F(x))\psi(x)$ for all $x\in I$.

(2) The boundaries of the diffusion $X$ are natural so $P_x(T_a < \infty)=0$ for all $x\in I$. Thus, to study whether or not Nash equilibrium exists we extend $V$ onto $[a,b)$ by setting $V(a)=V(X_{T_a})=0$.
Let $W^V(x):=(W^{G})_{H}^{\ast\ast}(x)$ for $x\in (0,\infty)$ and $W^V(0):=0$. When $l_a=0$, $W^{V}$ and is continuous on $\mathbb{R}_+$, so the first case follows in exactly the same way as in Theorem \ref{Theorem Nash}.

(3) In the case that $l_a>0$ the function $W^{V}$ is lower-semicontinuous on $\mathbb{R}_+$. In this
case there exists a maximising sequence of stopping times but their limiting value need not be attained. 
To see this, let $x_{\pm}$, $y_{\pm}$ are defined as in
(\ref{tau interval})-(\ref{sigma interval}) and suppose that $x_{-} = y_{-}=a$. Recall that $R_x(\tau,\sigma)$ is defined in (\ref{game objective function}). Take $\tau_{n}=T_{a+1/n} 
\wedge T_{x_{+}\wedge y_+}$ then by assumption 
$\lim_{n\rightarrow \infty}\tau_n = \tau^{\ast}$ and the approach used in the proof of Theorem 
\ref{Theorem Nash} shows that $R_x(\tau_n,\sigma^{\ast})\leq (G)_{H}^{\ast\ast}(x)$ for all 
$n\geq 1$ and $x\in (a,x_+\wedge y_+)$ and by construction 
$\lim_{n \rightarrow \infty}R_x(\tau_n,\sigma^{\ast})=(G)_{H}^{\ast\ast}(x)$. However, 
using the convention at the boundary, 
\begin{align*}
R_x(\tau^{\ast},\sigma^{\ast}) &= E_x\left[e^{-rT_{x_+ \wedge y_+}}(G(X_{T_{x_+}})\mathbb{I}_{[x_+\leq y_+]} + H(X_{T_{y_+}})\mathbb{I}_{[y_+<x_+]})
\mathbb{I}_{[T_{x_+ \wedge y_+}\leq T_a]}\right] \\
&= \left( \frac{G(x_+)}{\psi(x_+)}\frac{F(x)-F(a)}{F(x_+)-F(a)}\mathbb{I}_{[x_+\leq y_+]} 
+ \frac{H(y_+)}{\psi(y_+)}\frac{F(x)-F(a)}{F(y_+)-F(a)}\mathbb{I}_{[y_+<x_+]}\right)\psi(x) \\
&= G(x_+)\frac{\varphi(x)}{\varphi(x_+)}\mathbb{I}_{[x_+\leq y_+]} + H(y_+)
\frac{\varphi(x)}{\varphi(y_+)}\mathbb{I}_{[y_+< x_+]}.
\end{align*}
for all $x\in (a,x_+ \wedge y_+)$. Moreover,
\begin{equation*}
R_x(\tau_n,\sigma^{\ast}) = G(a+1/n)\frac{\psi(x)}{\psi(a+1/n)}\frac{F(x_+ \wedge y_+)-F(x)}{F(x_+ \wedge y_+)-F(a+1/n)}
+ R_x(\tau^{\ast},\sigma^{\ast}).
\end{equation*}
Since $F(a+)=0$ it follows that
\begin{equation*}
\lim_{n\rightarrow \infty}R_x(\tau_n,\sigma^{\ast}) = l_a\left(1-\frac{F(x)}{F(x_+ \wedge y_+)}\right)\psi(x) 
+ R_x(\tau^{\ast},\sigma^{\ast}) > R_x(\tau^{\ast},\sigma^{\ast})
\end{equation*}
and the first term is strictly positive as the function $F$ is strictly increasing. Thus a saddle point does not exist as the limit of the maximising sequence is not attained.

(4) The situation described above can not occur under the assumptions of part (3) of the corollary because
in this case it is not possible to take $x\in I$ such that $x_- \vee y_- = a$ where $x_{\pm}$ and
$y_{\pm}$ are as defined in (\ref{tau interval})-(\ref{sigma interval}).
\end{proof}

\medskip

A similar approach can be taken to study the case that $l_b>0$ but requires Theorem \ref{Theorem Game has value} and Theorem \ref{Theorem Nash} to be formulated in terms of the other ratio of the fundamental solutions using Corollary \ref{corollary general case}.

\begin{remark}
The previous result is heavily dependent on the assumption that
$U(X_{\tau}(\omega))=0$ on the event $\{ \tau =+\infty\}$. This assumption extends any continuous function $U:I\rightarrow \mathbb{R}$ onto $[a,b]$ by setting $U(a)=U(b)=0$. Consequently, the associated function $W^U(y):=(U/\psi)\circ F^{-1}(y)$ is extended onto $\mathbb{R}_+$ y setting $W^U(0)=0$. Hence, the function $U$ (resp. $W^U$) need not be upper-semicontinuous on $[a,b)$ (resp. $\mathbb{R}_+$) which allows for the possibility that the value associated with the optimising sequence of stopping times need not be attained. If instead we extend $G$ and $H$ onto the closure of $I$ using the approach discussed in Remark \ref{Remark duality} the optimal stopping game (\ref{lower value})-(\ref{upper value}) has a saddle point in all the cases discussed in Corollary \ref{corollary Nash}.
\end{remark}

We conclude by examining the Israeli-$\delta$ put-option introduced in \cite{Ky}.

\begin{example}[Israeli-$\delta$ Put]
Suppose that $X$ is a geometric Brownian motion. That is $X$ solves the SDE
\begin{equation*}
dX_{t}=rX_{t}\,dt+\sigma X_{t}\,dW_{t}
\end{equation*}
with initial point $X_0=x \in (0,\infty)$ where $r>0$ is the discount rate and $\sigma >0$ is a volatility parameter. In this case
\begin{equation*}
\mathbb{L}_{X}u(x)=rx\frac{du}{dx}(x) +\frac{1}{2}\sigma ^{2}\frac{%
d^{2}u}{dx^{2}}(x) 
\end{equation*}
and the two fundamental solutions to $\mathbb{L}_{X}u=ru$ are $\varphi(x)=x$ and $\psi(x) =x^{-2r/\sigma ^{2}}$. 
Moreover, 
\begin{equation*}
-\widetilde{F}(x) =\left( \frac{\psi}{\varphi }\right)(x) =x^{-( 1+2r/\sigma ^{2})}  
\quad, \quad F(x)= \left( \frac{\varphi}{\psi}\right)(x) = x^{1+2r/\sigma ^{2}} . 
\end{equation*}
and $(-\widetilde{F})^{-1}(y) =y^{-\alpha}$, $F^{-1}(y)=y^{\alpha}$ where $\alpha =1/(1+2r/\sigma ^{2}) \in [ 0,1]$. For some $K>0$ and $\delta>0$, let
\begin{equation*}
G(x) = (K-x)^{+} \quad ,\quad H(x) = (K-x)^+ + \delta . 
\end{equation*}
The maximising agent has bought a perpetual American put option with strike $K>0$ from the minimising agent but 
the minimising player retains the right to cancel the option by paying a fixed penalty of $\delta>0$. 
The gains functions are rescaled by taking 
\begin{align*}
\widetilde{W}^G(y) &=\left( \frac{G}{\varphi}\right) \circ
(-\widetilde{F})^{-1}(y) =( K y^{\alpha}-1)^{+}, \\
\widetilde{W}_H(y) &=\left( \frac{H}{\varphi }\right) \circ
(-\widetilde{F})^{-1}(y) = (K y^{\alpha} -1)^+ +\delta y^{\alpha }.
\end{align*}
for $y\in (0,\infty)$. The functions $G$, $H$ are illustrated in the left panel of Figure \ref{figure delta-put} and the functions $\widetilde{W}^G$ and $\widetilde{W}_H$ are illustrated in the right panel of Figure \ref{figure delta-put}. The functions $\widetilde{W}^G$ and $\widetilde{W}_H$ are concave on the interval $[K^{-1/\alpha},+\infty)$ and $\widetilde{W}_H$ is concave on the interval $[0,K^{-1/\alpha}]$.  

The problem is degenerate in the sense that the inf-player will choose $\sigma^{\ast}
=+\infty$ and the game has the same value as the perpetual put option examined in Example \ref{Example Legendre put}
if the concave biconjugate of $\widetilde{W}^G$ minorises $\widetilde{W}_H$, i.e.
$\widetilde{W}^G_{\ast \ast }(y) \leq \widetilde{W}_H(y)$ for all $y\in \mathbb{R}_{+}$
which is equivalent to 
\begin{equation*}
\delta \geq \widetilde{W}^G_{\ast\ast}(K^{-1/\alpha})\varphi(K) = \frac{\sigma ^{2}}{2r}\frac{K}{\left( 1+\frac{\sigma ^{2}}{%
2r}\right) ^{1+2r/\sigma ^{2}}} =: \delta ^{\ast }.
\end{equation*}
To avoid this case assume that $\delta \in (0,\delta^{\ast}]$.
On $[0,K^{-1/\alpha}]$ the function $\widetilde{W}_H$ is concave so 
\begin{equation*}
(\widetilde{W}^G)_H^{\ast\ast}(y) \leq \left(\frac{\widetilde{W}_H( 
K^{-1/\alpha})}{K^{-1/\alpha }}\right)y=\frac{\delta}{K}K^{1/\alpha}y 
\quad \forall y\in [0,K^{-1/\alpha}]. 
\end{equation*}
In fact this inequality holds with equality as for all $c<\widetilde{W}_H (K^{-1/\alpha})/K^{-1/\alpha}$
\begin{equation*}
cy \vee \widetilde{W}^G(y)<\widetilde{W}_H(y) \quad \forall \,y\in(0,\infty).
\end{equation*} 
On $[K^{-1/\alpha},\infty)$, we can use an approach similar to that used in Example \ref{Example Legendre put}. 
There exists $y^{\ast}$ such that
\begin{equation*}
\widetilde{W}_H(K^{-1/\alpha}) =\widetilde{W}^G(y^{\ast})-\frac{d}{dy}\widetilde{W}^G(y^{\ast})
(y-K^{-1/\alpha}) . 
\end{equation*}
This $y^{\ast}$ solves the polynomial
\begin{equation*}
\frac{\delta}{K}+1=Ky^{\alpha}(1-\alpha ) +\alpha y^{\alpha -1}K^{( \alpha
-1) /\alpha} .
\end{equation*}
Let $x^{\ast}$ be defined as $y^{\ast}=(x^{\ast})^{-1/\alpha}$, then the constant $x^{\ast}$ satisfies 
\begin{equation*}
(1+2r/\sigma^2)\left( 1+\frac{\delta }{K}\right)\left(\frac{x}{K}\right) 
=2r/\sigma^{2}+\left( \frac{x}{K}\right)^{(1+2r/\sigma ^{2})} 
\end{equation*}
which is the same condition as in \cite{Ky}. Moreover, $(\widetilde{W}^G)_{H}^{\ast\ast}(y) =
\widetilde{W}^G(y)$ for $y\geq y^{\ast}$ and on $[K^{-1/\alpha},y^{\ast}]$ we have 
\begin{eqnarray*}
(W^G)_{H}^{\ast\ast}(y) &=& \widetilde{W}_H(K^{-1/\alpha}) +\frac{\widetilde{W}^G(y^{\ast})
-\widetilde{W}_H(K^{-1/\alpha})}{y^{\ast}-K^{-1/\alpha}}(y-K^{-1/\alpha}) \\
&=&\frac{\delta}{K}+\left( K(y^{\ast})^{\alpha}-1-\frac{\delta}{K}\right) 
\frac{y-K^{-1/\alpha }}{y^{\ast}-K^{-1/\alpha}} \\
&=&\frac{\delta}{K} \frac{(x^{\ast})^{-1/\alpha}-y}{(x^{\ast})^{-1/\alpha}-K^{-1/\alpha}}
+\frac{1}{x^{\ast}}(K-x^{\ast}) \frac{y-K^{-1/\alpha}}{(x^{\ast})^{-1/\alpha}-K^{-1/\alpha}}.
\end{eqnarray*}
The function $y \mapsto (\widetilde{W}^G)_H^{\ast\ast}(y)$ is the black line in the lower panel 
of Figure \ref{figure delta-put}. 

\begin{figure}[h!]
\begin{center}

\begin{tabular}{cc}
\includegraphics[width=7cm]{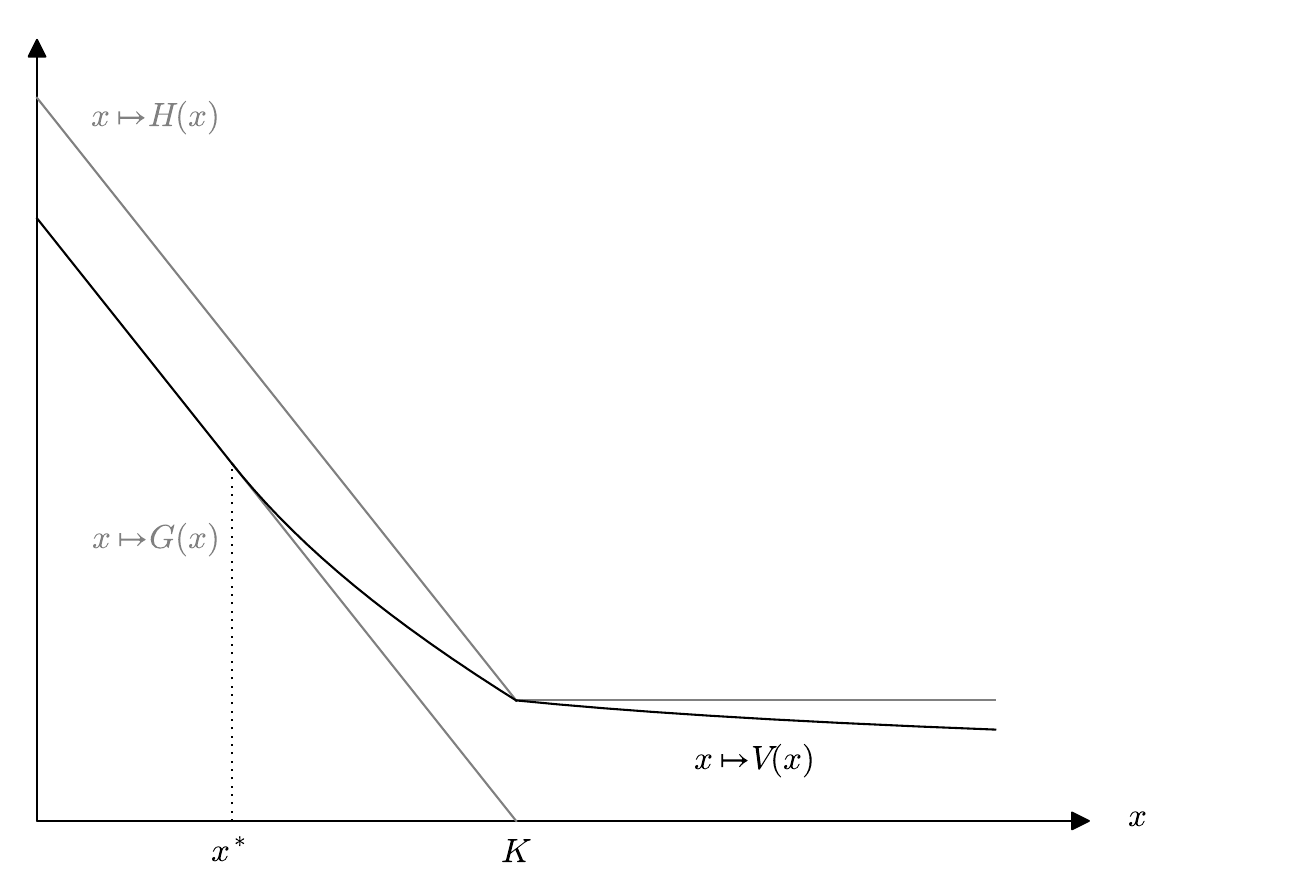} & \includegraphics[width=7cm]{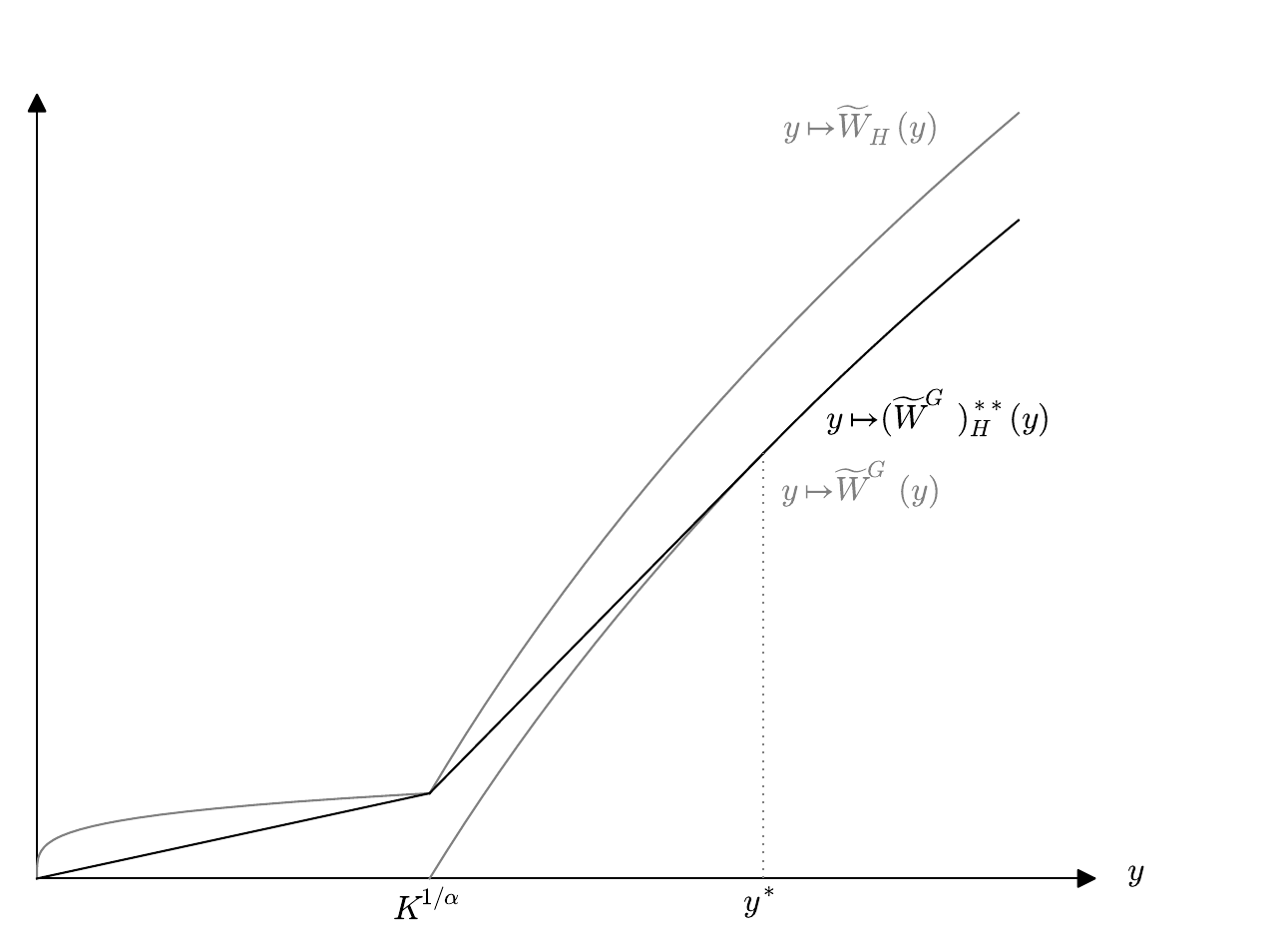}
\end{tabular} 

\begin{minipage}{0.7\textwidth}
\caption[Israeli-$\delta$ put option]
{The top figure illustrates the payoff functions of the game. The black line is the value function of the game. The lower panel illustrates the transformed payoffs $\widetilde{W}_H$ and $\widetilde{W}^G$ the black line is the function $(\widetilde{W}^G)_H^{\ast\ast}$.
\label{figure delta-put}} 
\end{minipage}
\end{center}
\end{figure}

\noindent 
According to Theorem \ref{Theorem Game has value}
the value of this game option satisfies $V(x) =(\widetilde{W}^G)_H^{\ast\ast}
(-\widetilde{F}(x))\varphi(x)$ so
\begin{equation*}
V(x) = \left\{
\begin{array}{cl}
(K-x) & \mathrm{for}\, x< x^{\ast} \\
\delta \left( \frac{x}{K}\right) \frac{(x^{\ast})^{-1/\alpha}-x^{-1/\alpha}}{(x^{\ast})^{-1/\alpha}-K^{-1/\alpha}}
+(K-x^{\ast}) \left( \frac{x}{x^{\ast }}\right) \frac{x^{-1/\alpha}-K^{-1/\alpha}}{
(x^{\ast})^{-1/\alpha}-K^{-1/\alpha}} & \mathrm{for}\, x\in [x^{\ast},K] \\
\delta \left( \frac{x}{K}\right)^{(\alpha-1)/\alpha} &\mathrm{for}\, x>K
\end{array}
\right.
\end{equation*}
which can be rearranged into the form derived in \cite{Ky}.
In this example assumptions (\ref{integ assumption}) and (\ref{boundary assumption}) hold so Theorem 
\ref{Theorem Nash} tells us that a saddle point of the game option is
\begin{equation*}
\tau^{\ast}=\inf \left\{ \left. t\geq 0\,\right\vert \,X_{t}\leq x^{\ast}\right\}
\quad , \quad  \sigma^{\ast}=\inf \left\{ \left. t\geq 0\,\right\vert
\,X_{t}=K\right\} . 
\end{equation*}
\end{example}

\end{document}